\documentclass[a4paper,12pt,english,english]{article}
\usepackage[latin1]{inputenc}
\usepackage[T1]{fontenc}
\usepackage{amsfonts}
\usepackage{amsmath}
\usepackage{amscd}
\usepackage{lipsum}
\usepackage{graphicx}
\usepackage{mathrsfs}
\usepackage{titletoc}
\usepackage{babel}
\usepackage{amssymb}
\usepackage{bbold}
\usepackage[colorlinks=true]{hyperref}
\hypersetup{urlcolor=blue, linkcolor=blue}
\usepackage{fancyhdr}
\usepackage{textcomp}
\usepackage{amsthm}
\usepackage{frcursive}
\usepackage[toc]{appendix}
\usepackage[a4paper,left=2cm,right=2cm,top=2cm,bottom=2cm]{geometry}
\numberwithin{equation}{section}
\newtheorem{thm}{Theorem}[section]
\newtheorem{dfn}[thm]{Definition}
\newtheorem{lem}[thm]{Lemma}
\newtheorem{cor}[thm]{Corollary}
\newtheorem{rmq}[thm]{Remark}
\newtheorem{prp}[thm]{Proposition}

\makeatletter
\makeatletter
\def\toclevel@paragraph{5}
\def\toclevel@subparagraph{6}
\makeatother
\usepackage{setspace}

\begin{document}
	\vspace*{1cm}
	\thispagestyle{empty}
	\begin{center}
		{\Large{A SIR Stochastic Epidemic Model in Continuous Space: \vspace*{0.2cm} \\Law of Large Numbers and  Central Limit Theorem  }}\vspace*{0.4cm}\\
		Alphonse Emakoua \hspace*{0.2cm} \hspace*{0.2cm}  \\
		emakouaal@gmail.com\vspace*{0.3cm}\\
		\textit{Aix-Marseille Université, Centre de mathématiques et informatique (C.M.I),}\\\textit{39 Rue Frédéric Joliot Curie, 13013 Marseille, France}\vspace*{0.2cm}\\
		\today \vspace*{0.3cm}
	\end{center}
	
	\begin{center}
		\textbf{\large{Abstract}}
			\vspace*{0.3cm}
	\end{center}
		The impact of spatial structure on the spread of an epidemic is an important issue in the propagation of infectious diseases. Recent studies, both deterministic and stochastic, have made it possible to understand the importance of the movement of individuals in a population on the persistence or extinction of an endemic disease. In this paper we study a compartmental SIR stochastic epidemic model  for a population that moves on $\mathbb{R}^{d}$ following SDEs driven by independent Brownian motions. We define the sequences of empirical measures, which describe the evolution of the positions of the susceptible, infected and removed individuals. Next, we obtain large population approximations of those sequence of  measures, as weak solution of a system of reaction-diffusion equations. Finaly we study the fluctuations of the stochastic model around its large population  limit via the central limit theorem. The limit is a distribution
		valued Ornstein-Uhlenbeck process and can be represented as the solution of system of stochastic
		partial differential equations.\\\\
	
\hspace*{-0.6cm}\textbf{Keywords}:  Stochastic   model $\cdotp$ Déterministic $\cdotp$ Law of large numbers $\cdotp$ Central limit theorem $\cdotp$ Measure valued processes.\\
	
	\section{Introduction}
	\hspace*{0.5cm} Deterministic models of epidemics have been developed significantly in recent decades. The study of stochastic models in contrast is more recent, see for example   \cite{dc}, \cite{hc}, \cite{sb},\cite{lpz},\cite{vc} and \cite{npy}. Anderson and Britton \cite{dc},
	and Britton and Pardoux \cite{hc}  show that deterministic models of epidemics are the law of large
	numbers limits (as the size of the population tends to $\infty$) of homogenous stochastic models, while the central
	limit theorem and moderate and large deviations (see \cite{hc} and \cite{vc}) give tools to accurately
	describe the gap between stochastic and deterministic models.\\ However, in their models, they ignore the fact that a population spreads over a spatial
	region. But spatial heterogeneity, habitat connectivity, and rates of movement play important roles in disease persistence and extinction. Movement of susceptible or infected individuals can enhance or suppress the spread of disease, depending on the heterogeneity and connectivity of the spatial environnement, see for example \cite{cc} and the references therein, for the deterministic case and \cite{ik} and \cite{npy} for the stochastic case.
	\par Having in mind the above conclusions,  Emakoua and  Pardoux  \cite{sb},  have studied the law of large numbers and central limit theorem of two spatial SIR epidemic models in a compact set. In the first considered model, the population is moving and in the second, there is no mouvement (to refer to plant epidemics). For the first   model,  the limiting law of large numbers model is a system of parabolic PDEs, which is a  deterministic
	epidemic model in continuous space, and for the second a system of ODEs. The study of the fluctuations around the limit law of large numbers through the  central limit theorem gives in the limit for the two models an Ornstein-Uhlenbeck processe with values  in a negative index Sobolev space. Moreover in the fisrt model it was assumed that the diffusion coefficient remains the same for all the individuals and that the infectious rate does not depends of the position of the infectious. The impact of environmental heterogeneity was not   taken into account. In this paper we study the law of large numbers and central limit theorem of a SIR stochastic epidemic models, for a population with constant size $N$, distributed on $\mathbb{R}^{d}.$ The initial condition will be the same as in \cite{sb}.\par Let us describe our model. 
	We consider  a population of  constant size $N$, living in $\mathbb{R}^{d}$.  We assume that at time t=0,  the population consists of two classes: the susceptible  S(0) and the infectious  I(0), such that $S(0)+I(0)=N$ described as follows.\\
	Given $ A$  an arbitrary Borel subset  of  $\mathbb{R}^{d}$  and $ 0 <p \le 1 $, each individual $i $ is placed in $\mathbb{R}^{d}$ independently of the others at the position $X_{0}^{i}$. If  $X_{0}^{i}\in A^{c}$  then the individual $i$   is susceptible and  if $X_{0}^{i}\in A $, the individual $i$ is infected with  probability $p$ and susceptible with  probability $1-p$. This situation is modelled by empirical measures
	\vspace*{0.2cm} \\$
	\hspace*{5cm}\displaystyle  \mu_{0}^{S,N}=\frac{1}{N}\sum_{i=1}^{N}\{1_{A}(X_{0}^{i})(1-\xi_{i})+1_{A^{c}}(X_{0}^{i})\}\delta_{X_{0}^{i}}
	\\\hspace*{5cm}
	\displaystyle  \mu_{0}^{I,N}=\frac{1}{N}\sum_{i=1}^{N}1_{A}(X_{0}^{i})\xi_{i}\delta_{X_{0}^{i}} 
	$\vspace*{0.07cm}\\
	where  $\{\xi_{i},1\leq i\leq N \}$ is a mutually independent family of Ber(p) random variables, globally independent of $\{X_{0}^{i},1\leq i\leq N \}$, which in turn is a mutually independent family of random variables.\vspace*{0.08cm}\\\hspace*{0.5cm}During the epidemic the population is divided  into three compartments: the susceptible S, the infectious I and the  recovered R (the recovered individuals are those who are dead  or who have recovered and have permanent immunity). We weaken  assumptions  made in  the first model in \cite{sb}, by assuming that an individual $i$ moves on  $\mathbb{R}^{d}$  according to a diffusion process driven by  the following stochastic differential equation.
	
	\begin{align}
		X^{i}_{t}=X_0^{i}+\displaystyle\int_{0}^{t}m(E^{i}_{r},X^{i}_{r})dr+\int_{0}^{t} \theta(E^{i}_{r},X^{i}_{t})dB_{r}^{i}. \label{uith}
	\end{align} 
	Where
	\begin{itemize} 
		\item $\{ B^{i},1\leq i\leq N \}$ is  a family of  independent standard   Brownian motions on $\mathbb{R}^{d}$ which is globaly independent of    $\{ X^{i}_{0},1\leq i\leq N \}.$ \vspace*{-0.12cm}
		\item  $E_{t}^{i}$ is the state at time $t$ of the individual $i$, $E_{t}^{i}\in \hspace*{-0.2cm}\mbox{\begin{cursive}	C	\end{cursive}}=\{S,I,R\}.$ \vspace*{-0.1cm}
		\item The environment
		heterogeneity is modelled by the  function $m$: $\hspace*{-0.2cm}\mbox{\begin{cursive}	C	\end{cursive}}\hspace*{-0.2cm}\times\mathbb{R}^{d}\longrightarrow \mathbb{R}^{d}$. \vspace*{-0.12cm}
		\item The diffusion matrix is the  function $\theta$: $\hspace*{-0.2cm}\mbox{\begin{cursive}	C	\end{cursive}}\hspace*{-0.2cm}\times\mathbb{R}^{d}\longrightarrow \mathcal{M}_{d}(\mathbb{R})$.\vspace*{-0.17cm}
		\item We assume  that the functions $m$ and $\theta$ are bounded and    Lipschitz continuous with respect to the space variable.
	\end{itemize} 
	Infections are non local and  a susceptible $i$ becomes infected at time t at the following  rate. 
	\begin{align}
\displaystyle\frac{1}{N^{\gamma}}  \sum_{j=1}^{N}\frac{K(X_{t}^{i},X_{t}^{j})}{\Big[\sum\limits_{\ell=1}^{N}K(X_{t}^{\ell},X_{t}^{j})\Big]^{1-\gamma}}
	1_{\{E_{t}^{j}=I\}}, \quad\textrm{with} \quad \gamma\in[0,1] \label{ftxg}
		\end{align}
	 where to have less parameter, we have let $K(X_{t}^{i},X_{t}^{j})=\beta(X_{t}^{j})F(X_{t}^{i},X_{t}^{j}),$ with $\beta$ a function on $\mathbb{R}^{d}$ and  $F$ a symetric function on $ \mathbb{R}^{d}\times \mathbb{R}^{d} $ . The formulation of such a rate of infections can be explained as follows. Since we take into account the spatial structure, an
	infectious individual $j$ has a contact with the individual   $i$ at the rate  $\displaystyle \frac{1}{N^{\gamma}}\frac{K(X_{t}^{i},X_{t}^{j})}{\Big[\sum\limits_{\ell=1}^{N}K(X_{t}^{\ell},X_{t}^{j})\Big]^{1-\gamma}},$ thus summing over the infectious individuals at time $t$ gives the above rate.\\The case $\gamma=0,$ has already been studied in \cite{sb}, in this paper we focus  on the case $\gamma=1,$ so (\ref{ftxg}) will becomes $ \displaystyle\frac{1}{N}  \sum_{j=1}^{N}K(X_{t}^{i},X_{t}^{j})1_{\{E_{t}^{j}=I\}}.$
	The case $\gamma\in (0,1)$ will be studied in a  future work.\\\\
	The evolution of the numbers of susceptible, infectious and removered individuals is described by the following equations.
	\begin{align*}
		\displaystyle  S(t)&=S(0)-P_{inf}\left(  \int_{0}^{t}  \frac{1}{N}  \sum_{i\neq j}K(X_{r}^{i},X_{r}^{j})1_{\{E_{r}^{i}=S\}} 1_{\{E_{r}^{j}=I\}} dr\right)
		\\
		\displaystyle  I(t)&=I(0)+P_{inf}\left(  \int_{0}^{t} \frac{1}{N}  \sum_{i\neq j}K(X_{r}^{i},X_{r}^{j})1_{\{E_{r}^{i}=S\}} 1_{\{E_{r}^{j}=I\}} dr\right)-P_{cu}\left(\alpha \int_{0}^{t}\sum_{j=1}^{N}1_{\{E_{r}^{j}=I\}} dr\right)
		\\
		\displaystyle R(t)&=R(0) + P_{cu}\left(\alpha \int_{0}^{t}\sum_{j=1}^{N}1_{\{E_{r}^{j}=I\}} dr\right),
	\end{align*}
	
	where $P_{inf}$ and $P_{cu}$ are two independent standard Poisson processes.\vspace*{0.2cm}\\
	We end  the description of the  model, by defining the renormalized point processes, which allows us to control the evolutions of the positions of susceptible, infectious and recovered individuals and the proportions of susceptible, infectious and recovered individuals in any subset of $\mathbb{R}^{d}$.
	$ \forall t>0, $ \vspace*{0.2cm}\\\hspace*{5cm}
	$ \displaystyle  \mu_{t}^{S,N}=\frac{1}{N}\sum_{i=1}^{N}1_{\{E_{t}^{i}=S\}}\delta_{X_{t}^{i}},
	\\\hspace*{5cm}
	\displaystyle  \mu_{t}^{I,N}=\frac{1}{N}\sum_{i=1}^{N}1_{\{E_{t}^{i}=I\}}\delta_{X_{t}^{i}},
	\\\hspace*{5cm}
	\displaystyle  \mu_{t}^{R,N}=\frac{1}{N}\sum_{i=1}^{N}1_{\{E_{t}^{i}=R\}}\delta_{X_{t}^{i}}.$
	\vspace*{0.25cm}\\ \hspace*{0.6cm}Since the law of large numbers and the central limit theorem of the initial sequence \\$(\mu_{0}^{I,N}, \mu_{0}^{S,N})_{N\geq 1}$ has already been studied in \cite{sb}, under the assumption (H0) that the law of $X^{1}_{0},$ is absolutly continuous with respect to the Lebesgue measure, in this paper, we will first write the equation of evolution of $(\mu_{t}^{S,N},\mu_{t}^{I,N},\mu_{t}^{R,N}),$ when the size of the population  $N$ is fixed. We shall next study the law of large numbers and the central limit theorem of those sequences. The law of large number
	result will be a convergence result in the space of measure valued processes. The convergence
	proof will start with tightness in the appropriate space, identification of the limit of any vaguely
	converging subsequence with the unique deterministic solution of a system of PDEs, from which
	the weak convergence, then  in probability of the whole sequence will follow.\\The central limit theorem is
	technically more involved. The first difficulty comes from the fact that our domain is not compact.  The approximating sequence lives in the space of signed measures
	valued processes and 
	one of the main problems to overcome is to find a suitable space in which this sequence, as well as its limit, take their values.\\
	We prove that the approximating sequence converges in the Skorokhod space $[\mathbb{D}(\mathbb{R}_{+},H^{-s,\sigma}(\mathbb{R}^{d}))]^{3},\\$($\sigma>d/2,$ $1+\lceil\frac{d}{2}\rceil<s<2+\lceil\frac{d}{2}\rceil$), to a continuous process characterized as the unique solution of a linear
	Gaussian processes valued stochastic partial differential equation (abbreviated below SPDE). The weighted Sobolev spaces $H^{-s,\sigma}(\mathbb{R}^{d})$ we consider here were introduced by Métivier \cite{mt}, for the integer values of $s.$ Mélérad \cite{meroe} uses that space for the  study of the central limit theorem of a sequence of  empirical (random) measures of interacting particle systems. Cl\'emençon and all \cite{cta} also use that space to study a central limit theorem for  a specific stochastic epidemic model accounting for the effect of contact-tracing on the spread of an infectious disease. The work of Löfstrom  \cite{lof} and \cite{loff} allows us to extend that space to the non interger values of $s,$ by using real interpolation techniques.\vspace*{0.17cm}\par The paper is  organized  as follows.
	In section 2 we recall some results that will be useful in the sequel. In section 3, we  first establish the evolution equations of the measure-valued processes   $\mu^{S,N},$ $\mu^{I,N} $ and $\mu^{I,N} $  then we  show  that   the sequence  $\{(\mu_{t}^{S,N},\mu_{t}^{I,N},\mu_{t}^{R,N}) ,t\geq0 \} $ converges in probability as $N\rightarrow \infty $  towards $(\mu^{S},\mu^{I},\mu^{R}),$ the unique solution of a system of parabolic PDEs. In section 5 we study the convergence of the sequence of fluctuations processes $(U^{N}=\sqrt{N}(\mu^{S,N}-\mu), V^{N}=\sqrt{N}(\mu^{I,N}-\mu^{I}),W^{N}=\sqrt{N}(\mu^{R,N}-\mu^{R}))$ as $N\rightarrow \infty$. 
	\section{Preliminaries}
	\textbf{Notation:}  For any metric space $\mathbb{E},$
	\begin{itemize}
	 \item $\mathcal{M}_{F}(\mathbb{E})$ denotes the space of finite measures on $\mathbb{E}$. 
	\item For any  integer $k\ge0$, $C^{k}(\mathbb{E})$ (resp. $C^k_{c}(\mathbb{E})$) denotes the space of continuous and $k$ times continuously  differentiable real valued functions defined on $\mathbb{E},$ (resp. with compact support). For $k=0,$ we write $C(\mathbb{E})$ (resp .$C_{c}(\mathbb{E})$) instead of $C^{0}(\mathbb{E}).$ (resp. $C^{0}_{c}(\mathbb{E})$).  
		\item For any  integer $k\ge0$, $C^{k}_{b}(\mathbb{E})$  denotes the space  of real 	valued functions of class 
		$C^{k}$ on $\mathbb{E}$ with bounded derivatives up to order $k$ (order 0 included ). 
	\item $C_{0}(\mathbb{E})$ denotes the 
	space of continuous functions on $\mathbb{E}$ vanishing at infinity.
	\item For $\mu \in \mathcal{M}_{F}(\mathbb{E}) $ and $ \varphi \in C(\mathbb{E})$,  we denote the integral  $\int_{\mathbb{E}} \varphi(x)\mu(dx)$ by $(\mu,\varphi)$.  
	\item In  the following, the letter $C$ will denote a (constant) positive  real number which can change from line to line. 
	\item We equip  $\mathcal{M}_{F}(\mathbb{E})$   with the topology of weak convergence.
	\item Let $\mathbb{E}$ be a complete separable metric space, $\mathbb{C}(\mathbb{R}_{+},\mathbb{E})$ (resp. $\mathbb{D}(\mathbb{R}_{+},\mathbb{E}$) is a space of continuous (resp.  càdlàg) functions   from $\mathbb{R}_{+}$  to $\mathbb{E},$  equipped with the locally uniform (resp.  Skorokhod) topology. We refer the reader to section 12 of \cite{fc} for a presentation of the Skorokhod topology and its associated metric. 
		\end{itemize}
		\subsection{Weighted Spaces of functions \label{sec3}}
	For every nonnegative integer $m$ and $\sigma\in\mathbb{R}_{+}$, we consider  the space of all real valued functions $\varphi$ defined on $\mathbb{R}^{d},$ with partial derivative up to oder $m$ such that:
	\begin{center}
		$\lVert \varphi \lVert_{m,\sigma}=\Big( \sum\limits_{\lvert \gamma\lvert\leq m}\displaystyle\int_{\mathbb{R}^{d}}\frac{\lvert D^{\gamma}\varphi(x)\lvert^{2}}{1+\lvert x\lvert^{2\sigma}}dx\Big)<+\infty,$
	\end{center}
	where $\lvert . \lvert $ denotes the euclidian norm on $\mathbb{R}^{d}$, and  for $\gamma=(\gamma_{1},\gamma_{2}.....\gamma_{d})\in\mathbb{N}^{d}$, $\lvert \gamma \lvert =\sum\limits_{i=1}^{d}\gamma_{i}$ and $D^{\gamma}\varphi=(\partial^{\lvert \gamma\lvert}\varphi)/(\partial x_{1}^{\gamma_{1}}\partial x_{2}^{\gamma_{2}}.....\partial x_{d}^{\gamma_{d}}).\\\\$
	Let $W^{m,\sigma}_{0}(\mathbb{R}^{d})$ be the closure of the set of functions of class $C^{\infty}$ with compact support for this norm. $W^{m,\sigma}_{0}(\mathbb{R}^{d})$ is a Hilbert space for the norm $\lVert . \lVert_{m,\sigma}$.  For a nonnegative real number $s$ we extend the above space as follows.\\Let $\mathcal{J}^{s}$ be the potential operator defined on $\mathbb{R}^{d}$ by $(\mathcal{J}^{s}\varphi)=\mathcal{F}^{-1}[(1+\lvert.\lvert^{2})^{s/2}\widehat{\varphi}]$ (where  the Fourier transform $\widehat{\varphi}$  of $\varphi$ is well defined, $\mathcal{F}^{-1}$ denotes the inverse of the Fourier transform). \\$H^{s,\sigma}(\mathbb{R}^{d})$ denotes the space of functions  $\varphi$ which satisfy the following.
	\begin{center}
		$\lVert \varphi \lVert_{s,\sigma}=\lVert (\mathcal{J}^{s}\varphi)\lVert_{0,\sigma}<\infty.$ 
	\end{center}
	It is shown in \cite{lof} that $H^{m,\sigma}(\mathbb{R}^{d})=W^{m,\sigma}_{0}(\mathbb{R}^{d})$, for any nonnegative integer $m.$\\
	We denote by $H^{-s,\sigma}(\mathbb{R}^{d})$ the  dual space of $H^{s,\sigma}(\mathbb{R}^{d})$.\\\\Let $C^{m,\sigma}(\mathbb{R}^{d})$ be the space of functions $\varphi$ with continuous  partial derivatives up to oder $m$ and such that $\lim\limits_{\lvert x\lvert\longrightarrow\infty}\lvert D^{\gamma}\varphi(x)\lvert^{2}/1+\lvert x\lvert^{2\sigma}=0$ for all $\lvert\gamma\lvert\leq m.$ \\
	This space is normed with 
	\begin{center}
		$\lVert \varphi \lVert_{C^{m,\sigma}}=\displaystyle \sum\limits_{\lvert \gamma\lvert\leq m}\underset{x\in \mathbb{R}^{d}}{\sup}\frac{\lvert D^{\gamma}\varphi(x)\lvert}{1+\lvert x\lvert^{\sigma}}.$
	\end{center}
	Let $C^{-m,\sigma}(\mathbb{R}^{d})$ denotes the dual of $C^{m,\sigma}(\mathbb{R}^{d}).$\\ We have the following continuous  embeddings (see  \cite{ac} and \cite{mt}).\\
	\begin{align}
		C^{m+j,\sigma}&\hookrightarrow C^{m,\sigma+r}\qquad m\geq 0, \quad j\geq 0, \quad \sigma >0,\quad r\geq0.\\
		H^{s,\sigma}(\mathbb{R}^{d})&\hookrightarrow C^{\ell,\sigma}(\mathbb{R}^{d}), \qquad s> d/2+\ell, \quad \ell\geq 0, \quad\sigma> 0.\label{em1}\\
		C^{m,\sigma}(\mathbb{R}^{d})&\hookrightarrow 	W^{m,\sigma+\eta}_{0}(\mathbb{R}^{d}), \qquad \eta>d/2, \quad m\geq 0,\quad \sigma > 0 \label{em2}.
	\end{align}
	\begin{lem}{  (A special case of theorem 2.1 in \cite{kotel}) } \\
		Let $(H, \lVert . \lVert_{H})$ be a separable Hilbert space, $M$ be an $H-$valued locally square integrable càdlàg martingale and \emph{\textbf{T}}(t) a contraction semigroup operator of $\mathcal{L}(H).$ Then there exists a finite constant $C$ depending only on the Hilbert norm $\lVert . \lVert_{H}$ such that for all $T>0.$
		
		\begin{center}
			$\mathbb{E}\left( \underset{0\leq t\leq T}{\sup} \Big \lVert \displaystyle \int_{0}^{t} \emph{\textbf{T}}(t-r)dM_{r} \Big\lVert^{2}_{H}\right)\leq Ce^{4\sigma T}\mathbb{E}\Big(\lVert M_{T}\lVert^{2}_{H}\Big),$ 
		\end{center}
		
		where $\sigma$ is a real number such that $\lVert \emph{T}(t)\lVert_{\mathcal{L}} \leq e^{\sigma t}. \label{bbn}$ 
	\end{lem}
 \begin{dfn} (White noise)\label{white}\\ White noise on $\mathbb{R}^{d}$ is a random distribution $\mathcal{W}$ defined on a probability space ($\Omega$, $\mathcal{F}$, $\mathbb{P}$) which is  such that the mapping $\varphi\mapsto (\mathcal{W},\varphi)$ is linear and continuous from $L^{2}(\mathbb{R}^{d})$ into $L^{2}(\Omega) $ and $\{(\mathcal{W},\varphi),\varphi\in L^{2}(\mathbb{R}^{d})\}$ is  a centered Gaussian generalized process satisfying:
	\vspace*{-0.12cm}
	\begin{center}
		$\mathbb{E}((\mathcal{W}, \varphi)(\mathcal{W}, \phi))=(\varphi, \phi)_{L^{2}}$, for any $\varphi, \phi \in L^{2}(\mathbb{R}^{d}).$
	\end{center}
	Where  $(.,.)_{L^{2}}$ denotes a scalar product on $ L^{2}(\mathbb{R}^{d}).$\\ Space-time white noise  is a white noise on $\mathbb{R}_{+}\times \mathbb{R}^{d}.$ 
	
\end{dfn}
	\section{Law of Large Numbers }
	The aim of this section is to study the convergence  of $
	(\mu^{S,N},\mu^{I,N},\mu^{R,N}),$ as $ N\rightarrow \infty $ under  Assumption (H1) below.  \vspace*{0.2cm}\\
	To this end   we are going to:
	\begin{itemize}
		
	\item  Write the system of evolution equations of $( \mu^{S,N},\mu^{I,N},\mu^{R,N}).$
	\item Study the tightness of $  (\mu^{S,N},\mu^{I,N},\mu^{R,N})_{ N \geq 1}$ in Skorokhod's space $[\mathbb{D}(\mathbb{R_{+}},(\mathcal{M}_{F}(\mathbb{R}^{d}),v))]^{3},$ where $(\mathcal{M}_{F}(\mathbb{R}^{d}),v)$ is the space of finite measure on $\mathbb{R}^{d},$ equipped with the vague topology. 
	\item Find the system of evolution equations satisfies by the limit in law  $(\mu^{S},\mu^{I },\mu^{R})$ of a  convergent  subsequence of  $ (\mu^{S,N} \mu^{I,N},\mu^{R,N})_{N\geq 1}$
	\item Show that the system of  PDEs verified by   $(\mu^{S}, \mu^{I},\mu^{R})$  admits a unique   solution in  \\  $\Lambda=\{(\mu^{1},\mu^{2},\mu^{3})/0\leq (\mu^{i},\mathbb{1}) \leq 1, i\in \{1,2,3\}\}.$ .
\end{itemize}
The following is assumed to hold throughout this section. \\\\
	\textbf{{\color{blue}{Assumption (H1)}}}
	\begin{itemize}
		\item The law of $X^{1}_{0}$ is absolutly continuous with respect to the Lebesgue measure and its  density  is bounded.
	\item 	 $K\in C_{c}(\mathbb{R}^{d}\times \mathbb{R}^{d}).$ 
		\item  For any $A\in \{S,I,R\},$ and $x\in \mathbb{R}^{d},$ the matrix $(\theta\theta^{t})(A,x)$ is invertible.
	\end{itemize}.  
	\subsection{ System of evolution equations  of  $ (\mu^{S,N},\mu^{I,N}, \mu^{R,N}) $  \label{hm}}
	In this subsection we shall establish the following result.
	
	\begin{prp}\label{bnv} 
		For any $\varphi \in C^{2}_{c}(\mathbb{R}^{d})$, $\{(\mu^{S,N}_{t},\varphi),(\mu^{I,N}_{t},\varphi),(\mu^{R,N}_{t},\varphi)\}$ satisfies,
		\begin{align}
			(\mu_{t}^{S,N},\varphi)&=(\mu_{0}^{S,N},\varphi)+\int_{0}^{t}(\mu_{r}^{S,N},\mathcal{Q}_{S}\varphi) dr-\int_{0}^{t}\left(\mu_{r}^{S,N}, \varphi (\mu_{r}^{I,N}, K)\right) dr + M_{t}^{N,\varphi},
			\label{e1}\\
			(\mu_{t}^{I,N}, \varphi)&= \displaystyle (\mu_{0}^{I,N},\varphi) + \int_{0}^{t}(\mu_{r}^{I,N},\mathcal{Q}_{I}\varphi)dr + \int_{0}^{t}\left(\mu_{r}^{S,N}, \varphi (\mu_{r}^{I,N}, K)\right) dr-\alpha \int_{0}^{t}(\mu_{r}^{I,N},\varphi)dr+L^{N,\varphi}_{t},\label{e2} \\
			(\mu_{t}^{R,N}, \varphi)&=\displaystyle   
			\int_{0}^{t}(\mu_{r}^{R,N},\mathcal{Q}_{R}\varphi)dr+\alpha \int_{0}^{t}(\mu_{r}^{I,N},\varphi)dr  +Y_{t}^{N,\varphi} \label{e3}.   
		\end{align}
		Where \begin{center}
			
			$ \displaystyle\left(\mu_{r}^{S,N}, \varphi (\mu_{r}^{I,N}, K)\right) =\int_{\mathbb{R}^{d}}\varphi(x)\int_{\mathbb{R}^{d}}K(x,y)\mu_{r}^{I,N}(dy)\mu_{r}^{S,N}(dx); $
		\end{center}
	\begin{center}  
		$\mathcal{Q}_{A}\varphi(x)=\displaystyle m(A,x).\bigtriangledown \varphi(x) + \frac{1}{2}\displaystyle \displaystyle\sum\limits_{1\leq \ell,u\leq d}(\theta \hspace*{0.1cm} \theta^{t})_{\ell,u}(S,x)\frac{\partial^{2}\varphi}{\partial  x_{\ell}x_{u}}(x),$
		\end{center}
		and with $\{M^i\}_{1\le i\le N}$, $\{Q^i\}_{1\le i\le N}$ two collections of  standard (i.e. with mean the Lebesgue measure) Poisson random measures (abbreviated below PRM) on $\mathbb{R}^2_+$, which are such that $B^1,M^1,Q^1,\ldots,B^N,M^N,Q^N$ are mutually independent, and denoting by $\overline{M}^i$ and $\overline{Q}^i$ the compensated PRMs associated to $M^i$ and $Q^i$, we have
		\begin{align*}
			M_{t}^{N,\varphi}&= - \displaystyle\frac{1}{N} \sum_{i=1}^{N} \int_{0}^{t} \int_{0}^{\infty}1_{\{E_{r^{-}}^{i}=S\}}\varphi(X_{r}^{i})1_{\{u\leq \frac{1}{N} \sum_{j=1}^{N}K(X_{r}^{i},X_{r}^{j}) 1_{\{E_{r}^{j}=I\}}\}} \overline{M}^{i}(dr,du) \\&+\frac{1}{N} \sum_{i=1}^{N} \int_{0}^{t}1_{\{E_{r}^{i}=S\}}\bigtriangledown\varphi(X_{r}^{i})\theta(S,X_{r}^{i})dB_{r}^{i}.\\
			L_{t}^{N,\varphi} &= \displaystyle\frac{1}{N} \sum_{i=1}^{N} \int_{0}^{t} \int_{0}^{\infty}1_{\{E_{r^{-}}^{i}=S\}}\varphi(X_{r}^{i})1_{\{u\leq \frac{1}{N} \sum_{j=1}^{N}K(X_{r}^{i},X_{r}^{j}) 1_{\{E_{r}^{j}=I\}}\}} \overline{M}^{i}(dr,du))  \\ &+\frac{1}{N} \sum_{i=1}^{N} \int_{0}^{t}1_{\{E_{r}^{i}=I\}}\bigtriangledown\varphi(X_{r}^{i})\theta(I,X_{r}^{i})dB_{r}^{i} -\displaystyle \frac{1}{N} \sum_{i=1}^{N} \int_{0}^{t} \int_{0}^{\alpha}1_{\{E_{r}^{i}=I\}}\varphi(X_{r}^{i})\overline{Q}^{i}(dr,du). \\
			Y_{t}^{N,\varphi} &= \displaystyle\frac{1}{N} \sum_{i=1}^{N} \int_{0}^{t}1_{\{E_{r}^{i}=R\}}\bigtriangledown\varphi(X_{r}^{i})\theta(R,X_{r}^{i})dB_{r}^{i}  +\displaystyle \frac{1}{N} \sum_{i=1}^{N} \int_{0}^{t} \int_{0}^{\alpha}1_{\{E_{r^{-}}^{i}=I\}}\varphi(X_{r}^{i})\overline{Q}^{i}(dr,du).
		\end{align*}
	\end{prp}
	
	\begin{proof}
		Let us first recall that for any $ t\geq 0$, 
		\[ X^{i}_{t}=X^{i}_{0}+\displaystyle\int_{0}^{t}m(E^{i}_{r},X^{i}_{r})dr+\int_{0}^{t} \theta(E^{i}_{r},X^{i}_{t})dB_{r}^{i}.\]
		Let $\varphi\in C_{c}^{2}(\mathbb{R}^{d}),$  according to It\^o's formula we have,
		\begin{align}
			\varphi(X_{t}^{i})&=\varphi(X^{i})+\displaystyle\int_{0}^{t}\bigtriangledown\varphi(X_{r}^{i}).m(E_{r}^{i},X_{r}^{i})dr+\frac{1}{2}\int_{0}^{t}\sum\limits_{1\leq \ell,u\leq d}(\theta \hspace*{0.1cm} \theta^{t})_{\ell,u}(E_{r}^{i},X_{r}^{i})\frac{\partial^{2}\varphi}{\partial  x_{\ell}x_{u}}(X_{r}^{i})dr  \nonumber\\&+\displaystyle\int_{0}^{t}\bigtriangledown\varphi(X_{r}^{i})\theta(E_{r}^{i},X_{r}^{i})dB_{r}^{i}.\label{l3}
		\end{align}
		On the other hand
		\begin{align}
			1_{\{E_{t}^{i}=S\}}&=1_{\{E_{0}^{i}=S\}}-\displaystyle\int_{0}^{t}\int_{0}^{\infty}1_{\{u\leq\frac{1}{N}\sum_{j=1}^{N} K(X_{r}^{i},X_{r}^{j})1_{\{E_{r}^{j}=I\}}\} } 1_{\{E_{r^{-}}^{i}=S\}}M^{i}(du,dr).\label{l2}
		\end{align}
		Hence using  (\ref{l3}) and (\ref{l2}), we have\\ $1_{\{E_{t}^{i}=S\}}\varphi({X_{t}^{i}})=\displaystyle 1_{\{E_{0}^{i}=S\}}\varphi({X^{i}})+\displaystyle\int_{0}^{t}1_{\{E_{r}^{i}=S\}}\bigtriangledown\varphi(X_{r}^{i}).m(S,X_{r}^{i})dr\\\hspace*{2cm}+\frac{1}{2}\int_{0}^{t}1_{\{E_{r}^{i}=S\}}\sum\limits_{1\leq \ell,u\leq d}(\theta \hspace*{0.1cm} \theta^{t})_{\ell,u}(S,X_{r}^{i})\frac{\partial^{2}\varphi}{\partial  x_{\ell}x_{u}}(X_{r}^{i})dr\\\hspace*{2cm}+\int_{0}^{t}1_{\{E_{r}^{i}=S\}}\bigtriangledown\varphi({X_{r}^{i}})\theta(S,X_{r}^{i})dB_{r}^{i}\\\hspace*{2cm}-\int_{0}^{t}\int_{0}^{\infty}1_{\{u\leq\frac{1}{N}\sum_{j=1}^{N}  K(X_{r}^{i},X_{r}^{j})1_{\{E_{r}^{j}=I\}}\} } 1_{\{E_{r^{-}}^{i}=S\}}\varphi({X_{r}^{i}})M^{i}(du,dr).$\vspace*{0.15cm}\\
		Taking the sum over $i$ and multiplying by $\frac{1}{N},$ we obtain \vspace*{0.2cm}\\$\displaystyle\frac{1}{N}\sum_{i=1}^{N}1_{\{E_{t}^{i}=S\}}\varphi({X_{t}^{i}})=\displaystyle \displaystyle\frac{1}{N}\sum_{i=1}^{N}1_{\{E_{0}^{i}=S\}}\varphi({X^{i}})\\\hspace*{3cm}+\displaystyle\displaystyle\frac{1}{N}\sum_{i=1}^{N}\int_{0}^{t}1_{\{E_{r}^{i}=S\}}\bigtriangledown\varphi(X_{r}^{i}).m(S,X_{r}^{i})dr\\\hspace*{3cm}+ \displaystyle\frac{1}{2N}\sum_{i=1}^{N}\int_{0}^{t}1_{\{E_{r}^{i}=S\}}\sum\limits_{1\leq \ell,u\leq2}(\theta \hspace*{0.1cm} \theta^{t})_{\ell,u}(S,X_{r}^{i})\frac{\partial^{2}\varphi}{\partial  x_{\ell}x_{u}}(X_{r}^{i})dr\\\hspace*{3cm}+\displaystyle\frac{1}{N}\sum_{i=1}^{N}\int_{0}^{t}1_{\{E_{r}^{i}=S\}}\bigtriangledown\varphi({X_{r}^{i}})\theta(S,X_{r}^{i})dB_{r}^{i}\\\hspace*{3cm}-\frac{1}{N}\sum_{i=1}^{N}\int_{0}^{t}\int_{0}^{\infty}1_{\{u\leq \frac{1}{N}\sum_{j=1}^{N}K(X_{r}^{i},X_{r}^{j})1_{\{E_{r}^{j}=I\} } \}} 1_{\{E_{r^{-}}^{i}=S\}}\varphi({X_{r}^{i}})\overline{M}^{i}(du,dr)\\\hspace*{3cm}-\frac{1}{N}\sum_{i=1}^{N}\int_{0}^{t} \frac{1}{N}\sum_{j=1}^{N}K(X_{r}^{i},X_{r}^{j})1_{\{E_{r}^{j}=I\}} 1_{\{E_{r}^{i}=S\}}\varphi({X_{r}^{i}})dr,$\\from which (\ref{e1}) follows. Similarly, with again $\varphi \in C^{2}_{c}(\mathbb{R}^{d})$, $\{1_{\{E_{t}^{i}=I\}}\varphi(X_{t}^{i}),t\geq0 \} $ is a jump process satisfying, \\
		$1_{\{E_{t}^{i}=I\}}\varphi({X_{t}^{i}})=\displaystyle 1_{\{E_{0}^{i}=I\}}\varphi({X^{i}})+\displaystyle\int_{0}^{t}1_{\{E_{r}^{i}=I\}}\bigtriangledown\varphi(X_{r}^{i}).m(I,X_{r}^{i})dr\\\hspace*{2.5cm} +\frac{1}{2}\int_{0}^{t}1_{\{E_{r}^{i}=I\}}\sum\limits_{1\leq \ell,u\leq2}(\theta \hspace*{0.1cm} \theta^{t})_{\ell,u}(I,X_{r}^{i})\frac{\partial^{2}\varphi}{\partial  x_{\ell}x_{u}}(X_{r}^{i})dr\vspace{0.2cm}\\\hspace*{2.5cm}+\int_{0}^{t}1_{\{E_{r}^{i}=I\}}\bigtriangledown\varphi({X_{r}^{i}})\theta(I,X_{r}^{i})dB_{r}^{i}\\\hspace*{2.5cm}+
		\int_{0}^{t}\int_{0}^{\infty}1_{\{u\leq\frac{1}{N}\sum_{j=1}^{N} K(X_{r}^{i},X_{r}^{j})1_{\{E_{r}^{j}=I\}}\} } 1_{\{E_{r^{-}}^{i}=S\}}\varphi({X_{r}^{i}})M^{i}(du,dr)\\\hspace*{2.5cm}-
		\int_{0}^{t}\int_{0}^{\alpha} 1_{\{E_{r^{-}}^{i}=I\}}\varphi({X_{r}^{i}})Q^{i}(du,dr).$\vspace*{0.13cm}\\ 
		Summing over $i$  and multiplying by $\frac{1}{N},$ we obtain \vspace*{0.1cm}\\
		\\$\displaystyle\frac{1}{N}\sum_{i=1}^{N}1_{\{E_{t}^{i}=I\}}\varphi({X_{t}^{i}})=\displaystyle \displaystyle\frac{1}{N}\sum_{i=1}^{N}1_{\{E_{0}^{i}=I\}}\varphi({X^{i}})\\\hspace*{3cm}+\displaystyle\displaystyle\frac{1}{N}\sum_{i=1}^{N}\int_{0}^{t}1_{\{E_{r}^{i}=I\}}\bigtriangledown\varphi(X_{r}^{i}).m(I,X_{r}^{i})dr\\\hspace*{3cm}+ \displaystyle\frac{1}{2N}\sum_{i=1}^{N}\int_{0}^{t}1_{\{E_{r}^{i}=I\}}\sum\limits_{1\leq \ell,u\leq d}(\theta \hspace*{0.1cm} \theta^{t})_{\ell,u}(I,X_{r}^{i})\frac{\partial^{2}\varphi}{\partial  x_{\ell}x_{u}}(X_{r}^{i})dr\\\hspace*{3cm}+\displaystyle\frac{1}{N}\sum_{i=1}^{N}\int_{0}^{t}1_{\{E_{r}^{i}=I\}}\bigtriangledown\varphi({X_{r}^{i}})\theta(I,X_{r}^{i})dB_{r}^{i}\\\hspace*{3cm}-\frac{1}{N}\sum_{i=1}^{N}\int_{0}^{t}\int_{0}^{\infty}1_{\{u\leq \frac{1}{N}\sum_{j=1}^{N}K(X_{r}^{i},X_{r}^{j})1_{\{E_{r}^{j}=I\} } \}} 1_{\{E_{r^{-}}^{i}=S\}}\varphi({X_{r}^{i}})\overline{M}^{i}(du,dr)\\\hspace*{3cm}-\frac{1}{N}\sum_{i=1}^{N}\int_{0}^{t} \frac{1}{N}\sum_{j=1}^{N}K(X_{r}^{i},X_{r}^{j})1_{\{E_{r}^{j}=I\}} 1_{\{E_{r}^{i}=S\}}\varphi({X_{r}^{i}})dr\\\hspace*{3cm}- \frac{1}{N}\sum_{i=1}^{N}\int_{0}^{t}\int_{0}^{\alpha} 1_{\{E_{r^{-}}^{i}=I\}}\varphi({X_{r}^{i}})\overline{Q}^{i}(du,dr)-  \frac{\alpha}{N}\sum_{i=1}^{N}\int_{0}^{t} 1_{\{E_{r}^{i}=I\}}\varphi({X_{r}^{i}})dr,$\\from which $(\ref{e2})$ follows. Similarly, with once again $\varphi \in C_{c}^{2}(\mathbb{R}^{d})$,  $\{1_{\{E_{t}^{i}=R\}}\varphi(X_{t}^{i}),t\geq0 \} $ is a jump processes satisfying, \\
		$1_{\{E_{t}^{i}=R\}}\varphi({X_{t}^{i}})=\displaystyle \int_{0}^{t}1_{\{E_{r}^{i}=R\}}\bigtriangledown\varphi(X_{r}^{i}).m(R,X_{r}^{i})dr \vspace*{0.15cm}\\\hspace*{2.3cm}+\int_{0}^{t}1_{\{E_{r}^{i}=R\}}\sum\limits_{1\leq \ell,u\leq d}(\theta \hspace*{0.1cm} \theta^{t})_{\ell,u}(R,X_{r}^{i})\frac{\partial^{2}\varphi}{\partial  x_{\ell}x_{u}}(X_{r}^{i})dr\vspace*{0.15cm}\\\hspace*{2.3cm}+ \int_{0}^{t}1_{\{E_{r}^{i}=R\}}\bigtriangledown\varphi({X_{r}^{i}})\theta(R,X_{r}^{i})dB_{r}^{i}+
		\int_{0}^{t}\int_{0}^{\alpha} 1_{\{E_{r^{-}}^{i}=I\}}\varphi({X_{r}^{i}})Q^{i}(du,dr).$\\ 
		Summing over $i$  and multiplying by $\frac{1}{N}$ we obtain, \vspace*{0.2cm}\\
		$\displaystyle\frac{1}{N}\sum_{i=1}^{N}1_{\{E_{t}^{i}=R\}}\varphi({X_{t}^{i}})=
		\displaystyle \frac{1}{N}\sum_{i=1}^{N}\int_{0}^{t}1_{\{E_{r}^{i}=R\}}\bigtriangledown\varphi(E_{r}^{i},X_{r}^{i}).m(R,X_{r}^{i})dr \\\hspace*{4cm}+\frac{1}{2N}\sum_{i=1}^{N}\int_{0}^{t}1_{\{E_{r}^{i}=R\}}\sum\limits_{1\leq \ell,u\leq d}(\theta \hspace*{0.1cm} \theta^{t})_{\ell,u}(R,X_{r}^{i})\frac{\partial^{2}\varphi}{\partial  x_{\ell}x_{u}}(X_{r}^{i})dr\vspace*{0.15cm}\\\hspace*{4cm}+\frac{1}{N}\sum_{i=1}^{N} \int_{0}^{t}1_{\{E_{r}^{i}=R\}}\bigtriangledown\varphi({X_{r}^{i}})\theta(R,X_{r}^{i})dB_{r}^{i}\\\hspace*{4cm}+ \frac{1}{N}\sum_{i=1}^{N}\int_{0}^{t}\int_{0}^{\alpha} 1_{\{E_{r^{-}}^{i}=I\}}\varphi({X_{r}^{i}})\overline{Q}^{i}(du,dr)+  \frac{\alpha}{N}\sum_{i=1}^{N}\int_{0}^{t} 1_{\{E_{r}^{i}=I\}}\varphi({X_{r}^{i}})dr,$\\from which $(\ref{e3})$ follows.
	\end{proof}
	\subsection{Tightness and Convergence of   $  (\mu^{S,N},\mu^{I,N},\mu^{R,N})$  in \small{ $[\mathbb{D}(\mathbb{R_{+}},\mathcal{M}_{F}(\mathbb{R}^{d}))]^{3}$}}
	Recall that  we equip $\mathcal{M}_{F}(\mathbb{R}^{d})$ with the topology of weak convergence and the Skorokhod space of c\'adl\'ag function from $\mathbb{R}_{+}$ into $\mathcal{M}_{F}(\mathbb{R}^{d})$ with the Skorokhod topology (we refer to page 63 of \cite{toc} for an explicit definition).
	\begin{rmq}
		For  $ A \in \{S,I,R\},$ we have $  (\mu_{t}^{A,N},1_{\mathbb{R}^{d}})=\frac{1}{N}\sum_{i=1}^{N}1_{\{E_{t}^{i}=A\}}\leq 1, $\\
		thus 
		for any  $ \varphi \in C_{c}(\mathbb{R}^{d})$,  $  A \in \{S,I,R\}$,                       
		$ \lvert ( \mu_{t}^{A,N},\varphi)\lvert\leq\lVert \varphi \lVert_{\infty}.$
	\end{rmq}
	We can now establish the wished tightness.
	\begin{prp}
		The sequences  $ (\mu^{S,N})_{N\geq 1}$;  $(\mu^{I,N})_{N \geq 1}$ and  $(\mu^{R,N})_{N \geq 1}$ are   tight in  \\ $  \mathbb{D}(\mathbb{R_{+}},(\mathcal{M}_{F}(\mathbb{R}^{d}),v)),$ where $\mathcal{M}_{F}(\mathbb{R}^{d})$  is equipped with the  topology of vague convergence. \label{ll1}
	\end{prp}
	
	\begin{proof}
		Let us  prove that  $ (\mu^{S,N})_{N \geq 1} $ is tight in $\mathbb{D}(\mathbb{R_{+}},(\mathcal{M}_{F}(\mathbb{R}^{d} ),v)),$ where $\mathcal{M}_{F}(\mathbb{R}^{d}) $ is equipped with the vague topology. 	We refer to Theorem 2.2  of Roelly \cite{uc}. Let  $\Xi$ be a dense subset of $C_{0}(\mathbb{R}^{d}),$ a sufficient condition for  $ (\mu^{S,N})_{N \geq 1 } $ to be tight in $\mathbb{D}(\mathbb{R_{+}},(\mathcal{M}_{F}(\mathbb{R}^{d}),v))$  is   that: \[\textrm{for any} \quad\varphi \in \Xi, \quad  \{(\mu^{S,N}_{t},\varphi), \quad t\geq0,\quad N \geq 1 \} \textrm{ is tight in }  \mathbb{D}(\mathbb{R_{+}} ,\mathbb{R}). \] We choose $\Xi=C^{\infty}_{c}(\mathbb{R}^{d})$ (= the space of infinitely differentiable functions with compact support).
		Let $\varphi \in C^{\infty}_{c}(\mathbb{R}^{d}) $,  we have \vspace*{0.2cm}\\ $ \displaystyle(\mu_{t}^{S,N},\varphi)=(\mu_{0}^{S,N},\varphi)+\int_{0}^{t}(\mu_{r}^{S,N},\mathcal{Q}_{S}\varphi) dr-  \int_{0}^{t}\left(\mu_{r}^{S,N}, \varphi (\mu_{r}^{I,N}, K)\right) dr + M_{t}^{N,\varphi}$,
		\vspace*{0.12cm}\\
		We notice that $  (\mu^{S,N},\varphi) $ is a semi-martingale since $ M^{N,\varphi}$ is a square integrable  martingale. \\
		Indeed, $ M^{N,\varphi}$ is a local martingale as the  sum of local martingales and   
		
		\begin{align*}
			<M^{N,\varphi}>_{t} &= \displaystyle\frac{1}{N} \int_{0}^{t}\left(\mu_{r}^{S,N}, \varphi^{2} (\mu_{r}^{I,N}, K)\right)dr\\&+ \frac{1}{N} \int_{0}^{t} \Big(\mu_{r}^{S,N}, \sum\limits_{1\leq \ell\leq d}\big(\frac{\partial \varphi}{\partial x_{\ell}}\big)^{2}\sum\limits_{1\leq u\leq d}\theta^{2}_{\ell,u}(S,.)+2\sum\limits_{\underset{1\leq e\leq d}{\underset{\ell+1\leq u\leq d}{1\leq \ell\leq d-1}}}\frac{\partial \varphi}{\partial x_{\ell}}\frac{\partial \varphi}{\partial x_{u}}\theta_{\ell,e}(S,.)\theta_{u,e}(S,.)\Big) dr, \\ 
			&\leq  \displaystyle\frac{1}{N} \int_{0}^{t}\left\lvert\int_{\overline{\mathbb{X}}} \varphi^{2}(x) \int_{\overline{\mathbb{X}}\times\overline{\mathbb{X}}} K(x,y) \mu_{r}^{I,N}(dy) \mu_{r}^{S,N}(dx)\right\lvert dr \\& + \frac{t}{N}   \sum\limits_{1\leq \ell,u\leq d}\Big\lVert\frac{\partial \varphi}{\partial x_{\ell}}\Big\lVert^{2}_{\infty}\big\lVert\theta_{\ell,u}(S,.)\big\lVert^{2}_{\infty}+\frac{2t}{N}\sum\limits_{\underset{1\leq e\leq d}{\underset{\ell+1\leq u\leq d}{1\leq \ell\leq d-1}}}\Big\lVert\frac{\partial \varphi}{\partial x_{\ell}}\Big\lVert_{\infty}\Big\lVert\frac{\partial \varphi}{\partial x_{u}}\Big\lVert_{\infty}\lVert\theta_{\ell,e}(S,.)\lVert_{\infty}\lVert\theta_{u,e}(S,.)\lVert_{\infty}, \\& 
			\leq\displaystyle  \frac{t \Arrowvert \varphi\Arrowvert_{\infty}^{2}\Arrowvert K\Arrowvert_{\infty}}{N}  + \frac{t}{N}   C.
		\end{align*}
		thus $\mathbb{E}(<M^{N,\varphi}>_{t})<\infty,\hspace*{0.15cm} \forall t\geq 0.$\vspace*{0.2cm}\\
		On other hand  \\\hspace*{2cm} $ (\mu_{t}^{S,N},\varphi)=\displaystyle(\mu_{0}^{S,N},\varphi) +\int_{0}^{t} \omega_{r}^{N,\varphi} dr+  M^{N,\varphi}_{t} $ with  $  < M^{N,\varphi}>_{t} =\displaystyle\int_{0}^{t} \varpi_{r}^{N,\varphi}dr, $\\ 
		and\vspace*{0.2cm}\\
		$    \omega_{t}^{N,\varphi}=  \Big(\mu_{t}^{S,N}, m(A,.).\bigtriangledown \varphi(.) + \frac{1}{2}\displaystyle \displaystyle\sum\limits_{1\leq \ell,u\leq d}(\theta \hspace*{0.1cm} \theta^{t})_{\ell,u}(S,.)\frac{\partial^{2}\varphi}{\partial  x_{\ell}x_{u}}(.)\Big)  -  \left(\mu_{t}^{S,N}, \varphi (\mu_{t}^{I,N}, K)\right), $\\  $ \varpi_{t}^{N,\varphi}=\displaystyle\frac{1}{N} \left(\mu_{t}^{S,N},\varphi^{2}(\mu_{t}^{I,N}, K) \right)  \\\hspace*{1cm}+\frac{1}{N}  \Big(\mu_{t}^{S,N},\sum\limits_{1\leq \ell\leq d}\big(\frac{\partial \varphi}{\partial x_{\ell}}\big)^{2}\sum\limits_{1\leq u\leq d}\theta^{2}_{\ell,u}(S,.)+2\sum\limits_{\underset{1\leq e\leq d}{\underset{\ell+1\leq u\leq d}{1\leq \ell\leq d-1}}}\frac{\partial \varphi}{\partial x_{\ell}}\frac{\partial \varphi}{\partial x_{u}}\theta_{\ell,e}(S,.)\theta_{u,e}(S,.)\Big). $ \\Furthermore   $\omega^{N,\varphi}$ and  $ \varpi^{N,\varphi}$ are progressively measurable since the are adapted and right continuous, so according to proposition 37 of \cite{wc} a sufficient condition for   $\{(\mu_{t}^{S,N}, \varphi ),t\geq0,\hspace*{0.1cm}N\geq 1\}$  to be tight in $ \mathbb{D}(\mathbb{R}_{+},\mathbb{R})$  is that both:
		\begin{itemize}
			\item $ (\mu_{0}^{S,N}, \varphi )_{ N \geq 1 }    $ is tight in  $ \mathbb{R}, $
			\item $ \forall T\geq 0, \underset{0\leq t\leq T}{\sup}(\mid \omega_{t}^{N,\varphi} \mid + \varpi_{t}^{N,\varphi} ) $ is tight in $ \mathbb{R}.\vspace*{-0.12cm} $
		\end{itemize} 
		These follow readily from the facts that:\vspace*{0.13cm}\\\hspace*{0.5cm} $-$ 
		$ \lvert (\mu_{0}^{S,N}, \varphi )\lvert \leq \Arrowvert \varphi \Arrowvert_{\infty}$.\vspace*{0.2cm}\\\hspace*{0.5cm} $-$ 
		$ \mid \omega_{t}^{N,\varphi} \mid \leq \sum\limits_{1\leq \ell\leq d}\lVert m_{\ell}\lVert_{\infty}\lVert \frac{\partial\varphi}{\partial x_{\ell}}(x_{1}....,x_{d})\lVert_{\infty}
		+ \frac{1}{2}\sum\limits_{1\leq \ell,u\leq d}\lVert(\theta\hspace*{0.1cm}\theta^{t})_{\ell,u}(S,.)\lVert_{\infty}\lVert\frac{\partial^{2}\varphi}{\partial x_{\ell}x_{u}}\lVert_{\infty}+\Arrowvert \varphi\Arrowvert_{\infty}\lVert K\lVert_{\infty}.\vspace*{0.2cm}\\\hspace*{2.2cm}\leqslant C$ \vspace*{0.15cm}\\\hspace*{0.5cm} $-$ $ \varpi_{t}^{N,\varphi}\leq \displaystyle  \frac{ \Arrowvert \varphi\Arrowvert_{\infty}^{2}\Arrowvert K\Arrowvert_{\infty}}{N} + C\leq C. $ \vspace*{0.2cm}\\
		The same arguments yield the tightness of   $ \{ \mu_{t}^{I,N}, t\geq 0 , N \geq 1 \} $  and $ \{ \mu_{t}^{R,N}, t\geq 0 , N \geq 1 \} $ in $ \mathbb{D}(\mathbb{R_{+}},(\mathcal{M}_{F}(\mathbb{R}^{d}),v))). $
	\end{proof}
The following Proposition follows from the fact that the jump of $\mu^{A,N}$ are order of $1/N$( see  the proof of Proposition 3.3 in \cite{sb} for the explicit proof).
	
	\begin{prp}
		The limit points $ (\mu^{S})$, $(\mu^{I})$ and  $(\mu^{R})$ of the sequences  $ (\mu^{S,N})_{N \geq 1 }$, $(\mu^{I,N})_{N \geq 1 }$ and $(\mu^{R,N})_{N \geq 1 }$ are  elements of $\mathbb{C}(\mathbb{R}_{+},\mathcal{M}_{F}(\mathbb{R}^{d})).  \label{rtre}$ 
	\end{prp}
	
	\begin{thm}
		The sequence $ (\mu^{S,N},\mu^{I,N},\mu^{R,N})_{N \geq 1 }$ converges in probability, in  $ \left[\mathbb{D}(\mathbb{R}_{+},\mathcal{M}_{F}(\mathbb{R}^{d}))\right]^{3} $ towards  $ (\mu^{S},\mu^{I},\mu^{R})$ $\in \left[ \mathbb{C}(\mathbb{R}_{+},\mathcal{M}_{F}(\mathbb{R}^{d}))\right]^{3}$ which is the unique solution of the following system of equations. For any  $ \varphi \in  
		C_{c}^{2}(\mathbb{R}^{d}),$
		\begin{align}
			\displaystyle   (\mu_{t}^{S},\varphi)&=(\mu_{0}^{S,N},\varphi)+\int_{0}^{t}(\mu_{r}^{S}, \mathcal{Q}_{S}\varphi) dr - \int_{0}^{t}\left(\mu_{r}^{S}, \varphi (\mu_{r}^{I}, K)\right) dr,\vspace*{-0.05cm}\label{e6}\\ (\mu_{t}^{I},\varphi)&=(\mu_{0}^{I},\varphi)+\int_{0}^{t}(\mu_{r}^{I}, \mathcal{Q}_{I}\varphi) dr+ \int_{0}^{t}\left(\mu_{r}^{S}, \varphi (\mu_{r}^{I}, K)\right) dr -\alpha\int_{0}^{t}(\mu_{r}^{I},\varphi)dr,\label{e7}\\
			\displaystyle  \displaystyle (\mu_{t}^{R},\varphi)&= \int_{0}^{t}(\mu_{r}^{R}, \mathcal{Q}_{R}\varphi) dr+\alpha\int_{0}^{t}(\mu_{r}^{I},\varphi)dr.\label{e8}
		\end{align}
		$\label{th1}$
	\end{thm}
	
	\subsubsection{Proof of Theorem \ref{th1}}
	By   Proposition \ref{ll1}, the sequence $(\mu^{S,N},\mu^{I,N},\mu^{R,N})_{N\geq 1}$  is tight in  $\left[\mathbb{D}(\mathbb{R_{+}},(\mathcal{M}_{F}(\mathbb{R}^{d}),v))\right]^{3}$, thus according to Prokhorov's Theorem there exists a subsequence of $(\mu^{S,N},\mu^{I,N},\mu^{R,N})_{N\geq }$ still  denoted $(\mu^{S,N},\mu^{I,N},\mu^{R,N})_{N \geq 1}$ which converges in law in $\left[\mathbb{D}(\mathbb{R_{+}},(\mathcal{M}_{F}(\mathbb{R}^{d}),v))\right]^{3}$ towards $(\mu^{S},\mu^{I},\mu^{R}),$ where $\mathcal{M}_{F}(\mathbb{R}^{d})$ is equipped with the vague topology.  \\ Hence to complete the proof of Theorm \ref{th1} it remains to:
	\begin{itemize}
        \item Find the  system of PDEs satisfes by $\{(\mu_{t}^{S},\mu_{t}^{I},\mu_{t}^{R}), t\geq 0\}$
	\item  Show that the system verifies by $(\mu_{t}^{S},\mu_{t}^{I},\mu_{t}^{R})$ admits a unique solution on \\  $\Lambda=\{(\mu^{1},\mu^{2},\mu^{3})/ 0\leq (\mu^{i},1_{\mathbb{R}^{d}})\leq 1, i\in \{1,2,3\}  \}$.
	\item Conclude.
		\end{itemize}
It is so easy to obtain the following Lemma, therefore we omit the proof.
	 \begin{lem}
If we let $\Sigma=\{(\mu^{1},\mu^{2})\in\mathcal{M}_{F}(\mathbb{R}^{d}), (\mu^{i},1_{\mathbb{R}^{d}})\leq1, \forall i \in \{1,2\}\},$		for any $\varphi,\psi\in C_{c}(\mathbb{R}^{d}),$ the following map is continuous.\vspace*{0.2cm}\\
		$ \hspace*{3cm}
		G_{\varphi,\psi}:(\Sigma,v)\times (\Sigma,v)
		\rightarrow (\mathcal{M}_{F}(\mathbb{R}^{d}\times \mathbb{R}^{d}),v)\\\hspace*{8cm}
		(\mu,\nu) \longmapsto (\mu\otimes \nu,\varphi \psi)$ \\
		
		where   $(\mu \otimes \nu,\varphi \psi)=\displaystyle\int_{\mathbb{R}^{d}\times\mathbb{R}^{d}}\varphi(x) \psi(y)\nu(dy)\mu(dx).$
		\label{e7e}
	\end{lem}
	
	The following Proposition  establishes the system of equations satisfied by $(\mu^{S},\mu^{I},\mu^{R}).$
	\begin{prp}
		The processes  $(\mu^{S},\mu^{I},\mu^{R})$ satisfies the system formed by the equations (\ref{e6}), (\ref{e7}) and (\ref{e8}).
	\end{prp}

	\begin{proof} 	We prove this Proposition by taking the limit in the equations (\ref{e1}), (\ref{e2}) and (\ref{e3}). Let us establish $(\ref{e6})$.\vspace{0.16cm}\\
		1- It has  been shwon in \cite{sb} that the sequence $\{(\mu_{0}^{S,N},\mu_{0}^{I,N}), N\geqslant 1\}$ converges a.s. towards the pair of   deterministic measures $(\mu_{0}^{S},\mu_{0}^{I}).$\vspace*{0.15cm}\\
		2- Since the map $x\in \mathbb{R}^{d}\mapsto \mathcal{Q}_{S}\varphi(x)= m(A,x).\bigtriangledown \varphi(x) + \frac{1}{2}\displaystyle \displaystyle\sum\limits_{1\leq \ell,u\leq d}(\theta \hspace*{0.1cm} \theta^{t})_{\ell,u}(S,x)\frac{\partial^{2}\varphi}{\partial  x_{\ell}x_{u}}(x)$ is continuous with compact support, $\displaystyle \int_{0}^{t}\left(\mu_{r}^{S,N}, \mathcal{Q}_{S}\varphi \right) dr$ converges in law towards $\displaystyle \int_{0}^{t}\left(\mu_{r}^{S}, \mathcal{Q}_{S}\varphi \right) dr.$ \vspace*{0.15cm}\\
		2- Form Proposition \ref{ll1}, the sequence $\mu^{S,N}\otimes\mu^{I,N}$ is also tight, thus from Prokorov's theorem it is possible to extract a sub-sequence still denotes  $(\mu^{S,N}\otimes\mu^{I,N})_{N\geq 1}$  such that  $(\mu^{S,N}\otimes\mu^{I,N})_{N\geq 1}$ and the above subsequences  $(\mu^{S,N},\mu^{I,N},\mu^{R,N})_{N \geq 1}$ converges towards $\chi^{S,I}$ and $(\mu^{S},\mu^{I},\mu^{R})$  respectively. Furthermore the fact that for any $t\geq0,$  $\chi_{t}^{S,I}=\mu^{S}_{t}\otimes\mu^{I}_{t}$ follows from Lemma \ref{e7e}. Consequently, we have \\\\
		$\hspace*{1.5cm}\displaystyle\int_{0}^{t}\left(\mu_{r}^{S,N},\varphi(\mu_{r}^{I,N}, K)\right) dr=\int_{0}^{t}\left(\mu_{r}^{S,N}\otimes\mu_{r}^{I,N}, \varphi  K\right) dr\xrightarrow{L}\int_{0}^{t}\left(\mu_{r}^{S}\otimes\mu_{r}^{I}, \varphi  K\right) dr.$\\\\
		3- Let us prove that the sequences 
		$ M_{t}^{N,\varphi}$ ;  $ L_{t}^{N,\varphi}$ and $Y_{t}^{N}$  converge to 0 in Probability.\vspace*{0.1cm}\\ \hspace*{0.5cm}$-$ Convergence of $ M_{t}^{N,\varphi}.$ We have seen above that\vspace*{0.2cm}\\
		$\hspace*{1cm}\mathbb{E}( \mid M^{N,\varphi}_{t} \mid^{2} ) $=$\displaystyle$ $\mathbb{E}( < M^{N,\varphi}>_{t} ) \\\hspace*{3.1cm}\leq \displaystyle \frac{t}{N} \lVert \varphi\lVert_{\infty}^{2} \lVert k \lVert_{\infty}+\frac{tC}{N}\xrightarrow{N\rightarrow\infty}0, $
		\vspace*{0.2cm} \\
		consequently  $  M^{N,\varphi}_{t} $ converges to 0 in  $ L^ {2} $, so also in probability. By  similar arguments,   we obtain the convergences in probability to 0 of the sequences $  L^{N,\varphi}_{t} $  and  $  Y^{N,\varphi}_{t} $.\\Hence $(\ref{e6})$ follows from 1-, 2-, 3-. Similar arguments yield $(\ref{e7})$ and $(\ref{e8}).$
	\end{proof}
Let us now prove the following proposition which will be useful to show that the system of equations (\ref{e6}), (\ref{e7}) and (\ref{e8}) admits a unique solution.
	\begin{lem}
	For $A\in \{S,I,R\}$, $\Upsilon_{A}(t)$ denotes the Markovian  semi-group of the diffussion process   with diffusion matrix  $\theta(A,.)$  and drift coefficient $m(A,.).$  For any $\varphi\in C_{c}(\mathbb{R}^{d}),$ we have
	\begin{align}
		(\mu_{t}^{S},\varphi)&=(\mu_{0}^{S},\Upsilon_{S}(t)\varphi) \displaystyle   - \int_{0}^{t}\left(\mu_{r}^{S}, \Upsilon_{S}(t-r)\varphi (\mu_{r}^{I}, K)\right) dr,\label{e9}\\ 
		(\mu_{t}^{I},\varphi)&=(\mu_{0}^{I},\Upsilon_{I}(t)\varphi) \displaystyle   + \int_{0}^{t}\left(\mu_{r}^{S}, \Upsilon_{I}(t-r)\varphi (\mu_{r}^{I}, K)\right) dr-\alpha\int_{0}^{t}(\mu_{r}^{I},\Upsilon_{I}(t-r)\varphi)dr ,\label{e10}\\
		(\mu_{t}^{R},\varphi)&= \alpha\int_{0}^{t}(\mu_{r}^{I},\Upsilon_{R}(t-r)\varphi)dr.\label{e11}
	\end{align}                  
	
\end{lem}
\begin{proof}
	We may classically derive from \eqref{e6} a similar formula where the test function $\varphi(x)$ is replaced by  $\psi_{r}(x)=\psi(r,x)$ which is of class $C^{1,2}$ on $[0,t]\times\mathbb{R}^{d}$:
	\begin{align}
		(\mu_{t}^{S},\psi_{t}(.))&=(\mu_{0}^{S},\psi_{0}(.))  +\int_{0}^{t}\Big(\mu_{r}^{S},\frac{\partial}{\partial r}\psi_{r}(.)\Big)dr + \int_{0}^{t}(\mu_{r}^{S},m(S,.).\bigtriangledown \varphi) dr\nonumber\\&+\displaystyle  \frac{1}{2}\int_{0}^{t}(\mu_{r}^{S},Tr[(\theta\hspace*{0.1cm}\theta^{t})(S,.)D^{2}\varphi]) dr - \int_{0}^{t}\left(\mu_{r}^{S}, \psi_{r}(.) (\mu_{r}^{I}, K)\right) dr. \label{e12}
	\end{align}
	
	Let us now consider a continuous fonctions $\varphi$ on $\mathbb{R}^{d},$ with compact support and fix a time $t\in \mathbb{R}_{+}$. We define for $(r,x)\in [0,t]\times\mathbb{R}^{d}$, $\psi_{r}(x)=\Upsilon_{A}(t-r)\varphi(x).$\\Then $\psi$ is the solution  of the  following equation \vspace*{0.2cm}\\
	\hspace*{4cm}$\displaystyle\frac{d\psi_{r}(x)}{dr} +\displaystyle \mathcal{Q}_{S}\varphi(x)=0 $ \quad on $[0,t]\times\mathbb{R}^{d}.$ \vspace*{0.2cm}\\Equation (\ref{e12}) applied to this function $\psi$ yields (\ref{e9}).
	We obtain  (\ref{e10}) and (\ref{e11}) by similar arguments.
\end{proof}
\begin{prp}
	The system formed by the equations (\ref{e9}), (\ref{e10}) and (\ref{e11}), admits a unique solution on the set $\Lambda=\{(\mu_{1},\mu_{2},\mu_{3})\in [\mathcal{M}_{F}(\mathbb{R}^{d})]^{3}/(\mu_{i},\mathbb{1})\leq 1\quad \forall i\in \{1,2,3\}\}. \label{ukkk}$
\end{prp}
\begin{proof}
	Let us recall that the distance in total variation on $\mathcal{M}_{F}(\mathbb{R}^{d})$ is defined by \vspace*{0.2cm}\\
	$\hspace*{0.5cm}\lVert \mu -\nu\lVert_{VT}=$sup$\{\lvert (\mu-\nu,\varphi)\lvert,\hspace*{0.1cm}\varphi$ continuous with compact support and $ \lVert \varphi \lVert_{\infty}\leq 1\}$.\vspace*{0.2cm}\\ Now let $(\mu_{t}^{1},\mu_{t}^{2},\mu_{t}^{3})$ and $(\nu_{t}^{1},\nu_{t}^{2},\nu_{t}^{3})$ be 
	two solutions of the system of  equations (\ref{e9}), (\ref{e10}) and (\ref{e11}) with the same initial condition and $\varphi \in C_{c}(\mathbb{R}^{d}).$
	Since $\Upsilon_{S}(t)$ is a contraction semi-group on $C_{c}(\mathbb{R}^{d}),$ we have\vspace*{0.15cm}\\\hspace*{1.5cm}$ \Big\lvert\left(\mu_{r}^{1}, \Upsilon_{S}(t-r)\varphi (\mu_{r}^{2}, K)\right)- \left(\nu_{r}^{1}, \Upsilon_{S}(t-r)\varphi (\nu_{r}^{2}, K)\right)\Big\lvert\vspace*{0.2cm}\\\hspace*{2.5cm}=\displaystyle\Big\lvert\left(\mu_{r}^{1}-\nu_{r}^{1}, \Upsilon_{S}(t-r)\varphi (\mu_{r}^{2}, K)\right)-\left(\nu_{r}^{1}, \Upsilon_{S}(t-r)\varphi (\nu_{r}^{2}-\mu_{r}^{2}, K)\right)\Big\lvert,\vspace*{0.2cm}\\\hspace*{2.5cm}\leq\Big\lvert \int_{\mathbb{R}^{d}}\Upsilon_{S}(t-r)\varphi(x)(\mu_{r}^{2}, K(x,.))(\mu_{r}^{1}-\nu_{r}^{1})(dx)\Big\lvert \vspace*{0.2cm}\\\hspace*{2.5cm}+ \displaystyle\Big\lvert\int_{\mathbb{R}^{d}}\Upsilon_{S}(t-r)\varphi(x)\int_{\mathbb{R}^{d}}K(x,y)(\mu_{r}^{2}-\nu_{r}^{2})(dy)\nu_{r}^{1}(dx)\Big\lvert,\vspace*{0.2cm}\\\hspace*{2.5cm}\leq \lVert \Upsilon_{S}(t-r)\varphi (\mu_{r}^{2}, K)\lVert_{\infty}\lVert \mu_{r}^{1}-\nu_{r}^{1} \lVert_{VT}+\lVert \Upsilon_{S}(t-r)\varphi \lVert_{\infty}\lVert K\lVert_{\infty}\lVert \mu_{r}^{2}-\nu_{r}^{2} \lVert_{VT}, \vspace*{0.2cm}\\\hspace*{2.5cm}\leq \lVert \varphi \lVert_{\infty}\lVert K\lVert_{\infty} \Big\{\lVert \mu_{r}^{1}-\nu_{r}^{1} \lVert_{VT}+\lVert \mu_{r}^{2}-\nu_{r}^{2} \lVert_{VT}\Big\}.$\vspace*{0.2cm}\\Thus using the equations (\ref{e9}), (\ref{e10}) and (\ref{e11}) respectively we obtain 
	\begin{align}
		\underset{\lVert \varphi \lVert_{\infty}\leq 1}{\sup}\lvert  (\mu_{t}^{1}-\nu_{t}^{1},\varphi)\lvert&\leq \lVert K\lVert_{\infty} \displaystyle\int_{0}^{t}\Big\{\lVert \mu_{r}^{1}-\nu_{r}^{1} \lVert_{VT}+\lVert \mu_{r}^{2}-\nu_{r}^{2} \lVert_{VT}\Big\}dr,\label{um}\\
		\underset{\lVert \varphi \lVert_{\infty}\leq 1}{\sup}\lvert  (\mu_{t}^{2}-\nu_{t}^{2},\varphi)\lvert&\leq \lVert K\lVert_{\infty}(1+\alpha) \displaystyle\int_{0}^{t}\Big\{\lVert \mu_{r}^{1}-\nu_{r}^{1} \lVert_{VT}+\lVert \mu_{r}^{2}-\nu_{r}^{2} \lVert_{VT}\Big\}dr,\label{um1}\\
		\underset{\lVert \varphi \lVert_{\infty}\leq 1}{\sup}\lvert  (\mu_{t}^{3}-\nu_{t}^{3},\varphi)\lvert&\leq \alpha\int_{0}^{t}\lVert \mu_{r}^{2}-\nu_{r}^{2} \lVert_{VT},\label{um3}
	\end{align}
	where the suppremum is taken over continuous functions with compact support. Consequently summing the equations (\ref{um}), (\ref{um1}) and (\ref{um3})  the result follows from Gronwall's Lemma.
\end{proof}
We can now complete the proof of Theorem \ref{th1}.\vspace*{0.1cm}\\ Since $(\mu^{S,N},\mu^{I,N},\mu^{R,N})_{N\geq 1}$ is tight in  $[\mathbb{D}(\mathbb{R}_{+},(\mathcal{M}_{F}(\mathbb{R}^{d}),v))]^{3}$, and all   converging subsequences of the sequence $(\mu^{S,N},\mu^{I,N},\mu^{R,N})_{N\geq 1}$  converge in law  in $[\mathbb{D}(\mathbb{R}_{+},(\mathcal{M}_{F}(\mathbb{R}^{d}),v))]^{3} $ to the same limit $(\mu^{S},\mu^{I},\mu^{R})$, where $\mathcal{M}_{F}(\mathbb{R}^{d})$ is equipped with the vague topology,  the sequence $(\mu^{S,N},\mu^{I,N},\mu^{R,N})_{N\geq 1}$  converge in law in  $[\mathbb{D}(\mathbb{R}_{+},(\mathcal{M}_{F}(\mathbb{R}^{d}),v))]^{3}$ towards  $(\mu^{S},\mu^{I},\mu^{R})$.
To extend this result to the weak topology, we use a criterion (Proposition 3) proved in \cite{qc}. Since from Proposition \ref{rtre}  the
limiting process $(\mu^{S},\mu^{I},\mu^{R})$ is continuous, it suffices to prove that the sequence \\ $\Big((\mu^{S,N},1),(\mu^{I,N},1),(\mu^{R,N},1)\Big)_{N\geq 1}$ converges in law to $\Big((\mu^{S},1),(\mu^{I},1),(\mu^{R},1)\Big)$ in $[\mathbb{D}(\mathbb{R}_{+},\mathbb{R})]^{3},$ which follows from and  Proposition \ref{ukkk} and the fact that: 
\begin{itemize}
	\item Proposition \ref{bnv} remains true for the functions $\varphi\in  C^{2}_{b}(\mathbb{R}^{d})$.
	\item Following the Proof of Proposition \ref{ll1}, we see  that the sequence \\ $\Big((\mu^{S,N},1),(\mu^{I,N},1),(\mu^{R,N},1)\Big)_{N\geq 1}$ is tight  in $[\mathbb{D}(\mathbb{R}_{+},\mathbb{R})]^{3}.$
	\item From Prokorov's Theorem we deduce the existence of a subsequence which converge in law towards $\Big((\mu^{S},1),(\mu^{I},1),(\mu^{R},1)\Big)$.
	\item Proposition \ref{e8} remains true  when the test function  $\varphi$ is a constant.
\end{itemize} 

 Finally since  the sequence $(\mu^{S,N},\mu^{I,N},\mu^{R,N})_{N\geq 1}$  weakly converge in $(\mathbb{D}(\mathbb{R}_{+},\mathcal{M}_{F}(\mathbb{R}^{d})))^{3}$ towards  $(\mu^{S},\mu^{I},\mu^{R})$ and   $(\mu^{S},\mu^{I},\mu^{R})$ is deterministic,  we have  convergence in probability.
\subsection{Existence of densities}
	The third assumptions   in $(H1),$  allow us to  obtain the following result (see Theorem 22 in \cite{sm}). 
\begin{rmq}
	Under the  assumption (H1), For $A\in \{S,I,R\},$ there exists  a  measurable function $\Upsilon_{A}(t)(x,y)$,  defined on $\mathbb{R}^{d}\times\mathbb{R}^{d}$, which is a density in $y\in \mathbb{R}^{d} $ and such that for each continuous function $\varphi$
	defined on $\mathbb{R}^{d},$ one has \vspace*{0.2cm}\\
	\hspace*{4.5cm} $\Upsilon_{A}(t)\varphi(x)=\displaystyle\int_{\mathbb{R}^{d}}\Upsilon_{A}(x,y)\varphi(y)dy$. \label{rvx}
\end{rmq}

	\begin{prp}
	 There exists $(f^{S}, f^{I},f^{R})\in L_{loc}^{\infty}(\mathbb{R}_{+},(L^{1}(\mathbb{R}^{d})^{3})$ which satisfies:
		\begin{align}
			\displaystyle\partial_{t}f_{t}^{S}(x)&=\mathcal{Q}_{S}^{*}f_{t}^{S}(x)-f_{t}^{S}(x)\int_{\mathbb{R}^{d}}K(x,y)f_{t}^{I}(y)dy, \nonumber\\
			\displaystyle\partial_{t}f_{t}^{I}(x)&=\mathcal{Q}_{I}^{*}f^{I}_{t}(x)+f_{t}^{S}(x)\int_{\mathbb{R}^{d}}K(x,y)f^{I}_{t}(y)dy-\alpha f_{t}^{I}(x),\nonumber\\
			\displaystyle\partial_{t}f^{R}_{t}(x)&=\mathcal{Q}_{R}^{*}f^{R}_{t}(x)+\alpha f^{I}_{t}(x),\label{aqw}
		\end{align}
	where $\mathcal{Q}_{A}^{*}$ is the adjoint operator. \label{gbn}
	\end{prp}

	\begin{proof}
		Let us recall that the initial measures $\mu_{0}^{S},$ $\mu_{0}^{I}$ are absolutly continuous with respect to the Lebesgue measure (see Theorem 3.1 in \cite{sb}).\\
	We construct by inductions a sequences of functions $(f^{n},g^{n},h^{n}),$ satisfying in a weak sense the following system.
		\begin{align}
			\displaystyle\partial_{t}f^{n+1}_{t}(x)&=\mathcal{Q}_{S}^{*}f^{n+1}_{t}(x)-f^{n+1}_{t}(x)\int_{\mathbb{R}^{d}}K(x,y)g_{t}^{n}(y)dy,\nonumber\\
			\displaystyle\partial_{t}g^{n+1}_{t}(x)&=\mathcal{Q}_{I}^{*}g^{n+1}_{t}(x)+f^{n+1}_{t}(x)\int_{\mathbb{R}^{d}}K(x,y)g_{t}^{n}(y)dy-\alpha g_{t}^{n+1}(x),\nonumber\\
			\displaystyle\partial_{t}h^{n+1}_{t}(x)&=\mathcal{Q}_{R}^{*}h^{n+1}_{t}(x)+\alpha g_{t}^{n+1}(x),\nonumber\\
			f_{0}^{n+1}(x)&=f_{0}^{S}(x),\hspace*{0.1cm} g_{0}^{n+1}(x)=f_{0}^{I}(x),\hspace*{0.1cm} h_{0}^{n+1}(x)=0. \label{im}
		\end{align}
		Thanks to the nonnegativity of $f_{0}^{S}$  and applying the Feyman-Kac formula,  we show that  $f^{n}$ is nonnegtive. The  non-negativity of   $g^{n}$, follows by recurrence and by using the comparaison principle, the Feyman Kac formula and the fact that $f_{0}^{I}$  is non-negative. Finaly  $h_{n}$ is non-negative by using the comparaison principle, the Feyman Kac formula.\\From system (\ref{im}),  it is easy to see that for any  $\varphi\in C_{c}(\mathbb{R}^{d})$,\vspace*{0.2cm}\\
		$  \displaystyle (f^{n+1}_{t},\varphi)=\int_{\mathbb{R}^{d}}f_{0}^{S}(x)\int_{\mathbb{R}^{d}}\Upsilon_{S}(t)(x,y)\varphi(y)dydx\\\hspace*{2cm}-\int_{0}^{t}\int_{\mathbb{R}^{d}}f^{n+1}_{r}(x)\int_{\mathbb{R}^{d}}K(x,u)g_{r}^{n}(u)du\int_{\mathbb{R}^{d}}\Upsilon_{S}(t-r)(x,y)\varphi(y)dydxdr,$
		\vspace*{0.2cm}\\$  \displaystyle (g^{n+1}_{t},\varphi)=\int_{\mathbb{R}^{d}}f_{0}^{I}(x)\int_{\mathbb{R}^{d}}\Upsilon_{I}(t)(x,y)\varphi(y)dydx\\\hspace*{2cm}+\int_{0}^{t}\int_{\mathbb{R}^{d}}f^{n+1}_{r}(x)\int_{\mathbb{R}^{d}}K(x,u)g_{r}^{n}(u)du\int_{\mathbb{R}^{d}}\Upsilon_{I}(t-r)(x,y)\varphi(y)dydxdr\vspace*{0.2cm}\\\hspace*{2cm}-\alpha\int_{0}^{t}\int_{\mathbb{R}^{d}}g_{r}^{n+1}(x)\int_{\mathbb{R}^{d}}\Upsilon_{I}(t-r)(x,y)\varphi(y)dydxdr,$
		\vspace*{0.2cm}\\$(h^{n+1}_{t},\varphi)=\displaystyle\alpha\int_{0}^{t}\int_{\mathbb{R}^{d}}g_{r}^{n+1}(x)\int_{\mathbb{R}^{d}}\Upsilon_{R}(t-r)(x,y)\varphi(y)dydxdr.$\vspace*{0.2cm}\\  Fubini's theorem yields, 
		\begin{align}
			\displaystyle f^{n+1}_{t}(y)&=\int_{\mathbb{R}^{d}}f_{0}^{S}(x)\Upsilon_{S}(t)(x,y)dx-\int_{0}^{t}\int_{\mathbb{R}^{d}}f^{n+1}_{r}(x)\int_{\mathbb{R}^{d}}K(x,u)g_{r}^{n}(u)du\Upsilon_{S}(t-r)(x,y)dxdr,\label{tym}\\
			\displaystyle g^{n+1}_{t}(y)&=\int_{\mathbb{R}^{d}}f_{0}^{I}(x)\Upsilon_{I}(t)(x,y)dx+\int_{0}^{t}\int_{\mathbb{R}^{d}}f^{n+1}_{r}(x)\int_{\mathbb{R}^{d}}K(x,u)g_{r}^{n}(u)du\Upsilon_{I}(t-r)(x,y)dxdr\nonumber\\&-\alpha\int_{0}^{t}\int_{\mathbb{R}^{d}}g_{r}^{n+1}(x)\Upsilon_{I}(t-r)(x,y)dxdr,\label{hp}\\
			h^{n+1}_{t}(y)&=\displaystyle\alpha\int_{0}^{t}\int_{\mathbb{R}^{d}}g_{r}^{n+1}(x)\Upsilon_{R}(t-r)(x,y)dxdr.\label{lmo}
		\end{align}
		Fisrt of all, since $\forall n \in \mathbb{N}$, $f^{n}_{t}\geq 0$ and $g^{n}_{t}\geq 0$,  $\displaystyle f^{n+1}_{t}(y)\leq\int_{\mathbb{R}^{d}}f_{0}^{S}(x)\Upsilon_{S}(t)(x,y)dx.$\vspace*{0.2cm}\\Thus integrating over $y\in\mathbb{R}^{d}$, using Fubini's Theorem and  the fact that $\displaystyle \int_{\mathbb{R}^{d}}\Upsilon_{S}(t)(x,y)dy=1$, $\forall t\geq 0,$ we see that 
		
		\begin{align}
			\underset{n}{\sup}\underset{0\leq t\leq T}{\sup}\lVert f_{t}^{n}\lVert_{L^{1}}&\leq \lVert f_{0}^{S}\lVert_{L^{1}}\leq 1.\label{ml}
		\end{align}
		The last inequality follows from the fact that if $H(x)$ is the density of the law of $X_{0}^{1},$\\ $f_{0}^{S}(x)=\{(1-p)1_{A}(x)+1_{A^{c}}(x)\}H(x)$ (see Theorem 3
		1 in \cite{sb}), where $A$ and $p$ have been defined in the  introduction.\\\\
		Moreover, since $\forall n \in \mathbb{N}$, $g^{n}_{t}\geq 0$, \vspace*{0.2cm}\\$\displaystyle g^{n+1}_{t}(y)\leq\int_{\mathbb{R}^{d}}f_{0}^{I}(x)\Upsilon_{I}(t)(x,y)dx+ \int_{0}^{t}\int_{\mathbb{R}^{d}}f^{n+1}_{r}(x)\int_{\mathbb{R}^{d}}K(x,u)g_{r}^{n}(u)du\Upsilon_{I}(t-r)(x,y)dxdr.$\vspace*{0.2cm}\\
		Thus integrating over $y\in\mathbb{R}^{d}$, using (\ref{ml}), Fubini's Theorem and Gronwall's Lemma, we easily deduce that
		\begin{align}
			\underset{n}{\sup}\underset{0\leq t\leq T}{\sup}\lVert g_{t}^{n}\lVert_{L^{1}}&\leq \lVert f_{0}^{I}\lVert_{L^{1}} \textrm{exp}(T\lVert K\lVert_{\infty} )\leq \textrm{exp}(T\lVert K\lVert_{\infty}),\label{mml}
		\end{align}
		where the  last inequality follows from the fact that if $H(x)$ is the density of the law of $X_{0}^{1},$\\ $f_{0}^{I}(x)=p1_{A}(x)H(x).$\\
		On the other hand, with the same argument as above, we deduce from (\ref{lmo}) and (\ref{mml}) that 
		\begin{align}
			\underset{n}{\sup}\underset{0\leq t\leq T}{\sup} \lVert h_{t}^{n}\lVert_{L^{1}}\leq \alpha T\textrm{exp}(T\lVert K\lVert_{\infty}).\label{mpm}
		\end{align}
		
		Let us now show  the convergence of the sequence $(f^{n},g^{n},h^{n})$ in $L^{\infty}_{loc}(\mathbb{R}_{+},[L^{1}(\mathbb{R}^{d})]^{3}).$\vspace*{0.2cm}\\
		A straightforward computation using (\ref{tym}),  and similar arguments as above yields \vspace*{0.2cm}\\
		$ \displaystyle f^{n+1}_{t}(y)-f^{n}_{t}(y)=-\int_{0}^{t}\int_{\mathbb{R}^{d}}(f^{n+1}_{r}(x)-f^{n}_{r}(x))\int_{\mathbb{R}^{d}}K(x,u)g_{r}^{n}(u)du\Upsilon_{S}(t-r)(x,y)dxdr\vspace*{0.2cm}\\\hspace*{3cm}+\int_{0}^{t}\int_{\mathbb{R}^{d}}f^{n}_{r}(x)\int_{\mathbb{R}^{d}}K(x,u)(g_{r}^{n-1}(u)-g_{r}^{n}(u))du\Upsilon_{S}(t-r)(x,y)dxdr$\vspace*{0.2cm}\\
		$\lVert f^{n+1}_{t}-f^{n}_{t} \lVert_{L^{1}}\leq \lVert K\lVert_{\infty}\{\underset{n}{\sup}\underset{0\leq t\leq T}{\sup} \lVert f_{t}^{n}\lVert_{L^{1}}+\underset{n}{\sup}\underset{0\leq t\leq T}{\sup} \lVert g_{t}^{n}\lVert_{L^{1}}\}\\\hspace*{8cm}\times\displaystyle\int_{0}^{t}\{\lVert f^{n+1}_{r}-f^{n}_{r} \lVert_{L^{1}}+\lVert g^{n}_{r}-g^{n-1}_{r} \lVert_{L^{1}}\}dr$ 
		\begin{align}
			&\leq C(T)\lVert K\lVert_{\infty}\int_{0}^{t}\{\lVert f^{n+1}_{r}-f^{n}_{r} \lVert_{L^{1}}+\lVert g^{n}_{r}-g^{n-1}_{r} \lVert_{L^{1}}\}dr.\label{fgkf}
		\end{align}

		Similarly, we have
		
		\begin{align}
			\lVert g^{n+1}_{t}-g^{n}_{t} \lVert_{L^{1}}\leq C(T)(1+\alpha)\lVert K\lVert_{\infty}\int_{0}^{t}\{\lVert g^{n}_{r}-g^{n-1}_{r} \lVert_{L^{1}}+\lVert g^{n+1}_{r}-g^{n}_{r} \lVert_{L^{1}}+\lVert f^{n+1}_{r}-f^{n}_{r} \lVert_{L^{1}}\}dr, \label{rtr}	
		\end{align}
		\begin{align}
			\lVert h^{n+1}_{t}-h^{n}_{t} \lVert_{L^{1}}\leq \alpha\int_{0}^{t}\lVert g^{n+1}_{r}-g^{n}_{r} \lVert_{L^{1}}dr. \label{rtrr}
		\end{align}

		Summing (\ref{fgkf}), (\ref{rtr}) and (\ref{rtrr}) and using Gronwall's Lemma, we have 
		\vspace*{0.2cm}\\$ \underset{0\leq r\leq t}{\sup}\Big\{\lVert f^{n+1}_{r}-f^{n}_{r} \lVert_{L^{1}}+\lVert g^{n+1}_{r}-g^{n}_{r} \lVert_{L^{1}}+\lVert h^{n+1}_{r}-h^{n}_{r} \lVert_{L^{1}}\Big\}$\vspace*{0.2cm}\\\hspace*{5cm}$\leq C(t)\displaystyle \int_{0}^{t} \underset{0\leq u\leq r}{\sup} \Big\{ \lVert f^{n}_{u}-f^{n-1}_{u} \lVert_{L^{1}} +\lVert g^{n}_{u}-g^{n-1}_{u} \lVert_{L^{1}}+\lVert h^{n}_{r}-h^{n-1}_{u} \lVert_{L^{1}}\Big\}dr,$ \vspace*{0.2cm}\\Picard's Lemma then yields \vspace*{0.2cm}\\\hspace*{1.5cm}$\displaystyle\sum\limits_{n} \underset{0\leq t\leq T}{\sup}\Big\{\lVert f^{n+1}_{t}-f^{n}_{t} \lVert_{L^{1}}+\lVert g^{n+1}_{t}-g^{n}_{t} \lVert_{L^{1}}+\lVert h^{n+1}_{t}-h^{n}_{t} \lVert_{L^{1}}\Big\} <\infty $, for any $T>0.$\vspace*{0.2cm}\\Therefore the sequence $(f^{n}, g^{n},h^{n})_{n}$ converge in $L^{\infty}_{loc}(\mathbb{R}_{+},(L^{1})^{3})$ towards $(f^{S},f^{I},f^{R})$ which satisfies by (\ref{ml}), (\ref{mml}) and (\ref{mpm})  respectively\vspace*{0.2cm}\\
		$\hspace*{1cm} \underset{0\leq t\leq T}{\sup}\lVert f_{t}^{S}\lVert_{L^{1}}\leq  1$,\hspace*{0.1cm} $\underset{0\leq t\leq T}{\sup}\lVert f^{I}_{t}\lVert_{L^{1}}\leq \textrm{exp}(T\lVert K\lVert_{\infty})$ and  $\underset{0\leq t\leq T}{\sup}\lVert f^{R}_{t}\lVert_{L^{1}}\leq T\alpha \textrm{exp}(T\lVert K\lVert_{\infty}\lVert).$\vspace*{0.2cm}\\Moreover it is easy to see that  $(f^{S},f^{I},f^{R})$ satisfies in the weak sense the  system (\ref{aqw}).
	\end{proof}

Since $(\mu_{t}^{S},\mu_{t}^{I},\mu_{t}^{R})$ satisfy (\ref{e6}), (\ref{e7}) and (\ref{e8}), from Proposition \ref{gbn} and \ref{ukkk}, we deduce the following result.
\begin{cor}
	For any $t\geq 0$, the measure $\mu_{t}^{S},$ $\mu_{t}^{I}$ and $\mu_{t}^{R}$ are aboslutely continuous with respect to the Lebesgue measure. Their density $(f^{S}, f^{I},f^{R})\in L_{loc}^{\infty}(\mathbb{R}_{+},(L^{1}(\mathbb{R}^{d})^{3}).$
\end{cor}
\begin{rmq}
	From the system of equations in Theorem \ref{th1} above, one can also prove that the measure $\mu_{t}^{S},$ $\mu_{t}^{I}$ and $\mu_{t}^{R}$ are aboslutely continuous with respect to the Lebesgue measure, using this time  assumption  (H1) and the Feyman-kac formula and the fact that the law of the markovian process having $\mathcal{Q}_{S}^{*}$ or $\mathcal{Q}_{I}^{*}$ or $\mathcal{Q}_{R}^{*}$ as infinitesimal generator is absolutely continuous with respect to the Lebesgue measure. 
\end{rmq}

 \section{Central Limit Theorem}
In this section, we  will study the convergence  of the sequence of fluctuations processes\\
\[(U^{N}=\sqrt{N}(\mu^{S,N}-\mu^{S}),V^{N}=\sqrt{N}(\mu^{I,N}-\mu^{I}), W^{N}=\sqrt{N}(\mu^{R,N}-\mu^{R})),\] as $N\rightarrow \infty $, under the assuption (H2) below. 
Note that the trajectories of these processes belong to  $ (\mathbb{D}(\mathbb{R}_{+},\mathcal{E}(\mathbb{R}^{d})))^{3} $,   
where $ \mathcal{E}(\mathbb{R}^{d}) $ is the space of signed measures on $\mathbb{R}^{d}$. However, since the limit processes may be less regular than their approximations  we will first:
\begin{itemize}
\item Formulated the equations verified by   $ (U^{N},V^{N},W^{N}) .$
\item Fix the space in which the convergence results will be established.
\end{itemize}
Then we will study the convergence of the above sequence.\\\\
Letting $D=\lceil d/2\rceil$ (where $\lceil .\rceil$ is the upper integer part), the following is supposed to hold throughout this section.\\\\ \textbf{{\color{blue}Assumption (H2)}}: 
\begin{itemize}
\item For any $A\in \{S,I,R\},$ for any $\ell ,u \in \{1,2...d\},$ both functions  $\theta_{\ell,u}(A,.)$ and  $m_{\ell}(A,.)$ belong to $C^{3+D}_{b}(\mathbb{R}^{d})$.
\item $ K\in C^{2+D}_{c}(\mathbb{R}^{d}\times\mathbb{R}^{d}).$
\end{itemize}	
	
	 \subsection{System of evolution equations of the Processes  $ (U^{N},V^{N},W^{N})$ \label{cc1}} 
	Let $ \varphi \in C^{2}_{b}(\mathbb{R}^{d}),$  we have\vspace*{0.12cm}\\$\hspace*{0.cm}\displaystyle(\mu_{t}^{S,N},\varphi)=(\mu_{0}^{S,N},\varphi)+\int_{0}^{t}(\mu_{r}^{S,N}, \mathcal{Q}_{S}\varphi) dr-  \int_{0}^{t}\left(\mu_{r}^{S,N}, \varphi (\mu_{r}^{I,N}, K)\right) dr + M_{t}^{N,\varphi},$ \vspace*{0.13cm}\\
	$\hspace*{0.1cm} (\mu_{t}^{S},\varphi)=\displaystyle (\mu_{0}^{S},\varphi)+\int_{0}^{t}(\mu_{r}^{S},\mathcal{Q}_{S}\varphi) dr - \int_{0}^{t}\left(\mu_{r}^{S}, \varphi (\mu_{r}^{I}, K)\right) dr. $\vspace*{0.1cm}\\
	Thus
	\begin{align*}
		(U_{t}^{N},\varphi)&=\displaystyle(U_{0}^{N},\varphi) + \int_{0}^{t}(U_{r}^{N}, \mathcal{Q}_{S}\varphi) dr-  \int_{0}^{t} (U_{r}^{N}, \varphi(\mu^{I,N}_{r},K) ) dr  -  \int_{0}^{t} (\mu_{r}^{S}, \varphi(V^{N}_{r},K) ) dr + \sqrt{N}M_{t}^{N,\varphi}.
\end{align*}
	Hence  letting  $\widetilde{M}_{t}^{N,\varphi}=\sqrt{N}M_{t}^{N,\varphi}$, we  have
	\begin{align}
(U_{t}^{N},\varphi)=\displaystyle (U_{0}^{N},\varphi) + \int_{0}^{t}(U_{r}^{N},\mathcal{Q}_{S} \varphi) dr-   \int_{0}^{t} (U_{r}^{N}, G_{r}^{I,N}\varphi ) dr- \int_{0}^{t}  (V^{N}_{r},G_{r}^{S}\varphi ) dr   +\widetilde{M}_{t}^{N,\varphi},\label{c1}
	\end{align}
	and also 
	\begin{align}
	 \displaystyle (V_{t}^{N},\varphi)&=\displaystyle (V_{0}^{N},\varphi) + \int_{0}^{t}(V_{r}^{N},\mathcal{Q}_{I}\varphi) dr+  \int_{0}^{t} (U_{r}^{N}, G_{r}^{I,N}\varphi ) dr+ \int_{0}^{t} (V^{N}_{r},G_{r}^{S}\varphi ) dr  -\alpha\int_{0}^{t}(V_{r}^{N},\varphi)dr \nonumber\\&+\widetilde{L}_{t}^{N,\varphi},\label{c2}
\end{align}
	and 	
	\begin{align}	
 \displaystyle (W_{t}^{N},\varphi)=  \int_{0}^{t}(W_{r}^{N},\mathcal{Q}_{R} \varphi) dr +\alpha\int_{0}^{t}(V_{r}^{N},\varphi)dr+\widetilde{Y}_{t}^{N,\varphi}.\label{c3}
 	\end{align}
 
  Where $\forall x,y \in \mathbb{R}^{d}$,\\ \hspace*{3.2cm} $G_{r}^{I,N}\varphi(x)= \varphi(x)(\mu_{r}^{I,N},  K(x,.))=\displaystyle\varphi(x) \int_{\mathbb{R}^{d}}K(x,y)\mu_{r}^{I,N}(dy),$  \vspace*{0.1cm} \\\hspace*{3.2cm} $G_{r}^{S}\varphi(y)= (\mu_{r}^{S}, \varphi K(.,y))=\displaystyle\int_{\mathbb{R}^{d}}\varphi(x) K(x,y)\mu_{r}^{S}(dx).$
	\subsection{The  Space of convergence of the sequences $(U^{N},V^{N},W^{N})$}
 Throughout this Subsection $\sigma$ is an arbitrary positive real number and the following is assumed to hold:
	\vspace*{0.1cm}\\ \textbf{{\color{blue}Assumption (H3)}}:   $\mathbb{E}(\lvert X^{1}_{0}\lvert^{2\sigma})<\infty$.\\\\
	The next lemma follows easilly from the definition of $X_{t}^{i}$ (see (\ref{uith})), the fact that the functions $m$ and $\theta$ are bounded and the inequality of Burkholder-Davis-Gundy.
	\begin{lem} Under the assumption (H3),
for any   $T>0,$ there exists $C(T)>0$ such that,
		\begin{center}
		$\underset{1\leq i\leq N}{\sup}\mathbb{E}(\underset{0\leq t \leq T}{\sup}\lvert X^{i}_{t}\lvert^{2\sigma})<C(T).$ 
			\end{center}
		\label{lt1}
	\end{lem}
\begin{cor} Under the assumption (H3), for any  $T>0$,  there exists $C(T)>0,$ such that 
	\begin{center}
	$\underset{N\geq1}{\sup}\mathbb{E}(\underset{0\leq t \leq T}{\sup}(\mu^{S,N}_{t}, \lvert .\lvert^{2\sigma}))<C(T),$ and   $\underset{0\leq t \leq T}{\sup}(\mu^{S}_{t}, \lvert .\lvert^{2\sigma})<C(T).$
\end{center}
\label{cor32}
\end{cor}

	\begin{lem}
		For every fixed $y\in \mathbb{R^{d}}$, $\ell\in \{1,2......d\},$   the mapping $\delta_{y},\mathcal{P}_{y}^{\ell}:H^{s,\sigma}\rightarrow \mathbb{R}$ defined by 
		 $\delta_{y}\varphi =\varphi(y)$ and $ \mathcal{P}_{y}^{\ell}\varphi=\frac{\partial}{\partial 
		 	y_{\ell}}\varphi(y)$ are continuous for $s>d/2$ and $s>1+d/2$ respectively and 
		 \begin{align}
		 \lVert \delta_{y}\lVert_{-s,\sigma}&\leq C(1+\lvert y\lvert^{\sigma})\quad \textrm{if s>d/2}, \label{rr1}\\
		 \lVert \mathcal{P}_{y}^{\ell}\lVert_{-s,\sigma}&\leq C(1+\lvert y\lvert^{\sigma}) \quad \textrm{if s>d/2+1}. \label{r2}
		   \end{align}
	   \label{rrrr}
	\end{lem}
\begin{proof}Continuity follows easily from Sobolev injections.
	On the other hand for every function $\varphi \in H^{s,\sigma}(\mathbb{R}^{d})$ one deduce from $(\ref{em1})$ (see subsection \ref{sec3}) that: \begin{center}
	$\lvert \varphi(y)\lvert\leq (1+\lvert y\lvert^{\sigma})\lVert \varphi \lVert_{C^{0,\sigma}} \leq C(1+\lvert y\lvert^{\sigma})\lVert \varphi \lVert_{s,\sigma}.$
\end{center} 
The inequality  	$(\ref{r2})$ is proved in a similar way.
\end{proof}
\begin{cor}
	If $(\varphi_{p})_{p\geq 1}$ is a complete orthonormal basis in $H^{s,\sigma}(\mathbb{R}^{d}),$ we have
	
	\begin{center}
	$\lVert \delta_{y}\lVert_{-s,\sigma}^{2}=\sum\limits_{p\geq 1} (\varphi_{p}(y))^{2} \leq C(1+\lvert y\lvert^{2\sigma})$, \quad if $s>d/2$ \\
	$\lVert \mathcal{P}^{\ell}_{y}\lVert_{-s,\sigma}^{2}=\sum\limits_{p\geq 1} (\frac{\partial}{\partial y_{\ell}}\varphi_{p}(y))^{2} \leq C(1+\lvert y\lvert^{2\sigma}),$ \quad if $s>d/2+1$.
\end{center}
$\label{r3}$
\end{cor}
	  \begin{prp} Under the assumption (H3),		every limit point $\mathcal{M}^{1}$ of the seguence $(\widetilde{M}^{N})_{N\geq 1}$ satisfies,\vspace*{0.1cm}\\
		\hspace*{3.3cm} $
		\forall T\geq 0,\hspace*{0.2cm} \underset{0\leq t\leq T}{\sup}\mathbb{E}(\lVert \mathcal{M}^{1}_{t}\lVert_{-s,\sigma}^{2})<\infty\quad$ if   $ s>1+d/2$.\label{add}
	\end{prp}
	
	\begin{proof}
		We have \\ $\widetilde{M}_{t}^{N,\varphi}= \displaystyle- \frac{1}{\sqrt{N}} \sum_{i=1}^{N} \int_{0}^{t} \int_{0}^{\infty}1_{\{E_{r^{-}}^{i}=S\}}\varphi(X_{r}^{i})1_{\{u\leq \frac{1}{N} \sum_{j=1}^{N} K(X_{r}^{i},X_{r}^{j}) 1_{\{E_{r}^{j}=I\}}\}} \overline{M}^{i}(dr,du) \vspace*{0.12cm} \\\hspace*{1.5cm}+\frac{1}{\sqrt{N}} \displaystyle\sum_{i=1}^{N} \int_{0}^{t}1_{\{E_{r}^{i}=S\}}\bigtriangledown\varphi(X_{r}^{i})\theta(S, X_{r}^{i})dB_{r}^{i},$
	\begin{align*}
		<\widetilde{M}^{N,\varphi}>_{t} &= \displaystyle \int_{0}^{t}\left(\mu_{r}^{S,N}, \varphi^{2} (\mu_{r}^{I,N}, K)\right)dr\\&+  \int_{0}^{t} \Big(\mu_{r}^{S,N}, \sum\limits_{1\leq \ell\leq d}\big(\frac{\partial \varphi}{\partial x_{\ell}}\big)^{2}\sum\limits_{1\leq u\leq d}\theta^{2}_{\ell,u}(S,.)+2\sum\limits_{\underset{1\leq u\leq d}{\underset{1\leq e\leq d}{1\leq \ell\leq d-1}}}\frac{\partial \varphi}{\partial x_{\ell}}\frac{\partial \varphi}{\partial x_{u}}\theta_{\ell,e}(S,.)\theta_{u,e}(S,.)\Big) dr,
	\end{align*}
		and   it follows from  Theorem \ref{th1}, that 	\begin{align*}
			<\widetilde{M}^{N,\varphi}>_{t} &\xrightarrow{P} \displaystyle \int_{0}^{t}\left(\mu_{r}^{S}, \varphi^{2} (\mu_{r}^{I}, K)\right)dr\\&+  \int_{0}^{t} \Big(\mu_{r}^{S}, \sum\limits_{1\leq \ell\leq d}\big(\frac{\partial \varphi}{\partial x_{\ell}}\big)^{2}\sum\limits_{1\leq u\leq d}\theta^{2}_{\ell,u}(S,.)+2\sum\limits_{\underset{1\leq e\leq d}{\underset{\ell+1\leq u\leq d}{1\leq \ell\leq d-1}}}\frac{\partial \varphi}{\partial x_{\ell}}\frac{\partial \varphi}{\partial x_{u}}\theta_{\ell,e}(S,.)\theta_{u,e}(S,.)\Big) dr.
		\end{align*}
		Furthemore  \\\\
		$ \displaystyle \int_{0}^{t}\left(\mu_{r}^{S}, \varphi^{2} (\mu_{r}^{I}, K)\right)dr\\\hspace*{4cm}+  \int_{0}^{t} \Big(\mu_{r}^{S}, \sum\limits_{\underset{1\leq u\leq d}{1\leq \ell\leq d}}\big(\frac{\partial \varphi}{\partial x_{\ell}}\big)^{2}\theta^{2}_{l,u}(S,.)+2\sum\limits_{\underset{1\leq e\leq d}{\underset{\ell+1\leq u\leq d}{1\leq \ell\leq d-1}}}\frac{\partial \varphi}{\partial x_{\ell}}\frac{\partial \varphi}{\partial x_{u}}\theta_{\ell,e}(S,.)\theta_{u,e}(S,.)\Big) dr,$ \\being the quadratic variation of a Gaussian martingale (we refer to Proposition \ref{ffg} below for the Gaussian property ) of the form $(\mathcal{M}^{1},\varphi)$,
		our aim  is to find the smallest value of s for which $\mathbb{E}(\lVert \mathcal{M}_{t}^{1}\lVert_{-s,\sigma}^{2})<\infty$. With again $(\varphi_{p})_{p\geq 1}$ an orthonormal basis of $H^{s,\sigma}(\mathbb{R}^{d}),$ we have\vspace*{0.13cm}\\ 
		$\mathbb{E}(\lVert \mathcal{M}_{t}^{1}\lVert_{-s,\sigma}^{2})=\mathbb{E}(\sum\limits_{p\geq 1}\lvert (\mathcal{M}_{t}^{1},\varphi_{p})\lvert^{2})$
		=$\sum\limits_{p\geq 1 }\mathbb{E}( <(\mathcal{M}^{1},\varphi_{p})>_{t}). $\vspace*{0.1cm}\\From  Corollary \ref{cor32} and  Corollary \ref{r3} and Asumption (H3), we  have \\
		$\sum\limits_{p\geq1}<\mathcal{M}^{1},\varphi_{p}>_{t}=\displaystyle\sum\limits_{p\geq 1}\Big\{\int_{0}^{t}\left(\mu_{r}^{S},\varphi_{p}^{2}
		(\mu_{r}^{I}, K)\right)dr\\\hspace*{2.5cm}+\displaystyle\int_{0}^{t} \Big(\mu_{r}^{S}, \sum\limits_{1\leq \ell\leq d}\big(\frac{\partial \varphi_{p}}{\partial x_{\ell}}\big)^{2}\sum\limits_{1\leq u\leq d}\theta^{2}_{\ell,u}(S,.)+2\sum\limits_{\underset{1\leq e\leq d}{\underset{\ell+1\leq u\leq d}{1\leq \ell\leq d-1}}}\frac{\partial \varphi_{p}}{\partial x_{\ell}}\frac{\partial \varphi_{p}}{\partial x_{u}}\theta_{\ell,e}(S,.)\theta_{u,e}(S,.)\Big) dr\Big\}, $\\
		$\hspace*{1cm}=\displaystyle\int_{0}^{t}\int_{\mathbb{R}^{d}}\left(\int_{\mathbb{R}^{d}}K(x,y)\mu_{r}^{I}(dy)\right)\sum\limits_{p\geq 1}\varphi_{p}^{2}(x)\mu_{r}^{S}(dx)dr\\+\int_{0}^{t}\int_{\mathbb{R}^{d}}\Big\{\sum\limits_{1\leq \ell\leq d}\sum\limits_{1\leq u\leq d}\theta^{2}_{\ell,u}(S,x)\sum\limits_{p\geq 1} \big(\frac{\partial \varphi_{p}}{\partial x_{\ell}}\big)^{2}+2\sum\limits_{\underset{1\leq e\leq d}{\underset{\ell+1\leq u\leq d}{1\leq \ell\leq d-1}}}\theta_{\ell,e}(S,x)\theta_{u,e}(S,x)\sum\limits_{p\geq 1} \frac{\partial \varphi_{p}}{\partial x_{\ell}}(x)\frac{\partial \varphi_{p}}{\partial x_{u}}(x)\Big\}\mu_{r}^{S}(dx)dr,$ \\
		$\hspace*{1cm}\leq\displaystyle\lVert K\lVert_{\infty}\int_{0}^{t}\int_{\mathbb{R}^{d}}\sum\limits_{p\geq 1}\varphi_{p}^{2}(x)\mu_{r}^{S}(dx)dr+\int_{0}^{t}\int_{\mathbb{R}^{d}}\sum\limits_{\underset{1\leq u\leq d}{1\leq \ell\leq d}}\lVert\theta^{2}_{\ell,u}\lVert_{\infty}\sum\limits_{p\geq 1} \big(\frac{\partial \varphi_{p}}{\partial x_{\ell}}\big)^{2}(x)\mu_{r}^{S}(dx)dr\\\hspace*{1cm}+4\sum\limits_{\underset{1\leq e\leq d}{\underset{\ell+1\leq u\leq d}{1\leq \ell\leq d-1}}}\lVert\theta_{\ell,e}\lVert_{\infty}\lVert\theta_{u,e}\lVert_{\infty}\int_{0}^{t}\int_{\mathbb{R}^{d}}\sum\limits_{p\geq 1}\Big\{ \Big(\frac{\partial \varphi_{p}}{\partial x_{\ell}}(x)\Big)^{2}+\Big(\frac{\partial \varphi_{p}}{\partial x_{u}}(x)\Big)^{2}\Big\}\mu_{r}^{S}(dx)dr,$\\$\hspace*{1cm}\leq C \displaystyle\int_{0}^{t}\int_{\mathbb{R}^{d}}(1+\lvert x\lvert^{2\sigma})\mu_{r}^{S}(dx)dr+C(d)\int_{0}^{t}\int_{\mathbb{R}^{d}}(1+\lvert x\lvert^{2\sigma})\mu_{r}^{S}(dx)dr,$\vspace*{0.17cm}\\$\hspace*{1cm}\leq C(d,T).$      
	\end{proof} 
	By  Doob's inequality and by calculations similar to those done above we obtain the following result.
	\begin{cor} Under the asumption (H4),
		for any $ T>0$, $s>1+D$, $ \exists \hspace*{0.1cm}C(T)>0,$   such that: $\vspace*{0.12cm}$\\ $\underset{N\geq 1}{\sup}\mathbb{E}(\underset{0\leq t \leq T}{\sup}\lVert \widetilde{M}_{t}^{N} \lVert_{-s,\sigma}^{2} )\leq C(T),  \hspace{0.1cm} \underset{N\geq 1}{\sup}\mathbb{E}(\underset{0\leq t \leq T}{\sup}\lVert \widetilde{L}_{t}^{N} \lVert_{-s,\sigma}^{2})\leq C(T),\hspace{0.1cm} \underset{N\geq 1}{\sup}\mathbb{E}(\underset{0\leq t \leq T}{\sup}\lVert \widetilde{Y}_{t}^{N} \lVert_{-s,\sigma}^{2} )\leq C(T).$ \label{ct}\\
	\end{cor} 

 \textbf{In the rest of this section    we arbitrarily choose  $\sigma>d/2$  and   $1+D< s <2+D.$ } Furthermore, in all the sequel, the assumption (H3) is supposed to hold  for that value of $\sigma.$ Thus  we will prove that 
the sequences   $(U^{N},V^{N},W^{N})_{N\geq 1}$    converges in   $[\mathbb{D}(\mathbb{R}_{+},H^{-s,\sigma})]^{3},$ where we have equipped  $\mathbb{D}(\mathbb{R}_{+},H^{-s,\sigma})$  with the Skorokhod topology (we refer to \cite{fc}  for the explicit definition of this topology).
\begin{rmq}
	Note that the assumption (H3) and the fact that $\sigma>d/2$, yield:
	\begin{center}
	 $\underset{x}{\sup}\lVert K(x,.) \lVert_{2+D,\sigma}<\infty$ and $\underset{y}{\sup}\lVert K(.,y) \lVert_{2+D,\sigma}<\infty.$
 \end{center}
\end{rmq}
  \subsection{Convergence of $ (U^{N},V^{N},W^{N})_{N\geq 1} $ } 
We  first derive an  estimate of the norm of the fluctuation processes $U^{N}$; $V^{N}$ and  $W^{N},$ which is not uniform in $N$.

\begin{lem}
	For any $N\geq 1,$ $T>0,$ there exists a constant $C(T)>0,$ such that\\\\$\displaystyle  \mathbb{E}(\underset{0\leq t\leq T}{\sup}\Arrowvert U_{t}^{N}\Arrowvert_{-s,\sigma}^{2})\leq C(T)N,$\quad $\displaystyle  \mathbb{E}(\underset{0\leq t\leq T}{\sup}\Arrowvert U_{t}^{N}\Arrowvert_{-s,\sigma}^{2})\leq C(T)N$ and $\displaystyle  \mathbb{E}(\underset{0\leq t\leq T}{\sup}\Arrowvert U_{t}^{N}\Arrowvert_{-s,\sigma}^{2})\leq C(T)N.$ 
\end{lem}

\begin{proof}
	Let us prove the result for $U^{N}.$
	We first recall that $C^{1,\sigma}\hookrightarrow C^{0,2\sigma}.$ Moreover since  $s>1+D$,  $H^{s,\sigma}\hookrightarrow C^{1,\sigma}$. Now  we have
	\begin{align*}
	\lvert(U^{N}_{t},\varphi)\lvert&=\sqrt{N}\Big\lvert\frac{1}{N}\sum\limits_{i=1}^{N}1_{\{E_{t}^{i}=S\}}\varphi(X^{i}_{t})-(\mu_{t}^{S},\varphi)\Big\lvert,\\&\leq\sqrt{N}\Big\{\frac{1}{N}\sum\limits_{i=1}^{N}1_{\{E_{t}^{i}=S\}}(1+\lvert X_{t}^{i}\lvert^{2\sigma})\frac{\lvert\varphi(X^{i}_{t})\lvert}{1+\lvert X_{t}^{i}\lvert^{2\sigma}}+\Big(\mu_{t}^{S},(1+\lvert.\lvert^{2\sigma})\frac{\lvert\varphi(.)\lvert}{1+\lvert .\lvert^{2\sigma}}\Big)\Big\},\\&\leq\sqrt{N}\lVert \varphi\lVert_{C^{0,2\sigma}}\Big\{\Big(\mu_{t}^{S,N}, 1+\lvert .\lvert^{2\sigma}\Big)+\Big(\mu_{t}^{S}, 1+\lvert.\lvert^{2\sigma}\Big)\Big\},\\&\leq\sqrt{N}\lVert \varphi\lVert_{s,\sigma}\Big\{\Big(\mu_{t}^{S,N}, 1+\lvert .\lvert^{2\sigma}\Big)+\Big(\mu_{t}^{S}, 1+\lvert.\lvert^{2\sigma}\Big)\Big\}. 
\end{align*}
This inequality combined  with Corollary \ref{cor32} and the fact that 
$	\lVert U^{N}_{t}\lVert_{-s,\sigma}=\underset{\varphi\neq 0, \varphi\in H^{s,\sigma}}{\sup}\frac{\lvert (U^{N}_{t},\varphi)\lvert}{\lVert \varphi \lVert_{s,\sigma}} $ yields \begin{center}

	$\displaystyle  \mathbb{E}(\underset{0\leq t \leq T}{\sup}\Arrowvert U_{t}^{N}\Arrowvert_{-s,\sigma}^{2})\leq 4C(T)N.$
\end{center}

	By the same arguments we obtain the same results for  $V^{N}$ and $W^{N}.$
\end{proof}
We now give the estimates for the fluctuations at time $0$. It is uniform in $N.$
\begin{prp}
	For any $s>d/2$, there exists   $C $ such that 
	\begin{center}
		$ \quad  \underset{N\geq 1}{\sup}\mathbb{E}(\lVert U_{0}^{N} \lVert_{-s,\sigma}^{2})<C $ and $\underset{N\geq 1}{\sup}\mathbb{E}(\lVert V_{0}^{N} \lVert_{-s,\sigma}^{2})<C. $
			\end{center}\label{uo}
\end{prp}
\begin{proof}
	We only prove that  $\underset{N\geq 1}{\sup}\mathbb{E}(\lVert V_{0}^{N} \lVert_{-s,\sigma}^{2})<C$. The other estimate follows by similar arguments. Since $1_{A}(X_{j})\xi_{j}\delta_{X_{j}}$ are i.i.d with law $\mu_{0}^{I}$, from Corollary \ref{r3} and from assumption (H3), if s>d/2, 
	we have  
	\begin{align*}
	 \mathbb{E}(\lVert V_{0}^{N} \lVert_{-s,\sigma}^{2} )&=\mathbb{E}\Big(\sum\limits_{p\geq 1}(V_{0}^{N},\varphi_{p})^{2}\Big),\\&=N\sum\limits_{p\geq 1}\mathbb{E}\left(\left((\mu_{0}^{I,N},\varphi_{p})-(\mu_{0}^{I},\varphi_{p})\right)^{2}\right),\\&=\frac{1}{N}\sum\limits_{i,n_{1},n_{2}}\mathbb{E}\left(\left[\sum\limits_{j=1}^{N}[1_{A}(X_{j})\xi_{j}\varphi_{p}(X_{j})-(\mu_{0}^{I},\varphi_{p})]\right]^{2}\right),\\&=\frac{1}{N}\sum\limits_{p\geq 1}\sum\limits_{j=1}^{N}\mathbb{E}\left(\left[1_{A}(X_{j})\xi_{j}\varphi_{p}(X_{j})-(\mu_{0}^{I},\varphi_{p})\right]^{2}\right)\\&\leq\frac{1}{N}\sum\limits_{p\geq 1}\sum\limits_{j=1}^{N}\mathbb{E}\left([1_{A}(X_{j})\xi_{j}\varphi_{p}(X_{j})]^{2}\right),\\&\leq p\displaystyle\int_{A}\sum\limits_{p\geq 1}\varphi_{p}^{2}(x)d\mathbb{P}_{X_{0}^{1}}(x),  \\&\leq pC\displaystyle\int_{A}(1+\lvert x\lvert^{2\sigma})d\mathbb{P}_{X_{0}^{1}}(x)\leq  C.  
	\end{align*}
\end{proof}
\begin{rmq}
	Using Proposition \ref{uo} and following the Proof of Theorem 3.3 in \cite{sb}, we prove easily that for any $s>d/2,$ the sequence $(U_{0}^{N},V_{0}^{N})_{N}$ converges in law in $H^{-s,\sigma}(\mathbb{R}^{d})$ towards $(U_{0},V_{0})$, where for any $\varphi,\psi\in H^{s,\sigma}, $  the expression of the Gaussian vector  $((U_{0},\varphi),(V_{0},\psi))$ is given  by 
	\begin{align}
		(U_{0},\varphi)&=W_{1}[\varphi\sqrt{g}\{(1-p)1_{A}+1_{A^{c}}\}-(1-p)W_{1}(\sqrt{g})\int_{A}\varphi(x)g(x)dx-W_{1}(\sqrt{g})\int_{A^{c}}\varphi(x)g(x)dx\nonumber\\
		&\qquad+W_{2}(1_{A}\varphi\sqrt{(p-p^{2})g}),\label{u1}\\
		(V_{0},\psi)&=pW_{1}(1_{A}\psi\sqrt{g})-pW_{1}(\sqrt{g})\int_{A}\psi(x)g(x)dx-W_{2}(1_{A}\psi\sqrt{(p-p^{2})g}),\label{u2}\\
		(Z_{0},\phi)&=W_{1}(\phi\sqrt{g})-W_{1}(\sqrt{g})\left(\int_{\mathbb{R}^{d}}\phi(x)g(x)dx\right),\label{u3}
	\end{align} 
	where   $g$ is the density of the law of $X_{0}^{1}$ and  $W_{1}$ and $W_{2}$ are mutually independent  two dimensional white noises.\label{lo}\label{rm59}
\end{rmq}
The proof of the next Lemma can be found in \cite{sb} (see the proof of Lemma 5.15 in \cite{sb}, using this time Corollary \ref{co42} below).
\begin{lem}
	For each $N\geq 1$, the processes $U^{N}$,  $V^{N}$  and  $W^{N}$  belong  to $\mathbb{D}(\mathbb{R}_{+},H^{-s,\sigma}).$ 
\end{lem}
Let us give the main result of this section.                       
\begin{thm}
	Under  (H2) and (H3), the sequence of processes $(U^{N},V^{N},W^{N})_{N\geq1}$  converges in law  in    $(\mathbb{D}(\mathbb{R}_{+},H^{-s,\sigma}))^{3}$ towards the  process $(U,V,W)$ whose trajectories belong to  $  (\mathbb{C}(\mathbb{R}_{+},H^{-s,\sigma}))^{3}$ 
	and which satisfies for any $ t\geq 0,$
	\begin{align*}
		U_{t}&=\displaystyle U_{0} + \int_{0}^{t} \mathcal{Q}_{S}^{*} U_{r}dr - \int_{0}^{t} (G_{r}^{I})^{*}U_{r} dr  -\int_{0}^{t} (G_{r}^{S})^{*}V_{r}dr  +\mathcal{M}_{t}^{1},\vspace*{0.1cm} \\
		V_{t}&=\displaystyle V_{0} + \int_{0}^{t} \mathcal{Q}_{I}^{*}V_{r} dr +  \int_{0}^{t} (G_{r}^{I})^{*}U_{r} dr +  \int_{0}^{t} ((G_{r}^{S})^{*}-\alpha I_{d}) V_{r}dr  + \mathcal{M}_{t}^{2},\\
		W_{t}&=\displaystyle \int_{0}^{t}\mathcal{Q}_{R}^{*} W_{r} dr+\alpha\int_{0}^{t} V_{r}dr  + \mathcal{M}_{t}^{3},
	\end{align*}
	where  for any $r>0,$ $G_{r}^{I}$ and $G_{r}^{S}$
are defined as in subsection \ref{cc1} (replacing $\mu_{r}^{S,N}$ and $\mu_{r}^{I,N}$ by $\mu_{r}^{S}$ and $\mu_{r}^{I}$ respectively ), and $\forall \varphi,\psi,\phi \in H^{s,\sigma},( \mathcal{M}^{1},\varphi) ,(\mathcal{M}^{2},\psi), (\mathcal{M}^{3},\phi)) $ is a centered Gaussian martingale   satisfying: 

		\begin{align*}
		<(\mathcal{M}^{1},\varphi)>_{t} &= \displaystyle \int_{0}^{t}\left(\mu_{r}^{S}, \varphi^{2} (\mu_{r}^{I}, K)\right)dr\\&+  \int_{0}^{t} \Big(\mu_{r}^{S}, \sum\limits_{1\leq \ell\leq d}\big(\frac{\partial \varphi}{\partial x_{\ell}}\big)^{2}\sum\limits_{1\leq u\leq d}\theta^{2}_{\ell,u}(S,.)+2\sum\limits_{\underset{1\leq e\leq d}{\underset{\ell+1\leq u\leq d}{1\leq \ell\leq d-1}}}\frac{\partial \varphi}{\partial x_{\ell}}\frac{\partial \varphi}{\partial x_{u}}\theta_{\ell,e}(S,.)\theta_{u,e}(S,.)\Big) dr,
		\end{align*}
	\begin{align*}
		<(\mathcal{M}^{2},\psi)>_{t} &= \displaystyle \int_{0}^{t}\left(\mu_{r}^{S}, \psi^{2} (\mu_{r}^{I}, K)\right)dr\\&+  \int_{0}^{t} \Big(\mu_{r}^{I}, \sum\limits_{1\leq \ell\leq d}\big(\frac{\partial \psi}{\partial x_{\ell}}\big)^{2}\sum\limits_{1\leq u\leq d}\theta^{2}_{\ell,u}(I,.)+2\sum\limits_{\underset{1\leq e\leq d}{\underset{\ell+1\leq u\leq d}{1\leq \ell\leq d-1}}}\frac{\partial \psi}{\partial x_{\ell}}\frac{\partial \psi}{\partial x_{u}}\theta_{\ell,e}(I,.)\theta_{u,e}(I,.)\Big) dr,
	\end{align*}
	\begin{align*}
	<(\mathcal{M}^{3},\phi)>_{t} &=\alpha \int_{0}^{t}(\mu_{r}^{R},\phi^{2})dr\\&+   \int_{0}^{t} \Big(\mu_{r}^{I}, \sum\limits_{1\leq \ell\leq d}\big(\frac{\partial \phi}{\partial x_{\ell}}\big)^{2}\sum\limits_{1\leq u\leq d}\theta^{2}_{\ell,u}(I,.)+2\sum\limits_{\underset{1\leq e\leq d}{\underset{\ell+1\leq u\leq d}{1\leq \ell\leq d-1}}}\frac{\partial \phi}{\partial x_{\ell}}\frac{\partial \phi}{\partial x_{u}}\theta_{\ell,e}(I,.)\theta_{u,e}(I,.)\Big) dr,
\end{align*}
\begin{align*}
		<(\mathcal{M}^{1},\varphi),(\mathcal{M}^{2},\psi)>_{t}&=\displaystyle-\int_{0}^{t}\left(\mu_{r}^{S},\varphi\psi(\mu_{r}^{I}, K)\right)dr,\\
		< (\mathcal{M}^{2},\psi),(\mathcal{M}^{3},\phi)>_{t}&=\displaystyle\alpha\int_{0}^{t}(\mu_{r}^{I},\psi\phi)dr, \hspace*{0.1cm}  and \hspace*{0.1cm} < (\mathcal{M}^{1},\varphi),(\mathcal{M}^{3},\phi)>_{t}=0.
	\end{align*} 
	 \label{c5} 
\end{thm}
Before we prove this Theorem we first state a condition of Aldous type for the tightness of a sequence of $H^{-s,\sigma}$-valued c$\grave{a}$dl$\grave{a}$g processes, exploiting the fact that $H^{-s,\sigma}$ is a Hilbert space (see Definition 2.2.1 of \cite{rc}).
\begin{prp}
	Let  $ (\vartheta^{n})_{n} $ be  a sequence of $H^{-s,\sigma}$-valued c$\grave{a}$dl$\grave{a}$g processes, their laws ($\widetilde{P}^{n}$)  form a tight sequence in $\mathbb{D}(\mathbb{R}_{+},H^{-s,\sigma})$  if: \vspace*{0.15cm}\\
	$(T_{1}) \quad  $ For each t in a   dense subset $\mathbb{T}$ of $\mathbb{R}_{+}$, the sequence  $(\vartheta_{t}^{n})_{n}$ is tight in $H^{-s,\sigma}.$\vspace*{0.16cm}\\
	$(T_{2}) \quad$ For each  $T>0, \forall \varepsilon_{1},\varepsilon_{2}>0,$  there exist $\delta>0, n_{0}\geq 1$ such that for any collection of stopping times $\tau^{n} \leq T, \vspace*{0.13cm}\\\hspace*{5cm}\underset{\underset{\varrho\leq\delta}{n\geq n_{0}}}{\sup} $ $\mathbb{P}(\lVert \vartheta^{n}_{(\tau^{n}+\varrho) }-\vartheta^{n}_{\tau^{n}}\lVert_{H}>\varepsilon_{1})\leq\varepsilon_{2}. $   \label{c6}
\end{prp}
The proof of  Theorem \ref{c5} is the content of subsection \ref{tqa} below, however let us first prove a
few preliminary results.
\subsubsection{Preliminary results}
\begin{prp}
	The sequences $(\widetilde{M}^{N})_{N\geq 1}$,  $(\widetilde{L}^{N})_{N\geq 1} $ and $(\widetilde{Y}^{N})_{N\geq 1} $ are  tight in  $ \mathbb{D}(\mathbb{R}_{+},H^{-s,\sigma}).$ \label{c8}
\end{prp}
\begin{proof}
	$-$ Tightness of  $ (\widetilde{M}^{N})_{N\geq1}.$\\\\
	Let us prove that $(\widetilde{M}^{N} )_{N\geq1}$ satisfies the conditions  (T1) and  (T2) of  Proposition \ref{c6}.\vspace*{0.1cm}\\
	$-$ To show (T1)  it  is enough to prove that:\vspace*{0.18cm}\\\hspace*{1cm}$\forall t\geq 0 $, $\forall \varepsilon >0$ there exists a compact subset $\mathcal{K}$ of $H^{-s,\sigma}$ such that $\mathbb{P}(\widetilde{M}_{t}^{N}\notin \mathcal{K})<\varepsilon$.\vspace*{0.1cm}\\ This follows readily from the fact that for each  $1+D<s'=\frac{s+1+D}{2}<s,$ $\sigma'>\sigma>d/2,$ there exists $C(T)$ such that\vspace*{0.12cm} \\\hspace*{5cm} $\mathbb{E}(\underset{0\leq t \leq T}{\sup}\lVert\widetilde{M}_{t}^{N} \lVert_{-s',\sigma'}^{2} )\leq C(T)$ (see Corollary \ref{ct}).\\ Indeed, since for any $ \hspace*{0.1cm}1+D<s'=\frac{s+1+D}{2}<s,$ the embedding $H^{-s',\sigma'}(\mathbb{R}^{d})\hookrightarrow H^{-s,\sigma}(\mathbb{R}^{d})$ is compact (see Corollary \ref{co42}, in the Appendix below), $\mathbb{B}_{-s',\sigma'}(R)=\{ \mu \in H^{-s',\sigma'} ; \lVert \mu \lVert_{H^{-s',\sigma'}}\leq R\},$ which is a closed and bounded subset of $H^{-s',\sigma'}$ is a compact subset of $H^{-s,\sigma}$. Thus 
	\begin{center}
		\begin{align*}
	 \mathbb{P}(\widetilde{M}_{t}^{N}\notin \mathbb{B}_{-s',\sigma'}(R))&=\mathbb{P}(\lVert\widetilde{M}_{t}^{N}\lVert_{-s',\sigma'}> R)\\&\leq \frac{1}{R^{2}}\mathbb{E}(\lVert \widetilde{M}_{t}^{N} \lVert_{-s',\sigma'}^{2})\\&\leq \frac{C(T)}{R^{2}},
	\end{align*}
	\end{center}for any $N\geq 1$. By choosing R arbitrarily large, we make the right hand side as small as we
wish, which yields the result. \\\\
	$-$ Proof of (T2). Note first  that  $ <\widetilde{M}^{N,\varphi}>_{t}=\displaystyle\int_{0}^{t}\Gamma_{r}^{N}(\varphi)dr $, where
	 \begin{align*}
		\Gamma_{r}^{N}(\varphi) &=\displaystyle \left(\mu_{r}^{S,N}, \varphi^{2} (\mu_{r}^{I,N}, K)\right)\\&+   \Big(\mu_{r}^{S,N}, \sum\limits_{1\leq \ell\leq d}\big(\frac{\partial \varphi}{\partial x_{\ell}}\big)^{2}\sum\limits_{1\leq u\leq d}\theta^{2}_{\ell,u}(S,.)+2\sum\limits_{\underset{1\leq e\leq d}{\underset{\ell+1\leq u\leq d}{1\leq \ell\leq d-1}}}\frac{\partial \varphi}{\partial x_{\ell}}\frac{\partial \varphi}{\partial x_{u}}\theta_{\ell,e}(S,.)\theta_{u,e}(S,.)\Big). 
	\end{align*}
According to Theorem 2.3.2 in \cite{rc} it is enough to prove that
\begin{center}
	$\forall T>0 \quad \forall \varepsilon_{1},\varepsilon_{2}>0\quad \exists\delta>0 , N_{0}\geq 1$ such as for any stopping times $\tau^{N} \leq T,$
	\begin{align}
	\underset{N\geq N_{0}}{\sup}\underset{\varrho\leq\delta}{\sup}\mathbb{P}(\lvert<\widetilde{M}^{N}>_{(\tau^{N}+\varrho) }-<\widetilde{M}^{N}>_{\tau^{N}}\lvert>\varepsilon_{1})<\varepsilon_{2} \label{cvb}
\end{align}
\end{center}
 Where    $<\widetilde{M}>$ is the increasing, continuous processes such that, $\lVert \widetilde{M}_{t}\lVert_{H^{-s,\sigma}}^{2}<\widetilde{M}>_{t}$ is a  martingale.
	\vspace*{0.2cm}\\Let $ T>0 , \varepsilon_{1}, \varepsilon_{2}>0 $, $\ell>1$, we find $\delta>0$ such that $\tau^{N}+\delta\leq \ell T$ and such that \ref{cvb} holds.\\ We have
	\begin{align}
	\lvert<\widetilde{M}^{N}>_{(\tau^{N}+\varrho)}-<\widetilde{M}^{N}>_{\tau^{N}}\lvert&=\lvert\sum_{p\geq 1}\{<\widetilde{M}^{N,\varphi_{p}}>_{(\tau^{N}+\varrho)}-<\widetilde{M}^{N,\varphi_{p}}>_{\tau^{N}}\}\lvert\nonumber\\&= \Big\lvert\sum_{p\geq 1}\displaystyle\int_{\tau^{N}}^{(\tau^{N}+\varrho)}\Gamma_{r}^{N}(\varphi_{p})dr\Big\lvert\nonumber\\&=\Big\lvert\sum_{p\geq 1}\displaystyle\int_{0}^{\varrho}\Gamma_{(\tau^{N}+r)}^{N}(\varphi_{p})dr \Big\lvert \nonumber\\&\leq C(\lVert K \lVert_{\infty},d)\displaystyle\int_{0}^{\varrho}\int_{\mathbb{R}^{d}}\{1+\lvert x\lvert^{2\sigma}\}\mu_{\tau^{N}+r}^{S,N}(dx)dr\nonumber\\&\leq \varrho C(\lVert K \lVert_{\infty},d)\displaystyle\underset{N\geq 1}{\sup}\underset{0\leq t \leq \ell T}{\sup}\big(\mu_{t}^{S,N},1+\lvert .\lvert^{2\sigma}\big).\label{xwx}
\end{align}
	Hence it follows from the  Markov inequality and  from Lemma \ref{cor32} and (\ref{xwx}) that 
	\begin{align*}
	\mathbb{P}(\lvert<\widetilde{M}^{N}>_{(\tau^{N}+\varrho)}-<\widetilde{M}^{N}>_{\tau^{N}}\lvert>\varepsilon_{1})&\leq\displaystyle\frac{\mathbb{E}(\lvert<\widetilde{M}^{N}>_{(\tau^{N}+\varrho)}-<\widetilde{M}^{N}>_{\tau^{N}}\lvert)}{\varepsilon_{1}}\\&\leq\displaystyle\frac{C\delta}{\varepsilon_{1}}.
\end{align*}
	(T2) follows.\vspace*{0.2cm}\\
	We conclude from (T1) and (T2) that $ (\widetilde{M}^{N})_{N\geq 1}$ is tight in $\mathbb{D}(\mathbb{R}_{+},H^{-s,\sigma})$. The same arguments yield the tightness of  $ (\widetilde{L}^{N})_{N\geq 1}$ and $ (\widetilde{Y}^{N})_{N\geq 1}$ in $\mathbb{D}(\mathbb{R}_{+},H^{-s,\sigma})$.
\end{proof}
\begin{lem}(We refer to  the proof of  Proposition 6.17 in \cite{sb}, for the proof)\\
	Every limit point $(\mathcal{M}^{1},\mathcal{M}^{2},\mathcal{M}^{3})$ of the sequence $(\widetilde{M}^{N}, \widetilde{L}^{N},\widetilde{Y}^{N})_{N\geq 1} $ is such that for any $\varphi,\psi,\phi\in H^{s,\sigma}$, $((\mathcal{M}^{1},\varphi),(\mathcal{M}^{2},\psi),(\mathcal{M}^{3},\phi)$ is a  martingale.\label{c9}
\end{lem}

The main argument for obtaining the following result is that the jumps of $U^{N}$, $V^{N}$ and $W^{N}$ respectively are of the order of  $\frac{1}{\sqrt{N}}.$
\begin{prp} (We refer to  the proof of  Proposition 6.16 in \cite{sb}, for the detail proof)\\
	Every limit point $(\mathcal{M}^{1},\mathcal{M}^{2},\mathcal{M}^{3})$ of the sequence $(\widetilde{M}^{N}, \widetilde{L}^{N},\widetilde{Y}^{N})_{N\geq 1} $ belongs to of $ (\mathbb{C}(\mathbb{R}_{+},H^{-s,\sigma}))^{3}. $ \label{c10}
\end{prp}
\begin{prp}
	The  sequence $ (\widetilde{M}^{N},\widetilde{L}^{N},\widetilde{Y}^{N})_{N\geq1} $ converges in law  in  $ (\mathbb{D}(\mathbb{R}_{+},H^{-s,\sigma}))^{3}$ towards  the process $(\mathcal{M}^{1},\mathcal{M}^{2},\mathcal{M}^{3})\in (\mathbb{C}(\mathbb{R}_{+},H^{-s,\sigma}))^{3} $ where $ \forall \varphi,\psi,\phi \in H^{s,\sigma},\\ ((\mathcal{M}^{1},\varphi),(\mathcal{M}^{2},\psi),(\mathcal{M}^{2},\phi))$ is a centered Gaussian martingale having the same law as
	\begin{align}
		(\mathcal{M}_{t}^{1},\varphi)&=-\displaystyle\int_{0}^{t}\int_{\mathbb{R}^{d}}\sqrt{ f_{S}(r,x)\int_{\mathbb{R}^{d}}f_{I}(r,y)K(x,y)}dy\varphi(x)\mathcal{W}_{1}(dr,dx)\nonumber\\&+\sum\limits_{\ell=1}^{d}\int_{0}^{t}\int_{\mathbb{R}^{d}}\sqrt{f_{S}(r,x)}\Big\{\sum\limits_{1\leq u\leq d}\frac{\partial \varphi}{\partial x_{u}}(x)\theta_{u,l}(S,x)\Big\}\mathcal{W}_{\ell+1}(dr,dx),\label{qws}\\
	(\mathcal{M}_{t}^{2},\psi)&=\displaystyle\int_{0}^{t}\int_{\mathbb{R}^{d}}\sqrt{ f_{S}(r,x)\int_{\mathbb{R}^{d}}f_{I}(r,y)K(x,y)}dy\psi(x)\mathcal{W}_{1}(dr,dx)\nonumber\\&+\sum\limits_{\ell=1}^{d}\int_{0}^{t}\int_{\mathbb{R}^{d}}\sqrt{f_{I}(r,x)}\Big\{\sum\limits_{1\leq u\leq d}\frac{\partial \psi}{\partial x_{u}}(x)\theta_{u,l}(I,x)\Big\}\mathcal{W}_{\ell+1+d}(dr,dx)\nonumber\\&-\int_{0}^{t}\int_{\mathbb{R}^{d}}\psi(x)\sqrt{\alpha f_{I}(r,x)}\mathcal{W}_{3d+2}(dr,dx),\label{c12}\\
	(\mathcal{M}^{3}_{t},\phi)&=+\sum\limits_{\ell=1}^{d}\int_{0}^{t}\int_{\mathbb{R}^{d}}\sqrt{f_{R}(r,x)}\Big\{\sum\limits_{1\leq u\leq d}\frac{\partial \phi}{\partial x_{u}}(x)\theta_{u,l}(R,x)\Big\}\mathcal{W}_{\ell+1+2d}(dr,dx)\nonumber\\&+\int_{0}^{t}\int_{\mathbb{R}^{d}}\phi(x)\sqrt{\alpha f_{I}(r,x)}\mathcal{W}_{3d+2}(dr,dx),\vspace*{-0.18cm}\label{c13}
	\end{align}
where
	$\mathcal{W}_{1}$, $\mathcal{W}_{2}$,........,$\mathcal{W}_{3d+1}$, $\mathcal{W}_{3d+2}$ are independent spatio-temporal standard white noises. \label{ffg}
\end{prp} 
\begin{proof}
	From Proposition \ref{c8}, $(\widetilde{M}^{N}, \widetilde{L}^{N},\widetilde{Y}^{N})_{N\geq1 }$ is  tight in $(\mathbb{D}(\mathbb{R}_{+},H^{-s,\sigma}))^{3}$, hence according \\to   Prokhorov's Theorem there exists a subsequence still denoted $(\widetilde{M}^{N}, \widetilde{L}^{N},\widetilde{Y}^{N})_{N\geq1 }$ which converges in law in $ (\mathbb{D}(\mathbb{R}_{+},H^{-s,\sigma}))^{3}$ towards $(\mathcal{M}^{1},\mathcal{M}^{2},\mathcal{M}^{3})$. By  Lemma \ref{c9} and  Proposition \ref{c10}, $ \forall \varphi,\psi,\phi \in H^{s,\sigma}$,  $((\mathcal{M}^{1},\varphi),(\mathcal{M}^{2},\psi),(\mathcal{M}^{3},\phi)$ is a  continuous martingale, thus we end the proof of Proposition \ref{ffg}   by showing that 
	the centered, continuous martingale \\$ ((\mathcal{M}^{1},\varphi),(\mathcal{M}^{2},\psi),(\mathcal{M}^{3},\phi))$  is Gaussian  and satisfies  (\ref{qws}), (\ref{c12}) and (\ref{c13}).\vspace*{0.15cm}\\
	For any $\varphi,\psi,\psi\in C^{2}_{b},$ we have \\ \\ 
	$\widetilde{M_{t}}^{N,\varphi}= \displaystyle- \frac{1}{\sqrt{N}} \sum_{i=1}^{N} \int_{0}^{t} \int_{0}^{\infty}1_{\{E_{r^{-}}^{i}=S\}}\varphi(X_{r}^{i})1_{\{u\leq \frac{1}{N} \sum_{j=1}^{N}K(X_{r}^{i},X_{r}^{j}) 1_{\{E_{r}^{j}=I\}}\}} \overline{M}^{i}(dr,du)  \\\hspace*{2cm}+\frac{1}{\sqrt{N}}\sum_{i=1}^{N} \int_{0}^{t}1_{\{E_{r}^{i}=S\}}\bigtriangledown\varphi(X_{r}^{i})\theta(S,X_{r}^{i})dB_{r}^{i}$\vspace*{0.15cm} \\
	\hspace*{1.2cm}=$-M_{t}^{1,N,\varphi}+ M_{t}^{2,N,\varphi},$\vspace*{0.18cm}\\
	$ \widetilde{L_{t}}^{N,\psi} =\displaystyle\frac{1}{\sqrt{N}} \sum_{i=1}^{N} \int_{0}^{t} \int_{0}^{\infty}1_{\{E_{r^{-}}^{i}=S\}}\psi(X_{r}^{i})1_{\{u\leq \frac{1}{N} \sum_{j=1}^{N}K(X_{r}^{i},X_{r}^{j}) 1_{\{E_{r}^{j}=I\}}\}} \overline{M}^{i}(dr,du))  + \vspace*{0.12cm}\\ \hspace*{1cm}+\displaystyle\frac{1}{\sqrt{N}} \sum_{i=1}^{N} \int_{0}^{t}1_{\{E_{r}^{i}=I\}}\bigtriangledown\psi(X_{r}^{i})\theta(I,X_{r}^{i})dB_{r}^{i} $ $   - \displaystyle\frac{1}{\sqrt{N}} \sum_{i=1}^{N} \int_{0}^{t} \int_{0}^{\alpha}1_{\{E_{r^{-}}^{i}=I\}}\psi(X_{r}^{i})\overline{Q}^{i}(dr,du) $\vspace*{0.18cm}\\
	\hspace*{1.1cm}=$ M_{t}^{1,N,\psi}+ M_{t}^{3,N,\psi}-M_{t}^{4,N,\psi},$\\
	$\widetilde{Y_{t}}^{N,\phi}=\displaystyle\frac{1}{\sqrt{N}}\sum_{i=1}^{N} \int_{0}^{t}1_{\{E_{r}^{i}=R\}}\bigtriangledown\phi(X_{r}^{i})\theta(R,X_{r}^{i})dB_{r}^{i}+\displaystyle\frac{1}{\sqrt{N}} \sum_{i=1}^{N} \int_{0}^{t} \int_{0}^{\alpha}1_{\{E_{r^{-}}^{i}=I\}}\phi(X_{r}^{i})\overline{Q}^{i}(dr,du)\vspace*{0.13cm}\\\hspace*{0.9cm}=M_{t}^{5,N,\phi}+M_{t}^{4,N,\phi}.$\vspace*{0.12cm}\\ 
	Consider for $\varphi,\psi,\phi \in C^{2}_{c},$ the following sequence of martingales \vspace*{0.15cm}\\\hspace*{0.5cm} $\widetilde{M_{t}}^{N,\varphi}+\widetilde{L_{t}}^{N,\psi}+\widetilde{Y_{t}}^{N,\phi}=-M_{t}^{1,N,\varphi}+ M_{t}^{2,N,\varphi}+M_{t}^{1,N,\psi}+M_{t}^{3,N,\psi}-M_{t}^{4,N,\psi}+M_{t}^{4,N,\phi}+M_{t}^{5,N,\phi}$ \\
	The martingales  $M_{t}^{1,N,\varphi},M_{t}^{2,N,\varphi},M_{t}^{3,N,\psi},M_{t}^{4,N,\psi},M_{t}^{5,N,\phi}$ being two by two orthogonal, \vspace*{0.2cm} \\
	$<\widetilde{M}^{N,\varphi}+\widetilde{L}^{N,\psi}>_{t}=<M^{1,N,\varphi}>_{t}+<M^{2,N,\varphi}>_{t}+<M^{1,N,\psi}>_{t}+<M^{3,N,\psi}>_{t}+<M^{4,N,\psi}>_{t}\vspace*{0.14cm}\\\hspace*{1.5cm}+<M^{4,N,\phi}>_{t}+<M^{5,N,\phi}>_{t}-2<M^{1,N,\varphi},M^{1,N,\psi}>_{t}-2<M^{4,N,\psi},M^{4,N,\phi}>_{t}.$\vspace*{0.17cm}\\
	In addition we have the following convergences in probability \vspace*{0.1cm}\\
	$\hspace*{1cm}<M^{1,N,\varphi}>_{t}\xrightarrow{P}\displaystyle\int_{0}^{t} \left(\mu_{r}^{S},\varphi^{2}(\mu_{r}^{I}, K)\right)dr$,\\\\
	$\hspace*{0.1cm}<M^{2,N,\varphi}>_{t}\xrightarrow{P} \displaystyle\int_{0}^{t} \Big(\mu_{r}^{S}, \sum\limits_{1\leq \ell\leq d}\big(\frac{\partial \varphi}{\partial x_{\ell}}\big)^{2}\sum\limits_{1\leq u\leq d}\theta^{2}_{\ell,u}(S,.)+2\sum\limits_{\underset{1\leq e\leq d}{\underset{\ell+1\leq u\leq d}{1\leq \ell\leq d-1}}}\frac{\partial \varphi}{\partial x_{\ell}}\frac{\partial \varphi}{\partial x_{u}}\theta_{\ell,e}(S,.)\theta_{u,e}(S,.)\Big) dr.$\\
	\\On the other hand:\vspace*{0.2cm}\\ $-$  $\widetilde{M}^{N,\varphi}+\widetilde{L}^{N,\psi}+\widetilde{Y}^{N,\phi}\xrightarrow{L}(\mathcal{M}^{1},\varphi)+(\mathcal{M}^{2},\psi)+(\mathcal{M}^{3},\phi)$ along a subsequence  since \\\hspace*{0.1cm}  $( \widetilde{M}^{N,\varphi},\widetilde{L}^{N,\psi}, \widetilde{L}^{N,\phi})\xrightarrow{L}((\mathcal{M}^{1},\varphi),(\mathcal{M}^{2},\psi),(\mathcal{M}^{3},\phi))$\\$ -$   $(\mathcal{M}^{1},\varphi)+(\mathcal{M}^{2},\psi)+(\mathcal{M}^{3},\phi)$
	is a continuous martingale since $(\mathcal{M}^{1},\varphi)$, $(\mathcal{M}^{2},\psi)$, and $(\mathcal{M}^{3},\phi)$ \hspace*{0.4cm}have this  property.\vspace*{0.12cm}\\
	Thus  $ (\mathcal{M}^{1},\varphi)+(\mathcal{M}^{2},\psi)+(\mathcal{M}^{3},\phi)$ is a time changed Brownian motion.\vspace*{0.2cm}\\  The quadratic variation \vspace*{0.12cm} \\ $<(\mathcal{M}^{1},\varphi)+(\mathcal{M}^{2},\psi)+(\mathcal{M}^{3},\phi)>_{t}\vspace*{0.25cm}\\\hspace*{0.4cm}=\displaystyle\int_{0}^{t}\big[ \left(\mu_{r}^{S},\varphi^{2}(\mu_{r}^{I}, K)\right)+\left(\mu_{r}^{S},\psi^{2}(\mu_{r}^{I}, K)\right)-2\left(\mu_{r}^{S},\varphi\psi(\mu_{r}^{I}, K)\right)\Big]dr\vspace*{0.2cm}\\\hspace*{0.3cm}+\sum\limits_{A\in \{S,I,R\}}\displaystyle \int_{0}^{t} \Big(\mu_{r}^{A}, \sum\limits_{\underset{1\leq u\leq d}{1\leq \ell\leq d}}\big(\frac{\partial \varphi_{A}}{\partial x_{\ell}}\big)^{2}\theta^{2}_{\ell,u}(A,.)+2\sum\limits_{\underset{1\leq e\leq d}{\underset{\ell+1\leq u\leq d}{1\leq \ell\leq d-1}}}\frac{\partial \varphi_{A}}{\partial x_{\ell}}\frac{\partial \varphi_{A}}{\partial x_{u}}\theta_{\ell,e}(A,.)\theta_{u,e}(A,.)\Big) dr\\\hspace*{0.2cm}\displaystyle+\alpha\int_{0}^{t}\Big\{(\mu_{r}^{I},\psi^{2})+(\mu_{r}^{I},\phi^{2})-2(\mu_{r}^{I},\psi\phi)\Big\}dr,$\vspace*{0.14cm}\\ (where we have let $\varphi_{S}=\varphi$, $\varphi_{I}=\psi$ and $\varphi_{R}=\phi$) of $ (\mathcal{M}^{1},\varphi)+(\mathcal{M}^{2},\psi)+(\mathcal{M}^{3},\phi) $ being deterministic then 
	we conclude that \\$ (\mathcal{M}^{1},\varphi)+(\mathcal{M}^{2},\psi)+(\mathcal{M}^{3},\phi) $ is a Gaussian martingale having the same law as
	\begin{align*}
	\mathcal{N}_{t}&=\displaystyle\int_{0}^{t}\int_{\mathbb{R}^{d}}\sqrt{ f_{S}(r,x)\int_{\mathbb{R}^{d}}f_{I}(r,y)K(x,y)dy}(\psi(x)-\varphi(x))\mathcal{W}_{1}(dr,dx)\\&+\sum\limits_{1=1}^{d}\int_{0}^{t}\int_{\mathbb{R}^{d}}\sqrt{f_{S}(r,x)}\Big\{\sum\limits_{1\leq u\leq d}\frac{\partial \varphi}{\partial x_{u}}(x)\theta_{u,l}(S,x)\Big\}\mathcal{W}_{l+1}(dr,dx)\\&+\sum\limits_{1=1}^{d}\int_{0}^{t}\int_{\mathbb{R}^{d}}\sqrt{f_{I}(r,x)}\Big\{\sum\limits_{1\leq u\leq d}\frac{\partial \psi}{\partial x_{u}}(x)\theta_{u,l}(I,x)\Big\}\mathcal{W}_{l+d+1}(dr,dx)\\&+\sum\limits_{1=1}^{d}\int_{0}^{t}\int_{\mathbb{R}^{d}}\sqrt{f_{R}(r,x)}\Big\{\sum\limits_{1\leq u\leq d}\frac{\partial \phi}{\partial x_{u}}(x)\theta_{u,l}(R,x)\Big\}\mathcal{W}_{2d+1}(dr,dx)\\&+\int_{0}^{t}\int_{\mathbb{R}^{d}}\sqrt{\alpha f_{I}(r,x)}(\phi(x)-\psi(x))\mathcal{W}_{3d+2}(dr,dx),
		\end{align*}
	where
	$\mathcal{W}_{1}$, $\mathcal{W}_{2}$, .............. $\mathcal{W}_{3d+1}$, $\mathcal{W}_{3d+2}$ are independent spatio-temporal white noises. \vspace*{0.2cm}\\So taking $(\psi\equiv0,\phi\equiv0)$, $(\varphi\equiv0,\phi\equiv0)$ and  $(\varphi\equiv0,\psi\equiv 0)$ respectively, in the above equation we see that $(\mathcal{M}^{1},\varphi)$, $(\mathcal{M}^{2},\psi)$ and $(\mathcal{M}^{3},\phi)$ satisfy (\ref{qws}), (\ref{c12}) and (\ref{c13}).
\end{proof}
\begin{prp}
	There exists a constant $C>0,$ such that  for any $\varphi\in H^{s,\sigma}(\mathbb{R}^{d})$, 
	\begin{align}
		\lVert G_{r}^{I,N}\varphi \lVert_{s,\sigma}&\leq C  \lVert \varphi \lVert_{s,\sigma}\underset{y\in \mathbb{R}^{d}}{\sup}\lVert K(.,y)\lVert_{2+D,\sigma},\label{le1}\\
		\lVert G_{r}^{S}\varphi \lVert_{s,\sigma}&\leq C\lVert \varphi \lVert_{s,\sigma}\underset{y\in \mathbb{R}^{d}}{\sup}\lVert K(.,y)\lVert_{2+D,\sigma}\label{le2},\\
		\lVert G_{r}^{I}\varphi \lVert_{s,\sigma}&\leq C  \lVert \varphi \lVert_{s,\sigma}\underset{y\in \mathbb{R}^{d}}{\sup}\lVert K(.,y)\lVert_{2+D,\sigma}.\label{le1e}
	\end{align}
\label{prp215}
\end{prp}
	\begin{proof} We first recall that $\forall x,y \in \mathbb{R}^{d}$,\\ \hspace*{3.2cm} $G_{r}^{I,N}\varphi(x)= \varphi(x)(\mu_{r}^{I,N},  K(x,.))=\displaystyle\varphi(x) \int_{\mathbb{R}^{d}}K(x,y)\mu_{r}^{I,N}(dy),$  \vspace*{0.1cm} \\\hspace*{3.2cm} $G_{r}^{S}\varphi(y)= (\mu_{r}^{S}, \varphi K(.,y))=\displaystyle\int_{\mathbb{R}^{d}}\varphi(x) K(x,y)\mu_{r}^{S}(dx).$\vspace*{0.1cm}\\ Proof of \eqref{le1}.
	Since $H^{s,\sigma}$ is a Banach algebra (see Remark \ref{r65} in the Appendix below),  we have
	
	\begin{align}
		\lVert G_{r}^{I,N}\varphi \lVert_{s,\sigma}\leq C \lVert \varphi \lVert_{s,\sigma}\lVert (\mu_{r}^{I,N}, K) \lVert_{s,\sigma}\leq C \lVert \varphi \lVert_{s,\sigma}\lVert (\mu_{r}^{I,N}, K) \lVert_{2+D,\sigma}.\label{le11}
		\end{align}
	 Furthermore 
	 
	 \begin{align} 
	 	\lVert (\mu_{r}^{I,N}, K) \lVert_{2+D,\sigma}^{2}&=\displaystyle \sum\limits_{\lvert \gamma \lvert \leq 2+D}\int_{\mathbb{R}^{d}}\frac{\lvert D^{\gamma}(\mu_{r}^{I,N}, K(x,.))\lvert^{2}}{1+\lvert x\lvert ^{2\sigma}}dx,\nonumber\\&=\displaystyle \sum\limits_{\lvert \gamma \lvert \leq 2+ D}\int_{\mathbb{R}^{d}}\frac{\Big\lvert \int_{\mathbb{R}^{d}}D^{\gamma} K(x,y)\mu_{r}^{I,N}(dy)\Big\lvert^{2}}{1+\lvert x\lvert ^{2\sigma}}dx,\nonumber\\&\leq \displaystyle  \displaystyle \int_{\mathbb{R}^{d}}\sum\limits_{\lvert \gamma \lvert \leq 2+ D}\int_{\mathbb{R}^{d}}\frac{\lvert D^{\gamma} K(x,y)\lvert^{2}}{1+\lvert x\lvert ^{2\sigma}}dx\mu_{r}^{I,N}(dy),\nonumber\\&\leq\underset{y\in \mathbb{R}^{d}}{\sup}\lVert K(.,y)\lVert_{2+D,\sigma}^{2}.
	 	\label{le12}
	 \end{align}
 Thus (\ref{le1}) follows from (\ref{le11}) and (\ref{le12}).\\\\
 Proof of \ref{le2}.
 	Once again since  $ H^{2+D,\sigma}\hookrightarrow H^{s,\sigma}$,  we have
 \begin{align}
 	\lVert G_{r}^{S}\varphi \lVert_{s,\sigma}=\lVert (\mu_{r}^{S},\varphi K) \lVert_{s,\sigma}\leq C\lVert (\mu_{r}^{S},\varphi K) \lVert_{2+D,\sigma}. \label{le21}
 \end{align}
 Furthermore from Corollary \ref{cor32}, we have
 \begin{align} 
 	\lVert (\mu_{r}^{S},\varphi K) \lVert_{D,\sigma}^{2}&=\displaystyle \sum\limits_{\lvert \gamma \lvert \leq 2+D}\int_{\mathbb{R}^{d}}\frac{\lvert D^{\gamma}(\mu_{r}^{S},\varphi K(.,y))\lvert^{2}}{1+\lvert y\lvert ^{2\sigma}}dx,\nonumber\\&=\displaystyle \sum\limits_{\lvert \gamma \lvert \leq 2+D}\int_{\mathbb{R}^{d}}\frac{\Big\lvert \int_{\mathbb{R}^{d}}\varphi(x)D^{\gamma} K(x,y)\mu_{r}^{S}(dx)\Big\lvert^{2}}{1+\lvert y\lvert ^{2\sigma}}dx,\nonumber\\&\leq \displaystyle \int_{\mathbb{R}^{d}}\varphi^{2}(x)\mu_{r}^{S}(dx) \displaystyle \int_{\mathbb{R}^{d}}\sum\limits_{\lvert \gamma \lvert \leq 2+D}\int_{\mathbb{R}^{d}}\frac{\lvert D^{\gamma} K(x,y)\lvert^{2}}{1+\lvert y\lvert ^{2\sigma}}dy\mu_{r}^{S}(dx),\nonumber\\&\leq \lVert \varphi\lVert_{C^{0,\sigma}}^{2}\underset{x\in \mathbb{R}^{d}}{\sup}\lVert K(x,.)\lVert_{2+D,\sigma}^{2}.
 	\label{le22}
 \end{align}
So (\ref{le2}) follows from (\ref{le21}) and (\ref{le22}) and the embedding $H^{s,\sigma}\hookrightarrow C^{0,\sigma}.$ \\The proof of (\ref{le1e}) is similar to that of (\ref{le1}).
	\end{proof}
\begin{cor}
	There exists a constant $C>0,$ such that for any $\mathcal{U}\in H^{-s,\sigma},$ we have 
	\begin{align}
			\lVert (G_{r}^{I,N})^{*}\mathcal{U}  \lVert_{-s,\sigma}&\leq C  \underset{y\in \mathbb{R}^{d}}{\sup}\lVert K(.,y)\lVert_{2+D,\sigma}\lVert \mathcal{U}\lVert_{-s,\sigma},\label{le3}\\
		\lVert (G_{r}^{S})^{*}\mathcal{U} \lVert_{-s,\sigma}&\leq C\underset{x\in \mathbb{R}^{d}}{\sup}\lVert K(x,.)\lVert_{2+D,\sigma}\lVert \mathcal{U}\lVert_{-s,\sigma},\label{le2r}\\
		\lVert (G_{r}^{I})^{*}\mathcal{U}  \lVert_{-s,\sigma}&\leq C  \underset{y\in \mathbb{R}^{d}}{\sup}\lVert K(.,y)\lVert_{2+D,\sigma}\lVert \mathcal{U}\lVert_{-s,\sigma}.\label{le3e}
	\end{align}
\label{uit}
\end{cor}
\subsubsection{The evolution semi group}
Let us define the evolution semigroup as a semi group of bounded linear  operators in a Banach space. Let us assume, for the moment in a general context, that for any $A\in\{S,I,R\},$  the coeficients $m(A,.)$ and $\theta(A,.)$  are in $C^{j+1}_{b},$ where $j$ is a positive  integer. Let $(B_{t})_{t\geq 0}$ be a standard Brownian motion on $\mathbb{R}^{d}.$ For any $A\in\{S,I,R\},$ one defined (see for example Kunita \cite{kun} ) the flow of diffeomorphisms (of class $C^{j}$)  as the unique  solution of the Itô stochastic differential equation started from $x\in \mathbb{R}^{d}$ at time $u:$
\begin{align}
	\displaystyle X_{u,t}^{A,x}=x+\int_{u}^{t}m(A,X_{u,t}^{A,x})dr+\int_{u}^{t}\theta(A,X_{u,t}^{A,x})dB_{r}.
\end{align}
Moereover for any measurable and bounded function $\varphi,$ $A\in\{S,I,R\},$ we define
\begin{center} $\Upsilon_{A}(t-u)\varphi(x)=\mathbb{E}(\varphi(X_{u,t}^{A,x}))$ and in the folowing $\Upsilon^{*}_{A}(t)$ denotes the adjoint operator.
\end{center}
When $u=0,$ $X_{u,t}^{A,x}$ is denotes $X_{t}^{A,x}.$\\
We note that under the Assumptions $(H2),$  for any $0\leq u <t$ the map $x\in \mathbb{R}^{d}\longmapsto X_{u,t}^{A,x}$ is of class $C^{2+D},$ and the following results  hold true.\\ 
$-$ (Thank to Corrolary 4.6.7 in \cite{kun}) For any  $0\leq\lvert \gamma \lvert \leq 2+D,$   for any $p\geq1$, there exits a constant  $C$ independent of $t$, such that \begin{align}
	 \underset{x}{\sup}\mathbb{E}(\lvert D^{\gamma}X_{t}^{A,x}\lvert^{2p})\leq C. \label{baq}
	\end{align}
$-$ (See Lemma 4.5.3 in \cite{kun}) For any real $p,$ there exists a positive constant $C_{p},$ such that 
	\begin{align}
		\mathbb{E}[(1+\lvert X_{t}^{A,x}\lvert^{2})^{p}]\leq C_{p}(1+\lvert x\lvert^{2})^{p}, \quad \forall x\in \mathbb{R}^{d}.\label{caq}
	\end{align}

Now we have the following result.
\begin{lem}
	Under the Asumption (H2), for any $A\in\{S,I,R\},$ $t>0,$ $m\in \{0,1,....2+D\},$ $\varphi\in W_{0}^{m,\sigma}(\mathbb{R}^{d}),$ there exists a positive constant $C,$ such that:
	\begin{center}
		$\lVert \Upsilon_{A}(t)\varphi\lVert_{m,\sigma}\leq Ce^{Ct}\lVert \varphi \lVert_{m,\sigma}$
	\end{center}
	\label{la1}
\end{lem}
\begin{proof}  We have\\    $\hspace*{4cm}\displaystyle\lVert \Upsilon_{A}(t)\varphi\lVert_{m,\sigma}^{2}=\sum\limits_{\lvert \gamma\lvert\leq m}\int_{\mathbb{R}^{d}}\frac{\lvert D^{\gamma}\Upsilon_{A}(t)\varphi(x)\lvert^{2}}{1+\lvert x\lvert^{2\sigma}}dx,$\\ furthermore  for $\lvert \gamma \lvert=0,$ and by using (\ref{caq}) and Lemma \ref{la2} in the Appendix below, we have 
	\begin{align*}
		\displaystyle\int_{\mathbb{R}^{d}}\frac{\lvert \Upsilon_{A}(t)\varphi\lvert^{2}}{1+\lvert x\lvert^{2\sigma}}dx=\int_{\mathbb{R}^{d}}\frac{\lvert \mathbb{E}(\varphi(X_{t}^{A,x}))\lvert^{2}}{1+\lvert x\lvert^{2\sigma}}dx&\leq\int_{\mathbb{R}^{d}}\frac{1}{1+\lvert x\lvert^{2\sigma}}\mathbb{E}[(1+\lvert X_{t}^{A,x}\lvert^{\sigma})^{2}] \displaystyle\mathbb{E}\Big(\frac{\lvert\varphi(X_{t}^{A,x})\lvert^{2}}{(1+\lvert X_{t}^{A,x}\lvert^{\sigma})^{2}}\Big)dx,\\&\leq C \int_{\mathbb{R}^{d}}\frac{(1+\lvert x\lvert ^{2})^{\sigma}}{1+\lvert x\lvert^{2\sigma}} \displaystyle\int_{\mathbb{R}^{d}}\Upsilon_{A}(t)(x,y)\frac{\lvert\varphi(y)\lvert^{2}}{1+\lvert y\lvert^{2\sigma}}dydx,\\&\leq C(\sigma)\lVert \varphi\lVert_{L^{2,\sigma}} \underset{y\in \mathbb{R}^{d}}{\sup}\Big(\int_{\mathbb{R}^{d}}\Upsilon_{S}(t)(x,y)dx\Big),\\&\leq C(\sigma)\lVert \varphi\lVert_{m,\sigma}e^{Ct}.
	\end{align*}
	For $ \gamma=(1,0,0...0)$ and by using (\ref{baq}) and Lemma \ref{la2}, we have 
	\begin{align*}
		\displaystyle\int_{\mathbb{R}^{d}}\frac{\lvert D^{\gamma}\Upsilon_{A}(t)\varphi\lvert^{2}}{1+\lvert x\lvert^{2\sigma}}dx&=\int_{\mathbb{R}^{d}}\frac{\lvert \mathbb{E}(D^{\gamma}\varphi(X_{t}^{A,x}))\lvert^{2}}{1+\lvert x\lvert^{2\sigma}}dx,\\&\leq\int_{\mathbb{R}^{d}}\frac{1}{1+\lvert x\lvert^{2\sigma}}\mathbb{E}[\{D^{\gamma}X_{t}^{A,x}(1+\lvert X_{t}^{A,x}\lvert^{\sigma})\}^{2}] \displaystyle\mathbb{E}\Big(\frac{\lvert\nabla\varphi(X_{t}^{A,x})\lvert^{2}}{(1+\lvert X_{t}^{A,x}\lvert^{\sigma})^{2}}\Big)dx,\\&\leq C \int_{\mathbb{R}^{d}}\frac{(1+\lvert x\lvert ^{2})^{\sigma}}{1+\lvert x\lvert^{2\sigma}} \displaystyle\mathbb{E}[(D^{\gamma}X_{t}^{A,x})^{4}]^{1/2}\mathbb{E}[(1+\lvert X_{t}^{A,x}\lvert^{\sigma})^{4}]^{1/2}\\&\hspace*{2cm}\times\int_{\mathbb{R}^{d}}\Upsilon_{A}(t)(x,y)\frac{\lvert\nabla\varphi(y)\lvert^{2}}{1+\lvert y\lvert^{2\sigma}}dydx,\\&\leq C(\sigma,d)\lVert \varphi\lVert_{1,\sigma} \underset{y\in \mathbb{R}^{d}}{\sup}\Big(\int_{\mathbb{R}^{d}}\Upsilon_{A}(t)(x,y)dx\Big),\\&\leq C(\sigma,d)\lVert \varphi\lVert_{m,\sigma}e^{Ct}.
	\end{align*}
	Similar argument allow us to have $\displaystyle\int_{\mathbb{R}^{d}}\frac{\lvert D^{\gamma}\Upsilon_{S}(t)\varphi\lvert^{2}}{1+\lvert x\lvert^{2\sigma}}dx\leq C(\sigma,d)\lVert \varphi\lVert_{m,\sigma}e^{Ct},$  for all values of $\lvert \gamma\lvert \leq m$. So the proof is complete.
\end{proof}
\begin{cor}
	For any positive noninteger $1+D<s<2+D$,  $\varphi\in H^{s,\sigma},$ there exists a positive contant $C,$ such that $\lVert \Upsilon_{S}(t)\varphi\lVert_{s,\sigma}\leq Ce^{Ct} \lVert \varphi\lVert_{s,\sigma}.$ \label{cor20}
\end{cor}
\begin{proof}
	We prove this result by using the definition by interpolation of the space $H^{s,\sigma}.$\\ For any noninterger $s>0,$ there exists  $\rho\in ]0,1[$ such that $s=(1-\rho)(1+D)+\rho (2+D)$ and $(W_{0}^{1+D,\sigma},W_{0}^{2+D,\sigma})_{\rho,2}=H^{s,\sigma},$ for the definition of the space $(.,.)_{\rho,q}$ we refer to \cite{lof} or \cite{tri}.\\So using Lemma 3.1.1 in \cite{loff}, we have an equivalent norm on $H^{s,\sigma},$ which is given by:\\\hspace*{4cm}$\lVert \varphi\lVert_{s,\sigma}=\displaystyle\left(\int_{0}^{\infty}\{t^{-\rho}K(t,1+D,2+D)\}^{2}\frac{dt}{t}\right)^{1/2},$\vspace*{0.2cm}\\ where $K(t,1+D,2+D)=\underset{\varphi=\varphi_{1}+\varphi_{2}}{\inf}\{\lVert\varphi_{1}\lVert_{1+D,\sigma}+t\lVert\varphi_{2}\lVert_{2+D,\sigma}\}.$  \\So it is easy to see that the result follows from Lemma $\ref{la1}$ and the above definition.
\end{proof}
Let us prove the following results which will be useful to prove the tightness of the sequence $(U^{N},V^{N},W^{N})_{N}$ in $\mathbb{D}(\mathbb{R}_{+},H^{-s,\sigma}).$
\begin{lem}
	The sequence  of processes   $ (U^{N},V^{N},W^{N})$ satisfies $\forall \hspace*{0.1cm}0\leq u<t$,
	
	\begin{align}
	\displaystyle U_{t}^{N}&=\displaystyle  \Upsilon^{*}_{S}(t-u)U_{u}^{N}   -\int_{u}^{t}\Upsilon^{*}_{S}(t-r)(G_{r}^{I,N})^{*}U_{r}^{N}dr- \int_{u}^{t} \Upsilon^{*}_{S}(t-r)(G_{r}^{S})^{*}V_{r}^{N}dr\nonumber\\ &+\int_{u}^{t}\Upsilon^{*}_{S}(t-r)d\widetilde{M}_{r}^{N},\label{c15}\\
	V_{t}^{N}&=\displaystyle \Upsilon^{*}_{I}(t-u)V_{u}^{N}  +\int_{u}^{t}\Upsilon^{*}_{I}(t-r)(G_{r}^{I,N})^{*}U_{r}^{N}dr + \int_{u}^{t} \Upsilon^{*}_{I}(t-r)[(G_{r}^{I,N})^{*}-\alpha ]V_{r}^{N}dr\nonumber\\& +\int_{u}^{t}\Upsilon^{*}_{I}(t-r)d\widetilde{L}_{r}^{N}, \label{c16}\\
	 \displaystyle W_{t}^{N}&= \Upsilon^{*}_{R}(t-u)W_{u}^{N}  +\alpha \int_{u}^{t} \Upsilon^{*}_{R}(t-r)V_{r}^{N}dr  +\int_{u}^{t}\Upsilon^{*}_{R}(t-r)d\widetilde{Y}_{r}^{N}. \label{c17}
	\end{align}
\label{sds}
\end{lem}                                                                    \begin{proof}
	Let us consider a fuction $\phi$ belonging to $ C^{1,2}_{c}(\mathbb{R}_{+}\times\mathbb{R}^{d})$. By Itô's formula applied   to $\phi(t,X_{t}^{i})$  and  using a similar computation as in subsections \ref{hm} and \ref{cc1}, we obtain for $0\leq u< t,$ \vspace*{0.12cm}\\
	$ \displaystyle (U_{t}^{N},\phi_{t})=\displaystyle  (U_{u}^{N},\phi_{u}) +  \int_{u}^{t}(U_{r}^{N},\mathcal{Q}_{S} \phi_{r}) dr+  \int_{u}^{t}(U_{r}^{N},\frac{\partial \phi_{r}}{\partial r}) dr -\int_{u}^{t}\left(U_{r}^{N} ,\phi_{r}(\mu_{r}^{I,N}, K) \right)dr \\\hspace*{1.6cm} - \int_{u}^{t}\left( V_{r}^{N},( \mu_{r}^{S},\phi_{r} K) \right)dr + \int_{u}^{t}(\phi_{r},d\widetilde{M}_{r}^{N}). $ \vspace*{0.2cm}\\Let $\varphi \in C^{2}_{b}$ and $ 0\leq u <t $, consider for $ r\in [u,t] $ the mapping $\psi_{r}(x)
	=\Upsilon_{S}(t-r)\varphi(x) $.\\ We have $\psi_{\cdot}(\cdot)\in C^{1,2}_{c}([u,t]\times\mathbb{R}^{d})$, indeed, \vspace*{0.12cm}\\
	\hspace*{0.5cm} - For any $ r\in [u,t]$, $\psi_{r}(\cdot)\in  C^{2}_{c}(\mathbb{R}^{d})$\\\hspace*{0.5cm} -  $\forall x\in \mathbb{R}^{d}$, the map  $r\in [u,t]\mapsto  \psi^{'}_{r}(x)=-\mathcal{Q}_{S}( \Upsilon_{S}(t-r)\varphi(x))$ is continuous  since  $\Upsilon_{S}(t)$ is a strongly continuous semi-group and  $-\mathcal{Q}_{S}( \Upsilon_{S}(t-r)\varphi(x))=\Upsilon_{S}(t-r)(-\mathcal{Q}_{S}\varphi(x))$. Thus  replacing  $ \phi $ by $ \psi $ in the above equation, we obtain \vspace*{0.12cm}\\$ \displaystyle (U_{t}^{N},\varphi)=\displaystyle  (U_{u}^{N},\Upsilon_{S}(t-u)\varphi)  -\int_{u}^{t}\left(U_{r}^{N} ,\Upsilon_{S}(t-r)\varphi(\mu_{r}^{I,N}, K) \right)dr \\\hspace*{1.6cm} - \int_{u}^{t}\left( V_{r}^{N},( \mu_{r}^{S},\Upsilon_{S}(t-r)\varphi K) \right)dr+ \int_{u}^{t}(\Upsilon_{S}(t-r)\varphi,d\widetilde{M}_{r}^{N}). $\vspace*{0.12cm}\\ This prove \eqref{c15}. We  obtain (\ref{c16}) and (\ref{c17}) by similar arguments.       
\end{proof}                      
\begin{prp} There exists
	 $ C>0$ such that for  any $T>0,$ $\varrho>0$ and any stopping times $\overline{\tau}$ such that $\overline{\tau}+\varrho<T$, one has
	 
	 \begin{align}
	 	\mathbb{E}\left(\Big\lVert \displaystyle\int_{\overline{\tau}}^{\overline{\tau}+\varrho}\Upsilon_{S}^{*}(\overline{\tau}+\varrho-r)d\widetilde{M}_{r}^{N} \Big\lVert_{-s,\sigma}^{2}\right)&\leq C\varrho,\label{c1r}\\ \mathbb{E}\left(\Big\lVert \displaystyle\int_{\overline{\tau}}^{\overline{\tau}+\varrho}\Upsilon_{I}^{*}(\overline{\tau}+\varrho-r)d\widetilde{L}_{r}^{N} \Big\lVert_{-s,\sigma}^{2}\right)&\leq C\varrho,\label{clom}\\
 \mathbb{E}\left(\Big\lVert \displaystyle\int_{\overline{\tau}}^{\overline{\tau}+\varrho}\Upsilon_{R}^{*}(\overline{\tau}+\varrho-r)d\widetilde{Y}_{r}^{N} \Big\lVert_{-s,\sigma}^{2}\right)&\leq C\varrho.\label{c20}
\end{align}
\label{cty}
\end{prp}
\begin{proof} Proof of $(\ref{c1r})$.\vspace*{0.1cm} Let us recall that \\\hspace*{0.1cm}  $\displaystyle\int_{\overline{\tau}}^{\overline{\tau}+\varrho}(\Upsilon_{S}(\overline{\tau}+\varrho-r)\varphi,d\widetilde{M}_{r}^{N})\\\hspace*{1.5cm}=\frac{1}{\sqrt{N}} \sum\limits_{i=1}^{N} \int_{\overline{\tau}}^{\overline{\tau}+\varrho}1_{\{E_{r}^{i}=S\}}\bigtriangledown\Upsilon_{S}(\overline{\tau}+\varrho-r)\varphi(X_{r}^{i})\theta(S,X_{r}^{i})dB_{r}^{i} \\\hspace*{1.5cm}- \displaystyle\sqrt{\frac{1}{N}} \sum_{i=1}^{N} \int_{\overline{\tau}}^{\overline{\tau}+\varrho} \int_{0}^{\infty}1_{\{E_{r^{-}}^{i}=S\}}\Upsilon_{S}(\overline{\tau}+\varrho-r)\varphi(X_{r}^{i})1_{\{u\leq \frac{1}{N} \sum\limits_{j=1}^{N}K(X_{r}^{i},X_{r}^{j}) 1_{\{E_{r}^{j}=I\}}\}} \overline{M}^{i}(dr,du).$\vspace*{0.1cm}\\Now  from  Remark \ref{r3}, (\ref{baq})  and (\ref{caq}) one has  for all $\ell \in\{1,2,.......,d\},$  $0\leq r \leq \varrho,$
	
	\begin{align}
			\displaystyle\sum\limits_{p\geq 1}\Big(\Upsilon_{S}(\varrho-r)\varphi_{p}(x)\Big)^{2}&=\sum\limits_{p\geq 1}\Big\{\mathbb{E}(\varphi_{p}(X_{\varrho-r}^{S,x}))\Big\}^{2}\leq\mathbb{E}\Big(\sum\limits_{p\geq 1}\lvert \varphi_{p}( X_{\varrho-r}^{S,x})\lvert^{2}\Big),\nonumber\\&\leq C \mathbb{E}(1+\lvert X_{\varrho-r}^{S,x}\lvert^{2\sigma} ),\nonumber\\&\leq C(\sigma)(1+\lvert x\lvert ^{2\sigma}).
		\label{azee} 
	\end{align}
	\begin{align}
		\displaystyle\sum\limits_{p\geq 1}\Big(\frac{\partial}{\partial x_{\ell}}\Upsilon_{S}(\varrho-r)\varphi_{p}(x)\Big)^{2}&=\sum\limits_{p\geq 1}\Big\{\frac{\partial}{\partial x_{\ell}}\mathbb{E}(\varphi_{p}(X_{\varrho-r}^{S,x}))\Big\}^{2}\leq\mathbb{E}(\lvert\partial_{x_{\ell}} X_{\varrho-r}^{S,x}\lvert^{2} )\mathbb{E}\Big(\sum\limits_{p\geq 1}\lvert(\nabla\varphi_{p})( X_{\varrho-r}^{S,x})\lvert^{2}\Big),\nonumber\\&\leq C(d) \mathbb{E}(1+\lvert X_{\varrho-r}^{S,x}\lvert^{2\sigma} ),\nonumber\\&\leq C(d,\sigma)(1+\lvert x\lvert ^{2\sigma}).
		\label{aze}
	\end{align} 

Thus from (\ref{azee}) and (\ref{aze}) and Lemma \ref{lt1}, we have 
\begin{align*}
	\displaystyle\mathbb{E}\left(\left\lVert \int_{\overline{\tau}}^{\overline{\tau}+\varrho}\Upsilon_{S}(\overline{\tau}+\varrho-r)d\widetilde{M}_{r}^{N} \right\lVert_{H^{-s}}^{2}\right)&= \sum\limits_{p\geq 1}\mathbb{E}\left( \left(\displaystyle\int_{\overline{\tau}}^{\overline{\tau}+\varrho}\Upsilon_{S}(\overline{\tau}+\varrho-r)\varphi_{p},d\widetilde{M}_{r}^{N} \right)^{2}\right),\\ &\hspace*{-4cm}=\sum\limits_{p\geq 1}\Big\{\mathbb{E}\left( \displaystyle\int_{0}^{\varrho}  \left(\mu_{r+\overline{\tau}}^{S,N},(\Upsilon_{S}(\varrho-r)\varphi_{p})^{2}(\mu_{r+\overline{\tau}}^{I,N}, K)\right)dr\right)\\&\hspace*{-4cm}+\mathbb{E}\Big(\displaystyle \int_{0}^{\varrho}\Big(\mu_{r+\overline{\tau}}^{S,N}, \sum\limits_{1\leq \ell\leq d}\big(\frac{\partial \Upsilon_{S}(\varrho-r)\varphi_{p}}{\partial x_{\ell}}\big)^{2}\sum\limits_{1\leq u\leq d}\theta^{2}_{\ell,u}(S,.)\Big)dr\Big)\\&\hspace*{-4cm}+\mathbb{E}\Big(\displaystyle \int_{0}^{\varrho}\Big(\mu_{r+\overline{\tau}}^{S,N},\sum\limits_{\underset{1\leq u\leq d}{\underset{1\leq e\leq d}{1\leq \ell\leq d-1}}}\frac{\partial \Upsilon_{S}(\varrho-r)\varphi_{p}}{\partial x_{\ell}}\frac{\partial \Upsilon_{S}(\varrho-r)\varphi_{p}}{\partial x_{u}}\theta_{l,e}(S,.)\theta_{u,e}(S,.)\Big)dr \Big)\Big\},\\ &\hspace*{-4cm}\leq\mathbb{E}\left( \displaystyle\int_{0}^{\varrho}  \left(\mu_{r+\overline{\tau}}^{S,N}, \sum\limits_{p\geq 1}(\Upsilon_{S}(\varrho-r)\varphi_{p})^{2}(\mu_{r+\overline{\tau}}^{I,N}, K)\right)dr\right)\\&\hspace*{-4cm}+\mathbb{E}\Big(\displaystyle \int_{0}^{\varrho}\Big(\mu_{r+\overline{\tau}}^{S,N}, \sum\limits_{\underset{1\leq u\leq d}{1\leq \ell\leq d}}\theta^{2}_{\ell,u}(S,.)\sum\limits_{p\geq 1}\big(\frac{\partial \Upsilon_{S}(\varrho-r)\varphi_{p}}{\partial x_{\ell}}\big)^{2}\Big)dr\Big)\\&\hspace*{-6.5cm}+\frac{1}{2}\mathbb{E}\Big(\displaystyle \int_{0}^{\varrho}\Big(\mu_{r+\overline{\tau}}^{S,N},\sum\limits_{\underset{1\leq u\leq d}{\underset{1\leq e\leq d}{1\leq \ell\leq d-1}}}\lvert\theta_{\ell,e}(S,.)\theta_{u,e}(S,.) \lvert \Big\{ \sum\limits_{p\geq 1} \big(\frac{\partial \Upsilon_{S}(\varrho-r)\varphi_{p}}{\partial x_{\ell}}\big)^{2}+\sum\limits_{p\geq 1}\big(\frac{\partial \Upsilon_{S}(\varrho-r)\varphi_{p}}{\partial x_{u}}\big)^{2}\Big\}\Big)dr \Big),\\ &\hspace*{-4cm}\leq\varrho\lVert K\lVert_{\infty}
		\underset{N\geq 1}{\sup}\underset{0\leq t\leq T}{\sup}\mathbb{E}( (\mu_{t}^{S,N},1+\lvert .\lvert ^{2\sigma}))\\&\hspace*{-4cm}+ \varrho C(\sigma) \sum\limits_{\underset{1\leq u\leq d}{1\leq \ell\leq d}}\lVert\theta^{2}_{\ell,u}(S,.)\lVert_{\infty}\underset{N\geq 1}{\sup}\underset{0\leq t\leq T}{\sup}\mathbb{E}( (\mu_{t}^{S,N},1+\lvert .\lvert ^{2\sigma}))\\&\hspace*{-4cm}+\varrho\sum\limits_{\underset{1\leq u\leq d}{\underset{1\leq e\leq d}{1\leq \ell\leq d-1}}}\lVert\theta_{\ell,e}(S,.)\lVert_{\infty}\lVert\theta_{u,e}(S,.)\lVert_{\infty} \underset{N\geq 1}{\sup}	\underset{0\leq t\leq T}{\sup}\mathbb{E}( (\mu_{t}^{S,N},1+\lvert .\lvert ^{2\sigma})), \\&\hspace*{-4cm}\leq\varrho C.
\end{align*}
Similar arguments yield (\ref{clom}) and (\ref{c20}).
\end{proof}
\begin{prp}
	For all  $T>0$,  
	\begin{align}
	\displaystyle  \underset{N\geq 1}{\sup}  \mathbb{E}(\underset{0\leq t\leq T}{\sup}\lVert U_{t}^{N} \lVert^{2}_{-s,\sigma})&<\infty, \\
	 \displaystyle   \underset{N\geq 1}{\sup} \mathbb{E}(\underset{0\leq t\leq T}{\sup}\lVert V_{t}^{N} \lVert^{2}_{-s,\sigma})&<\infty, \\  \displaystyle   \underset{N\geq 1}{\sup}  \mathbb{E}(\underset{0\leq t\leq T}{\sup}\lVert W_{t}^{N} \lVert^{2}_{-s,\sigma})&<\infty \label{f8}.
\end{align}  
\label{tfg}
\end{prp}        

\begin{proof} 
	Choosing $u=0$ in equation (\ref{c15}), (\ref{c16}) and (\ref{c17}) we have 
	
	\begin{align*}
	\displaystyle \lVert U_{t}^{N} \lVert_{-s,\sigma}^{2}&\leq \displaystyle 4 \lVert \Upsilon^{*}_{S}(t)U_{0}^{N}\lVert_{-s,\sigma}^{2} + 4t\int_{0}^{t}\Big\{\lVert\Upsilon_{S}^{*}(t-r)(G_{r}^{I,N})^{*}U_{r}^{N}\lVert_{-s,\sigma}^{2}+ \lVert\Upsilon_{S}^{*}(t-r)(G_{r}^{S})^{*}V_{r}^{N}\lVert_{-s,\sigma}^{2}\Big\}dr\\&+4\Big\lVert\int_{0}^{t}\Upsilon_{S}^{*}(t-r)d\widetilde{M}_{r}^{N}\Big\lVert_{-s,\sigma}^{2}, \\
	\lVert V_{t}^{N} \lVert_{-s,\sigma}^{2}&\leq\displaystyle5\lVert\Upsilon_{I}^{*}(t)V_{0}^{N} \lVert_{-s,\sigma}^{2}  +5t\int_{0}^{t}\Big\{\lVert\Upsilon_{I}^{*}(t-r)(G_{r}^{I,N})^{*}U_{r}^{N} \lVert_{-s,\sigma}^{2}+ \lVert\Upsilon_{I}^{*}(t-r)G_{r}^{S}V_{r}^{N} \lVert_{-s,\sigma}^{2}\Big\}dr\\+&5t\alpha^{2}\int_{0}^{t} \lVert \Upsilon_{I}^{*}(t-r)V_{r}^{N} \lVert_{-s,\sigma}^{2}dr +5 \Big\lVert\int_{0}^{t}\Upsilon_{I}^{*}(t-r)d\widetilde{L}_{r}^{N} \Big\lVert_{-s,\sigma}^{2},\\
	\displaystyle \lVert W_{t}^{N}\lVert^{2}_{-s,\sigma}&\leq 2\alpha^{2}t \int_{0}^{t} \lVert\Upsilon_{R}^{*}(t-r)V_{r}^{N}\lVert^{2}_{-s,\sigma}dr  +2\Big\lVert\int_{u}^{t}\Upsilon_{R}^{*}(t-r)d\widetilde{Y}_{r}^{N}\Big \lVert^{2}_{-s,\sigma}.
    \end{align*}

	From Corollarys \ref{uit} and \ref{cor20}, we have
	
	\begin{align*}
	\displaystyle \lVert U_{t}^{N}\lVert^{2}_{-s,\sigma}&\leq C4 e^{Ct}\lVert U_{0}^{N} \lVert^{2}_{-s,\sigma}+4Ct e^{Ct} \underset{y}{\sup}\lVert K(.,y)\lVert_{2+D,\sigma}^{2} \int_{0}^{t}\{ \lVert U_{r}^{N} \lVert_{-s,\sigma}^{2} +\lVert V_{r}^{N} \lVert^{2}_{-s,\sigma}\} dr\\&+4\Big\lVert\int_{0}^{t}\Upsilon_{S}^{*}(t-r)d\widetilde{M}_{r}^{N} \Big\lVert_{-s,\sigma}^{2},\\
	\displaystyle \lVert V_{t}^{N}\lVert^{2}_{-s,\sigma}&\leq 5 e^{Ct}\lVert V_{0}^{N} \lVert^{2}_{-s,\sigma}+5Ct e^{Ct}(1+\alpha^{2}) \underset{x}{\sup}\lVert K(x,.)\lVert_{2+D,\sigma}^{2} \int_{0}^{t}\{ \lVert U_{r}^{N} \lVert_{H^{-s}}+\lVert V_{r}^{N} \lVert^{2}_{-s,\sigma}\} dr\\&+5\Big\lVert\int_{0}^{t}\Upsilon_{I}^{*}(t-r)d\widetilde{L}_{r}^{N} \Big\lVert_{-s,\sigma}^{2},\\
	\displaystyle \lVert W_{t}^{N}\lVert^{2}_{-s,\sigma}&\leq 2C e^{Ct}\alpha^{2}t \int_{0}^{t} \lVert V_{r}^{N}\lVert^{2}_{-s,\sigma}dr  +2\Big\lVert\int_{0}^{t}\Upsilon_{R}^{*}(t-r)d\widetilde{Y}_{r}^{N}\Big \lVert^{2}_{-s,\sigma}.
\end{align*}

	Thus  from   Corrolary  \ref{ct} and Corollary \ref{cor20}, Lemma \ref{bbn} and  Assumption (H3), we have
	\begin{align}
	\displaystyle  \mathbb{E}(\underset{0\leq t\leq T}{\sup}\lVert U_{t}^{N}\lVert^{2}_{-s,\sigma})&\leq 4 e^{CT}\mathbb{E}(\lVert U_{0}^{N} \lVert_{-s,\sigma}^{2})+4CTe^{CT}\int_{0}^{T} \Big\{ \mathbb{E}(\underset{0\leq r\leq t}{\sup}\lVert U_{r}^{N} \lVert^{2}_{-s,\sigma})+\mathbb{E}(\underset{0\leq r\leq t}{\sup}\lVert V_{r}^{N} \lVert^{2}_{-s,\sigma})\Big\}dt \nonumber\\&+CT ,      \label{ttb}\\
	\displaystyle \mathbb{E}(\underset{0\leq t\leq T}{\sup}\lVert V_{t}^{N}\lVert^{2}_{-s,\sigma})&\leq 5e^{CT}\mathbb{E}(\lVert V_{0}^{N} \lVert^{2}_{-s,\sigma})\nonumber\\&+5TC(1 +\alpha^{2})e^{CT} \int_{0}^{T} \Big\{ \mathbb{E}(\underset{0\leq r\leq t}{\sup}\lVert U_{r}^{N} \lVert^{2}_{-s,\sigma})+\mathbb{E}(\underset{0\leq r\leq t}{\sup}\lVert V_{r}^{N} \lVert^{2}_{-s,\sigma})\Big\}dt+ CT, \label{ttbb}\\
	\displaystyle \mathbb{E}(\underset{0\leq t\leq T}{\sup}\lVert W_{t}^{N}\lVert^{2}_{-s,\sigma})&\leq 2\alpha^{2}Te^{CT} \int_{0}^{T} \Big\{\mathbb{E}(\underset{0\leq r\leq t}{\sup}\lVert W_{r}^{N} \lVert^{2}_{-s,\sigma})+\mathbb{E}(\underset{0\leq r\leq t}{\sup}\lVert V_{r}^{N} \lVert^{2}_{-s,\sigma})\Big\}dt+CT.\label{ttbbb}
\end{align}

Thus summing (\ref{ttb}), (\ref{ttbb}) and (\ref{ttbbb}) and applying Gronwall's lemma we deduce  the result from  the Proposition \ref{uo}.
\end{proof}
\subsubsection{Proof of Theorem \ref{c5}\label{tqa}}
We begin this subsection by showing the following result.
  \begin{prp}
	The  sequences of processes $U^{N}$, $V^{N}$ and $W^{N}$ are tight in $\mathbb{D}(\mathbb{R}_{+},H^{-s,\sigma})$.  \label{fff}
\end{prp} 
\begin{proof}
	We  establish the tightness of $U^{N}$ by showing that the conditions of Proposition \ref{c6} are satisfied.\vspace*{0.1cm}\\ 
	$-$  Based on Proposition \ref{tfg}, we deduce $(T1)$  by the same argument as used in  the proof of $(T1)$ in Proposition \ref{c8}. \vspace*{0.2cm}\\
	$-$ Proof of (T2). Let T>0, $\varepsilon_{1}, \varepsilon_{2}$ >0, $(\tau^{N})_{N}$ a family of stopping times with $\tau^{N}\leq T$.  We have\vspace*{0.1cm}\\ $\displaystyle U_{\tau^{N}+\varrho}^{N}-U_{\tau^{N}}^{N}=\displaystyle  (\Upsilon_{S}^{*}(\varrho)-I_{d})U_{\tau^{N}}^{N}  -\int_{\tau^{N}}^{\tau^{N}+\varrho}\Upsilon_{S}^{*}(\tau^{N}+\varrho-r)(G_{r}^{I,N})^{*}U_{r}^{N}dr  \\\hspace*{2.4cm}- \int_{\tau^{N}}^{\tau^{N}+\varrho} \Upsilon_{S}^{*}(\tau^{N}+\varrho-r)(G_{r}^{S})^{*}V_{r}^{N}dr+\int_{\tau^{N}}^{\tau^{N}+\varrho}\Upsilon_{S}^{*}(\tau^{N}+\varrho-r)d\widetilde{M}_{r}^{N},$
	\\\hspace*{2.4cm}= $(\Upsilon_{S}^{*}(\varrho)-I_{d})U_{\tau^{N}}^{N} - \displaystyle\displaystyle \int_{\tau^{N}}^{\tau^{N}+\varrho}\Upsilon_{S}^{*}(\tau^{N}+\varrho-r)J_{r}^{S,I,N}(U^{N},V^{N})dr\\\hspace*{2.5cm}+\int_{\tau^{N}}^{\tau^{N}+\varrho}\Upsilon_{S}^{*}(\tau^{N}+\varrho-r)d\widetilde{M}_{r}^{N},$ \vspace*{0.3cm}\\where $ J_{r}^{S,I,N}(U^{N},V^{N})=(G_{r}^{I,N})^{*}U_{r}^{N}+(G_{r}^{S})^{*}V_{r}^{N}.$ \vspace*{0.3cm}\\
	We find $\delta>0$ and $N_{0}\geq 1$ such that:  
	
	\begin{align}
	\underset{N\geq N_{0}}{\sup}\underset{\delta\geq\varrho}{\sup}\mathbb{P}(\lVert (\Upsilon_{S}^{*}(\varrho)-I_{d})U_{\tau^{N}}^{N} \lVert_{-s,\sigma}\geq\varepsilon_{1})&\leq \varepsilon_{2},\label{c21}\\
	\underset{N\geq N_{0}}{\sup}\underset{\delta\geq\varrho}{\sup}\mathbb{P}\left(\Big\lVert  \displaystyle\int_{\tau^{N}}^{\tau^{N}+\varrho}\Upsilon_{S}^{*}(\tau^{N}+\varrho-r)J_{r}^{S,I,N}(U^{N},V^{N})dr \Big\lVert_{-s,\sigma}\geq\varepsilon_{1}\right)&\leq \varepsilon_{2},\label{c22}\\
	\underset{N\geq N_{0}}{\sup}\underset{\delta\geq\varrho}{\sup}\mathbb{P}\left(\Big\lVert \displaystyle\int_{\tau^{N}}^{\tau^{N}+\varrho}\Upsilon_{S}^{*}(\tau^{N}+\varrho-r)d\widetilde{M}_{r}^{N}\Big\lVert_{-s,\sigma}\geq\varepsilon_{1}\right)&\leq \varepsilon_{2}.\label{ter}
\end{align}
	$1-$ Proof of $(\ref{c21})$.\\  Let us introduce a complete orthonormal basis in $H^{s,\sigma},$ of functions $(\varphi_{p})_{p\geq 1},$ $\varphi_{p}$ being of class $C^{\infty}$ with compact support,   and   $F_{m}$($m\in \mathbb{N}^{*}$) denotes  the  sub-space of $H^{s,\sigma}$ generated by  $ (\varphi_{p})_{1\leq p\leq m}.$  Let 
	$ (\Upsilon_{S}^{*}(\varrho)-I_{d})U^{N}_{t}\mid_{F_{m}}$ denotes the orthogonal projection of $ (\Upsilon_{S}^{*}(\varrho)-I_{d})U^{N}_{t}$ on the dual space of  $F_{m}.$ We have 
	\begin{align} 
		\mathbb{P}(\lVert (\Upsilon_{S}^{*}(\varrho)-I_{d})U^{N}_{\tau^{N}} \lVert_{-s,\sigma}\geq\varepsilon_{1})&\leq \mathbb{P}(\lVert (\Upsilon_{S}^{*}(\varrho)-I_{d})U_{\tau^{N}}^{N}\mid_{ F_{m}} \lVert_{-s,\sigma}\geq\frac{\varepsilon_{1}}{2})\nonumber\\&+\mathbb{P}(\lVert (\Upsilon_{S}^{*}(\varrho)-I_{d})U^{N}_{\tau^{N}}- (\Upsilon_{S}^{*}(\varrho)-I_{d})U_{\tau^{N}}^{N}\mid_{F_{m}} \lVert_{-s,\sigma}\geq\frac{\varepsilon_{1}}{2}) \label{erfv}
	\end{align}
Let us control each of the term of the above right hand side.\vspace*{0.2cm}\\
 $-$   One has
\begin{align}
	\mathbb{P}\big(\lVert(\Upsilon_{S}^{*}(\varrho)-I_{d})U^{N}_{\tau^{N}}- (\Upsilon_{S}^{*}(\varrho)-I_{d})U_{\tau^{N}}^{N}\mid_{F_{m}}\lVert_{-s,\sigma}\geq\frac{\varepsilon_{1}}{2}\big)&\leq\frac{4}{\varepsilon_{1}^{2}}\mathbb{E}\Big(\underset{0\leq t \leq T}{\sup}\sum\limits_{p> m}(U^{N}_{t},(\Upsilon_{S}(\varrho)-I_{d})\varphi_{p})^{2}\Big). \label{a243}
\end{align}
Furthermore the sequence $\Big(\underset{1\leq N }{\sup}\quad\underset{0\leq t \leq T}{\sup}\quad\underset{0\leq \varrho\leq \delta }{\sup}\sum\limits_{p> m}(U^{N}_{t},(\Upsilon_{S}(\varrho)-I_{d})\varphi_{p})^{2}\Big)$ converge towards $0$ as $m\longrightarrow \infty$, as the remainder of order $m$ of a uniformly convergent series of functions. Thus there exists $m_{0}\in \mathbb{N}^{*}$ independent of $N$ and $\varrho$ such that for any $m\geq m_{0},$ \begin{center}
	$\quad\underset{0\leq t \leq T}{\sup}\quad\underset{0\leq \varrho\leq \delta }{\sup}\sum\limits_{p> m}(U^{N}_{t},(\Upsilon_{S}(\varrho)-I_{d})\varphi_{p})^{2}<\varepsilon,$ for any $\varepsilon>0.$
\end{center} 
Moreover $\quad\underset{0\leq t \leq T}{\sup}\quad\underset{0\leq \varrho\leq \delta }{\sup}\sum\limits_{p> m}(U^{N}_{t},(\Upsilon_{S}(\varrho)-I_{d})\varphi_{p})^{2}$  is bounded by max($Ce^{C\delta},1)\underset{0\leq t \leq T}{\sup}\lVert U_{t}^{N}\lVert_{-s,\sigma}^{2},$ so uniformly integrable (see Proposition \ref{tfg}). Thus we deduce that there exists $m_{0}\in \mathbb{N}^{*}$ independent of $N$ and $\varrho$ such that for any $m\geq m_{0},$ 
\begin{align}
	\mathbb{E}\Big(\underset{0\leq t \leq T}{\sup}\sum\limits_{p> m}(U^{N}_{t},(\Upsilon_{S}(\varrho)-I_{d})\varphi_{p})^{2}\Big)<\varepsilon, \forall \varepsilon>0. \label{gty}
\end{align} 

	$-$ One has 
	\begin{align}
		 \displaystyle\mathbb{P}(\lVert (\Upsilon_{S}^{*}(\varrho)-I_{d})U_{\tau^{N}}^{N}\mid_{F_{m}} \lVert_{-s,\sigma}\geq\frac{\varepsilon_{1}}{2})\leq \frac{4}{\varepsilon_{1}^{2}}\mathbb{E}(\underset{0\leq t\leq T}{\sup}\lVert(\Upsilon_{S}^{*}(\varrho)-I_{d})U_{t}^{N}\mid_{F_{m}} \lVert_{-s,\sigma}^{2}). \label{ftz}
	\end{align}
 Furthermore according to Dynkin's formula and the contraction of $\Upsilon_{S}(t),$ one has
	\begin{align}
  \lVert(\Upsilon_{S}^{*}(\varrho)-I_{d})U_{t}^{N}\mid_{F_{m}} \lVert_{-s,\sigma}^{2}&=\sum\limits_{p= 1}^m\Big((\Upsilon_{S}^{*}(\varrho)-I_{d})U_{t}^{N},\varphi_{p}\Big)^{2},\nonumber\\&=\sum\limits_{p= 1}^m\Big(U_{t}^{N},(\Upsilon_{S}(\varrho)-I_{d})\varphi_{p}\Big)^{2},\nonumber\\&=\sum\limits_{p= 1}^m\Big(U_{t}^{N},\int_{0}^{\varrho}\Upsilon_{S}(r)\mathcal{Q}_S\varphi_{p}(.)dr\Big)^{2},\nonumber\\&\leq \varrho \underset{0\leq t \leq T}{\sup}\lVert U_{t}^{N} \lVert_{-s,\sigma}^{2}\sum\limits_{p= 1}^{m}\int_{0}^{\varrho}\lVert\Upsilon_{S}(r)\mathcal{Q}_S\varphi_{p}\lVert_{s,\sigma}^{2}dr,\nonumber\\&\leq \varrho^{2} \underset{0\leq t \leq T}{\sup}\lVert U_{t}^{N} \lVert_{-s,\sigma}^{2}\sum\limits_{p= 1}^{m}\lVert\mathcal{Q}_S\varphi_{p}\lVert_{s,\sigma}^{2}.
  \label{fgf}
\end{align}
Where $\mathcal{Q}_{S}$ in the infinitesimal generator of the operator $\Upsilon_{S}(t)$. Hence as from assumptions (H2), there exists $C>0$ such that 
\begin{center}
	$\sum\limits_{p= 1}^{m}\lVert\mathcal{Q}_S\varphi_{p}\lVert_{s,\sigma}^{2}\leq\sum\limits_{p= 1}^{m}\lVert\mathcal{Q}_S\varphi_{p}\lVert_{2+D,\sigma}^{2}\leq mC,$
\end{center}
 from (\ref{ftz}) and  (\ref{fgf}), we have 
 \begin{align}
 \mathbb{P}(\lVert (\Upsilon_{S}^{*}(\varrho)-I_{d})U_{\tau^{N}}^{N}\mid_{F_{m}} \lVert_{-s,\sigma}\geq\frac{\varepsilon_{1}}{2})&\leq\nonumber\\&\hspace*{-5cm}\leq\frac{4mC\underset{N\geq 1}{\sup}\mathbb{E}(\underset{0\leq t \leq T}{\sup}\lVert U_{t}^{N}\lVert_{-s,\sigma}^{2})}{\varepsilon_{1}^{2}}\varrho^{2}.\label{t1}
  \end{align}

 Hence (\ref{t1}) combined  with (\ref{gty}) and (\ref{erfv})
	  yields $(\ref{c21}).$\vspace*{0.3cm}\\
	Proof of $(\ref{c22})$.  Let $\ell >1$, we find $\delta>0$ such that $\tau^{N}+\delta\leq \ell T$ and such that $(\ref{c22})$ is satisfied. Since $\forall \varphi \in H^{s,\sigma}$, $\lVert \Upsilon(t)\varphi \lVert_{s,\sigma}\leq C e^{Ct}\lVert \varphi \lVert_{s,\sigma} $ then  form Proposition   \ref{tfg} and from  Corollary \ref{uit}, we have

	\begin{align*}
	\mathbb{P}\left(\left\lVert \displaystyle \int_{\tau^{N}}^{\tau^{N}+\varrho}\Upsilon_{S}^{*}(\tau^{N}+\varrho-r)J_{r}^{S,I,N}(U^{N},V^{N})dr \right\lVert_{-s,\sigma}\geq\varepsilon_{1}\right)&\leq\\&\hspace*{-7.5cm}\leq \frac{1}{\varepsilon_{1}^{2}}\mathbb{E}\left(\left\lVert  \displaystyle\int_{\tau^{N}}^{\tau^{N}+\varrho}\Upsilon_{S}^{*}(\tau^{N}+\varrho-r)J_{r}^{S,I,N}(U^{N},V^{N})dr\right\lVert_{-s,\sigma}^{2}\right),
	\\&\hspace*{-7.5cm}\leq \frac{\varrho}{\varepsilon_{1}^{2}}\displaystyle\mathbb{E}\left(\int_{\tau^{N}}^{\tau^{N}+\varrho}\lVert  \Upsilon_{S}^{*}(\tau^{N}+\varrho-r)J_{r}^{S,I,N}(U^{N},V^{N})\lVert_{-s,\sigma}^{2}dr\right),\\&\hspace*{-7.5cm}
	\leq\frac{\varrho C}{\varepsilon_{1}^{2}}e^{C\ell T} \underset{y}{\sup}\lVert K(.,y)\lVert_{s,\sigma}\displaystyle\mathbb{E}\left(\int_{\tau^{N}}^{\tau^{N}+\varrho}\{ \lVert U^{N}_{r} \lVert_{-s,\sigma}^{2}+\lVert V^{N}_{r} \lVert_{-s,\sigma}^{2} \}dr\right),\\&\hspace*{-7.5cm}\leq \frac{\delta^{2} C}{\varepsilon_{1}^{2}}e^{\ell T}\underset{N\geq 1}{\sup}\hspace*{0.1cm}\mathbb{E}( \underset{0\leq t\leq \ell T}{\sup}\{\lVert U^{N}_{t} \lVert_{-s,\sigma}^{2}+\underset{0\leq t\leq \ell T}{\sup}\lVert V^{N}_{t} \lVert_{-s,\sigma}^{2}\}), \\&\hspace*{-7.5cm}\leq \frac{\delta^{2} C}{\varepsilon_{1}^{2}}.
\end{align*}      
	
	So $(\ref{c22})$ is proved. \vspace*{0.25cm}\\
	- Proof of $(\ref{ter}).$ Let $\ell>1$, we find $\delta>0$ such that $\tau^{N}+\delta\leq \ell T$ and such that $(\ref{ter})$ is satisfied. From  Proposition \ref{cty}, we have 
	\begin{align*}
	\mathbb{P}\left(\Big\lVert \displaystyle\int_{\tau^{N}}^{\tau^{N}+\theta}\Upsilon_{S}^{*}(\tau^{N}+\theta-r)d\widetilde{M}_{r}^{N} \Big\lVert_{-s,\sigma}\geq\varepsilon_{1}\right)&\leq \frac{1}{\varepsilon_{1}^{2}}\mathbb{E}\left(\Big\lVert \displaystyle\int_{\tau^{N}}^{\tau^{N}+\theta}\Upsilon_{S}^{*}(\tau^{N}+\theta-r)d\widetilde{M}_{r}^{N} \Big\lVert_{-s,\sigma}^{2}\right),\\&\leq\frac{ C}{\varepsilon_{1}^{2}} \delta.
		\end{align*}
	 Hence $(\ref{ter})$ is proved. 
	
\end{proof}

To  establish the system of limiting equations of  all converging subsequences of $(U^{N},V^{N},W^{N})_{N\geq 1}, $ we will need the next  Lemma.
\begin{lem}
	For any $ t\geq0$, $\varphi\in H^{s,\sigma}(\mathbb{R}^{d})$,  as $N\longrightarrow \infty$, \\\hspace*{5cm} $ \displaystyle\int_{0}^{t}\mathbb{E}\Big(\lVert [G_{r}^{I,N}-G_{r}^{I}]\Upsilon_{S}(t-r)\varphi \lVert_{s,\sigma}^{2}\Big)dt\longrightarrow 0.$\label{f11}
\end{lem}

\begin{proof} Since $H^{s,\sigma}$ is a Banach algebra (see Remark \ref{r65})   and $H^{2+D,\sigma}\hookrightarrow H^{s,\sigma}$ (since $s<2+D),$  and  $\lVert \Upsilon(t)\varphi\lVert_{s,\sigma}\leq C e^{Ct} \lVert \varphi \lVert_{s,\sigma}$, \vspace*{0.2cm} \\
	\hspace*{2cm}$\Big\lVert \Upsilon(t-r)\varphi \displaystyle\Big(\mu_{r}^{I,N}-\mu_{r}^{I},K\Big)\Big\lVert_{s,\sigma}\leq C\lVert \varphi \lVert_{s,\sigma}\Big\lVert \displaystyle \Big(\mu_{r}^{I,N}-\mu_{r}^{I},K\Big) \Big\lVert_{2+D,\sigma}$. \vspace*{0.2cm}\\On the other hand  \\
	$\Big\lVert \displaystyle\Big(\mu_{r}^{I,N}-\mu_{r}^{I},K\Big) \Big\lVert_{2+D,\sigma}^{2}=\displaystyle\sum\limits_{\lvert\eta\lvert\leq 2+D}\int_{\mathbb{R}^{d}}(1+\lvert x\lvert^{2\sigma})^{-1}\Big\lvert\int_{\mathbb{R}^{d}} D^{\eta}K(x,y)(\mu_{r}^{I,N}-\mu_{r}^{I})(dy)\Big\lvert^{2}dx$,\vspace*{0.2cm}\\
	furthermore from Asumptions (H2) the map $y\in \mathbb{R}^{d}\mapsto D^{\eta}K(x,y)$ is continuous and bounded. So we deduce from  Theorem \ref{th1} that \vspace*{0.1cm}\\\hspace*{3.2cm} $\displaystyle\Big\lvert\int_{\mathbb{T}^{2}} D^{\eta}K(x,y)(\mu_{r}^{I,N}-\mu_{r}^{I})(dy)\Big\lvert^{2}\xrightarrow{P}0$.\\ According to Lebesgue's dominated convergence theorem, $\mathbb{E}\Big(\Big\lVert \displaystyle\Big(\mu_{r}^{I,N}-\mu_{r}^{I}, K\Big) \Big\lVert_{2+D,\sigma}^{2}\Big)\longrightarrow 0$, as $N\longrightarrow\infty$. Thus \vspace*{0.2cm}\\ $\hspace*{2.5cm}\displaystyle\int_{0}^{t}\mathbb{E}\Big(\lVert [G_{r}^{I,N}-G_{r}^{I}]\Upsilon(t-r)\varphi \lVert_{2+D,\sigma}^{2}\Big)dt\\\hspace*{5cm}\leq C\lVert  \varphi \lVert_{s,\sigma}^{2}\displaystyle\int_{0}^{t}\mathbb{E}\Big(\Big\lVert \displaystyle\Big(\mu_{r}^{I,N}-\mu_{r}^{I}, K\Big) \Big\lVert_{2+D,\sigma}^{2}\Big)dr\vspace*{0.1cm}\\\hspace*{5cm}\longrightarrow0$, as $N\longrightarrow\infty$.\\
	Hence the result.
\end{proof}
The next Proposition establishes the evolution equations of all limit points  $ (U,V,W)$  of the sequence $ (U^{N},V^{N},W^{N})$.                

\begin{prp}
	All limit points  $ (U,V,W)$  of the sequence $ (U^{N},V^{N},Z^{N})$  satisfy 
	\begin{align}
	\displaystyle U_{t}&=\displaystyle  \Upsilon_{S}^{*}(t)U_{0}  -\int_{0}^{t}\Upsilon_{S}^{*}(t-r)(G_{r}^{I})^{*}U_{r}dr - \int_{0}^{t}\Upsilon_{S}^{*}(t-r)(G_{r}^{S})^{*} V_{r}dr + \int_{0}^{t}\Upsilon_{S}^{*}(t-r)d\mathcal{M}_{r}^{1},\label{ll3} \\
	 \displaystyle V_{t}&=\displaystyle  \Upsilon_{I}^{*}(t)V_{0}  +\int_{0}^{t}\Upsilon_{I}^{*}(t-r)(G_{r}^{I})^{*}U_{r} dr + \int_{0}^{t}\Upsilon_{I}^{*}(t-r)(G_{r}^{S})^{*}V_{r}dr -\alpha\int_{0}^{t}\Upsilon_{I}^{*}(t-r)V_{r}dr\nonumber\\&+\int_{0}^{t}\Upsilon_{I}^{*}(t-r)d\mathcal{M}_{r}^{2},\label{ll4}\\  \displaystyle W_{t}&=\displaystyle     \alpha\int_{0}^{t}\Upsilon_{R}^{*}(t-r)V_{r}dr+\int_{0}^{t}\Upsilon_{R}^{*}(t-r)d\mathcal{M}_{r}^{3}.\label{ll5}
\end{align}
\end{prp}
\begin{proof}
	We prove this Proposition by taking the weak limit   in the equations  (\ref{c15}) and (\ref{c16}) and (\ref{c17}). Note first that from Propositions \ref{fff},  there exists a subsequence along which the sequences $(U^{N},V^{N},W^{N})_{N}$ converges in law towards  $(U,V,W)$. For any $\varphi\in H^{s,\sigma},$ one has
	\vspace*{-0.14cm}
	\begin{align*}
		\displaystyle(\Upsilon_{S}^{*}(t)U_{0}^{N},\varphi)+\int_{0}^{t}\big(\Upsilon_{S}^{*}(t-r)\varphi, d\widetilde{M}_{r}^{N}\big)&= 
		\displaystyle (U_{t}^{N},\varphi)  +\int_{0}^{t}(\Upsilon_{S}^{*}(t-r)(G_{r}^{I})^{*}U_{r}^{N},\varphi)dr
		\\&\hspace*{-5.5cm}+ \int_{0}^{t} (\Upsilon_{S}^{*}(t-r)(G_{r}^{S})^{*}V_{r}^{N},\varphi)dr+\int_{0}^{t}(\Upsilon_{S}^{*}(t-r)[(G_{r}^{I,N})^{*}-(G_{r}^{I})^{*}]U_{r}^{N},\varphi)dr, 
	\end{align*}
	\vspace*{-0.3cm}
	\begin{align*}
		\displaystyle( \Upsilon_{I}^{*}(t)V_{0}^{N},\varphi) +\int_{0}^{t}\big(\Upsilon_{I}^{*}(t-r)\varphi,d\widetilde{L}_{r}^{N}\big)&=(V_{t}^{N},\varphi) -\int_{0}^{t}(\Upsilon_{I}^{*}(t-r)(G_{r}^{I})^{*}U_{r}^{N},\varphi)dr\\&\hspace*{-5.5cm}- \int_{0}^{t} (\Upsilon(t-r)[(G_{r}^{S})^{*}-\alpha ]V_{r}^{N},\varphi)dr-\int_{0}^{t}([\Upsilon_{I}^{*}(t-r)(G_{r}^{I,N})^{*}-(G_{r}^{I})^{*}]U_{r}^{N},\varphi)dr,  
	\end{align*}
\begin{align}
	\int_{0}^{t}\big(\Upsilon_{I}^{*}(t-r)\varphi,d\widetilde{Y}_{r}^{N}\big)&=(W_{t}^{N},\varphi) +\alpha \int_{0}^{t} (\Upsilon_{R}^{*}(t-r)V_{r}^{N},\varphi)dr.  \label{cvn}
\end{align}
	Thus in view of (\ref{cvn}), it is enough to show that $(U,V)$ satisfy (\ref{ll3}) and (\ref{ll4}). \\Hence, we have 
	\vspace*{-0.3cm}
	\begin{align*}
		\displaystyle(\Upsilon_{S}^{*}(t)U_{0}^{N},\varphi)+\int_{0}^{t}\big(\Upsilon_{S}^{*}(t-r)\varphi,d\widetilde{M}_{r}^{N}\big)&=\Psi_{1,t,\varphi}\big(U^{N},V^{N}\big)\\&+\int_{0}^{t}([\Upsilon_{S}^{*}(t-r)(G_{r}^{I,N})^{*}-(G_{r}^{I})^{*}]U_{r}^{N},\varphi)dr, 
	\end{align*}
	\vspace*{-0.3cm}
	\begin{align*}
		\displaystyle (\Upsilon_{I}^{*}(t)V_{0}^{N},\varphi) +\int_{0}^{t}\big(\Upsilon_{I}^{*}(t-r)\varphi,d\widetilde{L}_{r}^{N}\big)&=\Psi_{2,t,\varphi}\big(U^{N},V^{N}\big) \\&-\int_{0}^{t}([\Upsilon_{I}^{*}(t-r)(G_{r}^{I,N})^{*}-(G_{r}^{I})^{*}]U_{r}^{N},\varphi)dr.  
	\end{align*}
	
	With 
	\begin{align*}
	\displaystyle\Psi_{1,t,\varphi}\big(U^{N},V^{N}\big)&=\displaystyle (U_{t}^{N},\varphi) +\int_{0}^{t}(\Upsilon_{S}^{*}(t-r)(G_{r}^{I})^{*}U_{r}^{N},\varphi)dr
	\\&+ \int_{0}^{t} (\Upsilon_{S}^{*}(t-r)(G_{r}^{S})^{*}V_{r}^{N},\varphi)dr,\\
	\displaystyle\Psi_{2,t,\varphi}\big(U^{N},V^{N}\big)&=\displaystyle (V_{t}^{N},\varphi)-\int_{0}^{t}(\Upsilon_{I}^{*}(t-r)(G_{r}^{I})^{*}U_{r}^{N},\varphi)dr\\&- \int_{0}^{t} (\Upsilon_{I}^{*}(t-r)[(G_{r}^{S})^{*}-\alpha ]V_{r}^{N},\varphi)dr.
		\end{align*}
	
	Furthermore.\\$1-$ From Lemma \ref{f11} and  Proposition \ref{tfg}, $\displaystyle\int_{0}^{t}\left(U_{r}^{N} ,[G_{r}^{I,N}-G_{r}^{I}]\Upsilon(t-r)\varphi \right)dr\longrightarrow 0$ in $L^{1}(\mathbb{P}),$ locally unformly in t.\\ Indeed, $\mathbb{E}\left(\Big \lvert\underset{0\leq t \leq T}{\sup} \displaystyle\int_{0}^{t}\left(U_{r}^{N} ,[G_{r}^{I,N}-G_{r}^{I}]\Upsilon(t-r)\varphi \right)dr \Big \lvert\right)\leq\\\hspace*{5cm}\leq \sqrt{T} \underset{N\geq 1}{\sup}\hspace*{0.04cm}\mathbb{E}(\underset{0\leq t \leq T}{\sup}\lVert U_{t}^{N}\lVert_{-s,\sigma}^{2})^{\frac{1}{2}}\displaystyle\Big[\int_{0}^{T}\mathbb{E}(\lVert [G_{r}^{I,N}-G_{r}^{I}]\Upsilon(t-r)\varphi\lVert_{s,\sigma}^{2})dr\Big]^{\frac{1}{2}}.
	$\vspace*{0.12cm}\\
	2- Using (\ref{le1e}) in Proposition \ref{prp215}, it is easy to see that the maps 
	$(\Psi_{1,.,\varphi},\Psi_{2,.,\varphi})$  is continuous from $ [\mathbb{D}(\mathbb{R}_{+},H^{-s,\sigma}) ]^{2}$ into  $\mathbb{C}(\mathbb{R}_{+},\mathbb{R}^{2}).$ Thus as $(U^{N},V^{N})$ converges in law in  $[\mathbb{D}(\mathbb{R}_{+},H^{-s,\sigma}) ]^{2}$ towards $(U,V)$, then $\Big(\Psi_{1,.,\varphi}(U^{N},V^{N}),\Psi_{2,.,\varphi}(U^{N},V^{N})\Big)$ converges in law towards \\ $\Big(\Psi_{1,.,\varphi}(U,V),\Psi_{2,.,\varphi}(U,V)\Big).$\\\\
	3- $\Big(\displaystyle (\Upsilon_{S}^{*}(.)U_{0}^{N},\varphi) +\int_{0}^{.}\big(\Upsilon_{S}^{*}(.-r)\varphi,d\widetilde{M}_{r}^{N}\big)$, $\displaystyle  (\Upsilon_{I}^{*}(.)V_{0}^{N},\varphi) +\int_{0}^{.}\big(\Upsilon_{I}^{*}(.-r)\varphi,d\widetilde{L}_{r}^{N}\big)\Big)$ converges in law towards  $\Big(\displaystyle (\Upsilon_{S}^{*}(.)U_{0},\varphi) +\int_{0}^{.}\big(\Upsilon_{S}^{*}(.-r)\varphi,d\mathcal{M}^{1}_{r}\big), \displaystyle  (\Upsilon_{I}^{*}(.)V_{0},\varphi) +\int_{0}^{.}\big(\Upsilon_{I}^{*}(.-r)\varphi,d\mathcal{M}^{2}_{r}\big)\Big)$ in $(\mathbb{D}(\mathbb{R}_{+},H^{-s,\sigma}))^{2}$ since \\ $\Big(\displaystyle (\Upsilon_{S}^{*}(.)U_{0}^{N},\varphi),\quad (\Upsilon_{I}^{*}(.)V_{0}^{N},\varphi),  \int_{0}^{.}\big(\Upsilon_{S}^{*}(.-r)\varphi,d\widetilde{M}_{r}^{N}\big),  \displaystyle   \int_{0}^{.}\big(\Upsilon_{I}^{*}(.-r)\varphi,d\widetilde{L}_{r}^{N}\big)\Big)$ converges in law towards  $\Big(\displaystyle (\Upsilon_{S}^{*}(.)U_{0},\varphi), (\Upsilon_{I}^{*}(.)V_{0},\varphi), \int_{0}^{.}\big(\Upsilon_{S}^{*}(.-r)\varphi,d\mathcal{M}^{1}_{r}\big), \displaystyle   \int_{0}^{.}\big(\Upsilon_{I}^{*}(.-r)\varphi,d\mathcal{M}^{2}_{r}\big)\Big)$ in\\ $(\mathbb{C}(\mathbb{R}_{+},H^{-s,\sigma}))^{2}\times (\mathbb{D}(\mathbb{R}_{+},H^{-s,\sigma}))^{2}$ , which in turn follows from the fact that  \\
	$\Big ((\Upsilon_{S}^{*}(.)U_{0}^{N},\varphi),(\Upsilon_{I}^{*}(.)V_{0}^{N},\varphi)\Big)$ converges in law towards  $\Big ((\Upsilon_{S}^{*}(.)U_{0},\varphi),(\Upsilon_{I}^{*}(.)V_{0},\varphi)\Big)$ in \\ $(\mathbb{C}(\mathbb{R}_{+},H^{-s,\sigma}))^{2}$(see Remark $\ref{rm59} )$ and  $\displaystyle\Big(\int_{0}^{.}\big(\Upsilon_{S}^{*}(.-r)\varphi,d\widetilde{M}_{r}^{N}\big),\int_{0}^{.}\big(\Upsilon_{I}^{*}(.-r)\varphi,d\widetilde{L}_{r}^{N}\big)\Big)$  converges in law towards $\displaystyle\Big(\int_{0}^{.}\big(\Upsilon_{S}^{*}(.-r)\varphi,dW^{1}_{r}\big),\int_{0}^{.}\big(\Upsilon_{I}^{*}(.-r)\varphi,dW^{2}_{r}\big)\Big)$ in $(\mathbb{D}(\mathbb{R}_{+},H^{-s,\sigma}))^{2}$(which follows from Proposition \ref{ffg})
	and $\Big ((\Upsilon_{S}^{*}(.)U_{0}^{N},\varphi),(\Upsilon_{I}^{*}(.)V_{0}^{N},\varphi)\Big)$ is globally independant of  $\displaystyle\Big(\int_{0}^{.}\big(\Upsilon_{S}^{*}(.-r)\varphi,d\widetilde{M}_{r}^{N}\big),\int_{0}^{.}\big(\Upsilon_{I}^{*}(.-r)\varphi,d\widetilde{L}_{r}^{N}\big)\Big)$.\vspace*{0.1cm}\\Thus from 1-, 2-, and 3-, we obtain the result of the statement.
\end{proof}
From Proposition \ref{c10} we deduce that  all limit points $(U,V,W)$ of $ (U^{N},V^{N},W_{N})_{N\geq1} $ are elements of  $( \mathbb{C}(\mathbb{R}_{+},H^{-s}))^{3}$, thus we end the proof of Theorem \ref{c5} by showing that   the system of equations (\ref{ll3}) and (\ref{ll4}) and (\ref{ll5}) admits a unique solution $(U,V,W)\in( \mathbb{C}(\mathbb{R}_{+},H^{-s}))^{3}$.
\begin{prp}
	Suppose that $(U^{1},V^{1},W^{1})$  and $ (U^{2},V^{2},W^{2})$   are solutions to equations (\ref{ll3}), (\ref{ll4}) and (\ref{ll5}) with $ (U_{0}^{1},V_{0}^{1})=(U_{0}^{2},V_{0}^{2})$  then  $(U^{1},V^{1},W^{1})$ = $ (U^{2},V^{2},W^{2}).$
\end{prp}
\begin{proof}
	All we need  to show is that the system of equations (\ref{ll3}) and (\ref{ll4}) admits a unique solution. Thus we have \\\\
	$U_{t}^{1}-U_{t}^{2}=-\displaystyle \int_{0}^{t} \Upsilon_{S}^{*}(t-r)(G_{r}^{I})^{*}(U_{r}^{1}-U_{r}^{2})dr-\int_{0}^{t}\Upsilon_{S}^{*}(t-r)(G_{r}^{S})^{*}(V_{r}^{1}-V_{r}^{2})dr,$\\Hence \\$\lVert U_{t}^{1}-U_{t}^{2}\lVert_{H^{-s}}\leq\displaystyle \int_{0}^{t} \lVert\Upsilon_{S}^{*}(t-r)(G_{r}^{I})^{*}(U_{r}^{1}-U_{r}^{2})\lVert_{-s,\sigma}dr+\int_{0}^{t}\lVert\Upsilon_{S}^{*}(t-r)(G_{r}^{S})^{*}(V_{r}^{1}-V_{r}^{2})\lVert_{-s,\sigma}dr.$\vspace*{0.2cm}\\ So from Corollary \ref{uit},  we deduce that
	\begin{align}
\lVert U_{t}^{1}-U_{t}^{2}\lVert_{-s,\sigma}\leq C\underset{y}{\sup}\lVert K(.,y)\lVert_{2+D,\sigma}\displaystyle \int_{0}^{t}\{\lVert U_{r}^{1}-U_{r}^{2}\lVert_{-s,\sigma}  +  \lVert V_{r}^{1}-V_{r}^{2} \lVert_{-s,\sigma}\}dr.\label{tbh}
\end{align}
Similarly, we obtain 
	\begin{align}
	\lVert V_{t}^{1}-V_{t}^{2}\lVert_{-s,\sigma}\leq C( \underset{x}{\sup}\lVert K(x,.)\lVert_{2+D,\sigma}+\alpha)\displaystyle \int_{0}^{t}\{\lVert U_{r}^{1}-U_{r}^{2}\lVert_{-s,\sigma}  + \lVert V_{r}^{1}-V_{r}^{2} \lVert_{-s,\sigma}\}dr. \label{gtb}
\end{align}
	Summing $(\ref{tbh})$, $ (\ref{gtb}) $  and applying Gronwall's lemma we obtain that  $(U^{1},V^{1})=(U^{2},V^{2}),$ and the proof is complete.
\end{proof} 
	\section{Appendix}
	In this Appendix we prove the following Lemmas.\\
	\begin{lem}
		For any  nonnegative integer $m_{1}> m_{2}\geq0$, and any real $0<\sigma<\sigma'$, the following embedding is compact
		\begin{center}
		 $W_{0}^{m_{1},\sigma}(\mathbb{R}^{d})\hookrightarrow W_{0}^{m_{2},\sigma'}(\mathbb{R}^{d}).$
		\end{center} \label{la}
	\end{lem}
\begin{proof}
	To prove this Lemma it is enough to show that for any sequence $(\varphi_{n})_{n}$ of $ W_{0}^{m_{1},\sigma}(\mathbb{R}^{d})$ which weakly converges towards $0$ in $W_{0}^{m_{1},\sigma}(\mathbb{R}^{d}),$ strongly converges in $W_{0}^{m_{2},\sigma'}(\mathbb{R}^{d}).$\\ Let $\overline{B}(R)=\{x\in \mathbb{R}^{d}/\lvert x\lvert \leq R\}$,
	one has:\\ $\lVert \varphi_{n}\lVert_{m_{2},\sigma'}^{2}=\ \sum\limits_{\lvert \gamma\lvert\leq m_{2}}\displaystyle\int_{\mathbb{R}^{d}}\frac{\lvert D^{\gamma}\varphi_{n}(x)\lvert^{2}}{1+\lvert x\lvert^{2\sigma'}}dx= \sum\limits_{\lvert \gamma\lvert\leq m_{2}}\displaystyle\int_{\overline{B}(R)}\frac{\lvert D^{\gamma}\varphi_{n}(x)\lvert^{2}}{1+\lvert x\lvert^{2\sigma'}}dx+\sum\limits_{\lvert \gamma\lvert\leq m_{2}}\displaystyle\int_{\overline{B}^{c}(R)}\frac{\lvert D^{\gamma}\varphi_{n}(x)\lvert^{2}}{1+\lvert x\lvert^{2\sigma'}}dx,$ \\furthermore, since the function $R\in]1,\infty[\mapsto\displaystyle\frac{1+R^{2\sigma}}{1+R^{2\sigma'}}$ is nonincreasing, 
	                 \begin{align*}
	\sum\limits_{\lvert \gamma\lvert\leq m_{2}}\displaystyle\int_{\overline{B}^{c}(R)}\frac{\lvert D^{\gamma}\varphi_{n}(x)\lvert^{2}}{1+\lvert x\lvert^{2\sigma'}}dx&=\sum\limits_{\lvert \gamma\lvert\leq m_{2}}\displaystyle\int_{\overline{B}^{c}(R)}\frac{1+\lvert x\lvert^{2\sigma}}{1+\lvert x\lvert^{2\sigma'}}\frac{\lvert D^{\gamma}\varphi_{n}(x)\lvert^{2}}{1+\lvert x\lvert^{2\sigma}}dx,\\&\leq \frac{1+R^{2\sigma}}{1+R^{2\sigma'}}\sum\limits_{\lvert \gamma\lvert\leq m_{2}}\displaystyle\int_{\overline{B}^{c}(R)}\frac{\lvert D^{\gamma}\varphi_{n}(x)\lvert^{2}}{1+\lvert x\lvert^{2\sigma}}dx,\\&\leq \frac{1+R^{2\sigma}}{1+R^{2\sigma'}} \underset{n\geq1}{\sup}\lVert  \varphi_{n}\lVert^{2}_{m_{1},\sigma}\underset{R\longrightarrow+\infty}{\longrightarrow}0.
		\end{align*}
	On the other hand since $W_{0}^{m_{1},\sigma}(\mathbb{R}^{d})\hookrightarrow W_{0}^{m_{1},\sigma'}(\overline{B}(R))\hookrightarrow W_{0}^{m_{2},\sigma'}(\overline{B}(R)),$ where the second embedding is compact, $W_{0}^{m_{1},\sigma}(\mathbb{R}^{d})\hookrightarrow W_{0}^{m_{2},\sigma'}(\overline{B}(R))$ is compact (see Remark 6.3 in \cite{ac}). Thus $\sum\limits_{\lvert \gamma\lvert\leq m_{2}}\displaystyle\int_{\overline{B}(R)}\frac{\lvert D^{\gamma}\varphi_{n}(x)\lvert^{2}}{1+\lvert x\lvert^{2\sigma'}}dx\underset{n\longrightarrow+\infty}{\longrightarrow}0$.
\end{proof}
\begin{cor}
	For any  non integer   $1+D<s<2+D$ and  $1+D<s'=\frac{s+1+D}{2}<s<2+D$ and any $0<\sigma<\sigma',$ the following embedding is compact. \label{co42}
		\begin{center}
		$H^{s,\sigma}(\mathbb{R}^{d})\hookrightarrow H^{s',\sigma'}(\mathbb{R}^{d}).$
	\end{center}
\label{ca}
\end{cor}
\begin{proof}
	We use the definition by interpolation of the space $H^{s,\sigma}(\mathbb{R}^{d})$ to prove this result.We Prove this result for $d=2$, simlar arguments allow us to obtain the result for other values of $d$. Since $2<s'=\frac{s+2}{2}<s<3$. There exists $\rho\in (1/2,1),$  sucht that $s'=3(1-\rho)+2\rho$ and  $s=4(1-\rho)+2\rho.$ Furthemore,  \vspace*{0.1cm}\\\hspace*{1.5cm} $\big(W_{0}^{3,\sigma'}(\mathbb{R}^{2}),W_{0}^{2,\sigma'}(\mathbb{R}^{2}) \big)_{\rho,2}=H^{s',\sigma'}(\mathbb{R}^{2})$ and $\big(W_{0}^{4,\sigma}(\mathbb{R}^{2}),W_{0}^{2,\sigma}(\mathbb{R}^{2}) \big)_{\rho,2}=H^{s,\sigma}(\mathbb{R}^{2}),$ \vspace*{0.2cm}\\ see \cite{loff} or \cite{tri} for the explicit definition of the real interpolation space $(.,.)\rho,q.$\\Thus as from Lemma \ref{la} the embeddings   $W_{0}^{4,\sigma}(\mathbb{R}^{2})\hookrightarrow W_{0}^{3,\sigma'}(\mathbb{R}^{2})$ is compact, we deduce from  Corollary 3.5 in \cite{fer} that the following embedding is compact.
	 \begin{center}
		 $H^{s,\sigma}(\mathbb{R}^{d})=\big(W_{0}^{4,\sigma}(\mathbb{R}^{2}),W_{0}^{2,\sigma}(\mathbb{R}^{2}) \big)_{\rho,2}\hookrightarrow\big(W_{0}^{3,\sigma'}(\mathbb{R}^{2}),W_{0}^{2,\sigma'}(\mathbb{R}^{2}) \big)_{\rho,2}=H^{s,\sigma'}(\mathbb{R}^{2}).$
	\end{center}
\end{proof}

\begin{lem}
	Under the Assumption (H2), for any $A\in \{S,I,R\},$ the Markovian semi-group generated by the operator $\mathcal{Q}_{A},$ is such that there exists a constant $C>0$, such that 
	\begin{center}
	 $\displaystyle\underset{y}{\sup}\Big(\int_{\mathbb{R}^{d}}\Upsilon_{A}(t)(x,y)dx\Big)\leq e^{Ct}.$ \label{la2}
	\end{center}
\end{lem}
\begin{proof}
	We recall that   $\{X_{t}^{A}, t\geq\}$ is the Markov process having the operator $\mathcal{Q}_{A}$ as its infinitesimal generator.
Let  $P_{A}(t)(y)=\displaystyle \int_{\mathbb{R}^{d}}\Upsilon_{A}(t)(x,y)dx,$ using Dynkin's formula it is easy to see that
 \begin{align}
 \displaystyle\Upsilon_{A}(t)(x,y)=\delta_{0}(x-y)+\int_{0}^{t}(\mathcal{Q}_{A})^{*}_{y}\Upsilon_{A}(r)(x,y)dr,\label{sde}
\end{align}

 thus integrating (\ref{sde}) over $x$, we obtain  \vspace*{0.1cm}\\\hspace*{5cm}
	$\displaystyle P_{A}(t)(y)=1+\int_{0}^{t}(\mathcal{Q}_{A})^{*}_{y}P_{A}(r)(y)dr.$\\ Furthermore,\\ $\displaystyle\frac{\partial}{\partial t} P_{A}(t)(y)= - \sum\limits_{1\leq \ell\leq d}m_{\ell}(A,y)\frac{\partial}{\partial y_{\ell}}P_{A}(t)(y) + \frac{1}{2}\sum\limits_{1\leq \ell,u\leq d}\frac{\partial}{\partial y_{\ell}}(\theta \theta^{t})_{\ell,u}(A,y)\frac{\partial}{\partial y_{u}}P_{A}(t)(y) \\\hspace*{9cm}+\frac{1}{2}\sum\limits_{1\leq \ell,u\leq d}\frac{\partial}{\partial y_{u}}(\theta \theta^{t})_{\ell,u}(A,y)\frac{\partial}{\partial y_{\ell}}P_{A}(t)(y)\\\hspace*{2.5cm} +\frac{1}{2}\sum\limits_{1\leq \ell,u\leq 2}(\theta\theta^{t})_{\ell,u}(A,y)\frac{\partial^{2}}{\partial y_{u}y_{\ell}}P_{A}(t)(y)\vspace*{0.18cm}\\\hspace*{2.6cm}+\Big[-\sum\limits_{1\leq \ell\leq d}\frac{\partial}{\partial y_{\ell}}m_{l}(A,y) +\frac{1}{2}\sum\limits_{1\leq \ell,u\leq d}\frac{\partial^{2}}{\partial y_{u}\partial y_{\ell}}(\theta\theta^{t})_{\ell,u}(A,y)\Big]P_{A}(t)(y).\vspace*{0.18cm}$\\Consequently, $P_{A}(t)$ is the solution of the following system.\vspace*{0.18cm}\\ $\displaystyle\frac{\partial}{\partial t} P_{A}(t)(y)=\sum\limits_{\ell}F_{\ell}^{m,\theta}(A,y)\frac{\partial}{\partial y_{\ell}}P_{A}(t)(y)+\frac{1}{2}\sum\limits_{1\leq \ell,u\leq d}(\theta\theta^{t})_{\ell,u}(A,y)\frac{\partial^{2}}{\partial y_{u}y_{\ell}}P_{A}(t)(y) \vspace*{0.18cm}\\\hspace*{2.5cm}+H(A,y)P_{A}(t)(y)\vspace*{0.18cm}\\\hspace*{1.9cm}=\mathcal{G}_{A}P_{A}(t)(y)+H(A,y)P_{A}(t)(y)\vspace*{0.18cm}\\P_{A}(0)(y)= 1.$\\\\ Thus fix $T>0$,  according to the Feyman-Kac formula, for any $t\in[0,T]$, we have \vspace*{0.18cm}\\
	$\hspace*{1cm}P_{A}(t)(y)=P_{A}(T-t_{1})(y)=\mathbb{E}\Big(exp\Big(-\displaystyle\int_{t_{1}}^{T}H(A,Y_{r})dr\Big)/Y_{t_{1}}=y\Big),$ (with $t+t_{1}=T$). $\vspace*{0.18cm}\\$
	where $\{Y_{t},t\geq 0\}$ is the markovian processes having $\mathcal{G}_{A}$ as the infinitesimal generator.\\So as from assumption (H2) the function $H$ is bounded, for any $y\in \mathbb{R}^{d},$ we have\[ P_{A}(t)(y)=\displaystyle\int_{\mathbb{R}^{d}}\Upsilon_{A}(t)(x,y)dx\leq e^{Ct}.\label{tvc}\] .
\end{proof}

\begin{rmq}
	Since it is easy to adap without difficulty the proof of Theorem 5.4 of \cite{ac} to the space $W_{0}^{m,\sigma}$ $(m\in \mathbb{N}),$  by following the proof of Theorem 5.23 of \cite{ac}, we prove easily that for any integer $m>d/2$, the space $W_{0}^{m,\sigma}$ is a Banach algebra. Furthemore using the result on the complex interpolation of \cite{lof} (Theorem 4) and \cite{yi}(subsection 10.2), we conclude that for any real number $s>d/2$, the space $H^{s,\sigma}$ is a Banach algebra.
	\label{r65}
\end{rmq}


\begin{thebibliography}{99}
	\bibitem{ac} R.A. Adams. \textit{Sobolev Spaces}. Academic Press, 1975.
	\bibitem{bc}
	D. Aldous . Stopping times and tightness. \textit{The Annals of Probability} 6(2), 335-340, 1978
	\bibitem{cc} L.J.S. Allen, B.M Bolker, Y. Lou and A.L. Nevai Asymptotic profiles of the steady states for an SIS epidemic reaction diffusion model, \textit{Discrete Contin. Dyn. Syst.} Ser. A 21, 1-20, 2008.
	\bibitem{dc} H. Andersson, T. Britton. \textit{Stochastic epidemic models and their statistical analysis}. Springer Lecture Notes in Statistics. Springer, New York, 2000.
	\bibitem{ec} H. Bahouri, J. Y. Chemin and R. Danchin. \textit{Fourier Analysis and Nonlinear Partial Differential Equations}, Springer, 2011.
	\bibitem{fc} P. Billingsley, \textit{Convergence of Probability Measures}, 2nd edn. Wiley, New York, 1999.
	\bibitem{sb} S. Bowong, A. Emakoua. and E Pardoux. A Spatial Stochastic Epidemic Model: Law of Large Numbers and Central Limit Theorem. 
	arXiv:2007.0663v1. [math.PR] 13 jul 2020
	\bibitem{hc} T. Britton, E. Pardoux. Stochastic epidemic in a homogeneous community, Part I of \textit{Stochastic epidemic models with inference.} T. Britton, E. Pardoux. eds, Lecture Notes in Mathematics \textbf{2225}, pp. 1--120, Springer 2019.
	\bibitem{ic} A. B\"ucher and I. Kojadinovic, A note on conditional versus joint unconditional weak convergence in bootstrap consistency results, arXiv:1706.01031v4  [math.ST]  1 Mar 2018.
	\bibitem{yi} A. P. Calder\'on. Intermediate spaces and interpolation, the complex method. Studia Math. 24, 113-190 (1964).
	\bibitem{cta} S. Cl\'emen\c con, V.C. Tran and H. de Arazoza. A stochastic SIR model with contact tracing: large population limits and statistical inference, {\it J. of Biol. Dynamics} {\bf2:4}, pp. 392--414, 2008.
	\bibitem{fer}F. Cobos and J. Peetre. Interpolation of compactness using Aronszajn-Gagliardo functors? \textit{Israel J. Math.} \textbf{68} (1989), 220-240.
	\bibitem{gih} I.  Gihman, A. V. Skorohod. Stochastic differential equations. Springer-Verlag Berlin New york 1972. 
		\bibitem{kc} L. Grafakos. \textit{Classical Fourier analysis} 2nd. ed Springer, 2008. 
	\bibitem{han} B. Hanouzet.
	Espaces de Sobolev avec poids. Application au
	problème de Dirichlet dans un demi espace. Rendiconti del Seminario Matematico della Università di Padova, Tome 46 (1971), p. 227-272. 
	\bibitem{lc} J. Jacod and A.N. Shiryaev. \textit{Limit Theorems for Stochastic Processes}. Springer-Verlag, Berlin, 1987. 
	\bibitem{ik} I. Kaj. A weak interaction epidemic among diffusive particles, in {\it Stochastic Partial Differential Equations}, A. Etheridge ed., London Math. Soc. Lecture Note Series {\bf216}, pp. 189--208, Cambridge Univ. Press, 1995.
	\bibitem{toc} C. Kipnis and C. Landim. Scaling limits of interacting particle systems, volume 320 of  \textit{Grundlehren der Mathematischen Wissenschaften [Fundamental Principles of Mathematical Sciences]}. Springer-verlag, Berlin, 1999.
	\bibitem{kotel} P. Kotelenez, A stopped Doob inequality for stochastic convolution integrals and stochastic evolution equations, \textit{Stochastic analysis and applications}, 2(3), 245-265, 1984.
	\bibitem{kun}H.  Kunita. Stochastic flows and stochastic differential equations. Cambridge University Press, Cambridge 1990.
	\bibitem{lpz} S.P. Lalley, E.A. Perkins and X. Zheng. A phase transition for mesure--valued SIR epidemic processes, {\it The Annals of Probab.} {\bf42}, pp. 237--310, 2014. 
	\bibitem{lof}J. Löfström. Interpolation of weighted spaces of differentiable functions on $\mathbb{R}^{d}.$ \textit{Ann. Mat. Pura. Appl.} \textbf{132} (1982), 189?214.
	\bibitem{loff} J. Löfström.  Interpolation spaces, an introduction, Springer, 1976. 
	\bibitem{meroe}  S. M\'el\'eard. Convergence of the fluctuations for interacting diffusions with jumps associated with Boltzmann equation. \textit{Stochastics and Stochastics Reports,} 63 :195-225, 1998. 
	\bibitem{qc} S. M\'el\'eard et S. Roelly: Sur les convergences \'etroite ou vague de processus \`a valeurs mesures. \textit{Comptes rendus de l'acad\'emie des Sciences de Paris} S\'er. 1, 317:785-788, 1993.
	\bibitem{sm} S. M\'el\'eard.  Mouvement brownien et calcul stochastique, \textit{Techniques de l?ingénieur}, (2003).
	\bibitem{rc} M. M\'etivier. Convergence faible et principe d'invariance pour des martingales \`a valeurs dans des espaces de Sobolev. \textit{Annales de l'IHP,} 20(4) :329-348, 1984.
	\bibitem{mt} M. M\'etivier. (1987). Weak convergence of measure valued processes using Sobolev 
	imbedding techniques, Proceedings Stochastic partial differential equations, Trenio 1985. 
	\textit{Lect}. Notes in Math. \textbf{1236}. 172-183.
	\bibitem{npy} M. N'zi, \'E. Pardoux and T. Yeo. A SIR model on a refining spatial grid I - Law of Large Numbers,
	\textit{Applied Math.\& Optimization}, to appear.
	\bibitem{wc} E. Pardoux. \textit{Probabilistic Models of Population Evolution, Scaling Limits, Genealogies and Interactions,} Springer 2016.
	\bibitem{vc} E. Pardoux. Moderate Deviations and Extinction of an  Epidemic. \textit{Election. J. Probab} \textbf{25}, paper  25, 1-27, 2020.
	\bibitem{uc} S. Roelly-Coppoletta. A criterion of convergence of measure-valued processes: application to measure branching processes \textit{Stochastics}, 17 :43-65, 1985.
	\bibitem{rbbsb} L. Roques, O. Bonnefon, V. Baudrot, S. Soubeyrand, H. Berestycki : A parsimonious model for spatial transmission and heterogeneity in the COVID-19 propagation. \textit{R. Soc. Open Sci.} 7: 201382, 2020. 
	\bibitem{rwvb}W. Rudin.: Real and Complex Analysis, 3rd edn. McGraw-Hill, \textit{New York} (1987)  
	\bibitem{st}   k. Sato. and T. Ueno. Multi-dimensional diffusion and the Markov process on the boundary. \textit{J. Math. Kyoto univ.} 4(3) 529-605.
	
	\bibitem{yc} M.E. Taylor. \textit{Partial differential equations III Nonlinear equations}. Springer 1991. 
	\bibitem{tri} H. Triebel, Interpolation theory, Hunction spaces, Differential operators, \textit{ North-Holland 
	Publishing Comp.}, 1978.
\bibitem{vct} V.C. Tran Mod\`eles particulaires stochastiques pour des probl\`emes d'\'evolution adaptative et pour l'approximation de solutions statistiques, Thesis 2006. 
\end{thebibliography}
\end{document}